\documentclass[11pt,reqno,a4paper]{amsart}
\usepackage{amssymb,amscd,amsmath}
\usepackage{pspicture}
\usepackage{graphicx}

\usepackage{mathptmx}

\usepackage[matrix,arrow,curve,frame]{xy}    

\xymatrixcolsep{1.9pc}                          
\xymatrixrowsep{1.9pc}
\newdir{ >}{{}*!/-5pt/\dir{>}}                  

 \calclayout

\makeatletter

\renewcommand{\subsection}{\@startsection{subsection}{1}{0pt}{-3.25ex plus -1ex minus-.2ex}{1.5ex plus.2ex}{\normalfont\it}}
\renewcommand{\section}{\@startsection{section}{1}{\parindent}{3.5ex plus 1ex minus .2ex}{2.3ex plus.2ex}{\sc}}

\renewcommand{\phi}{\varphi}
\renewcommand{\leq}{\leqslant}
\renewcommand{\geq}{\geqslant}
\renewcommand{\epsilon}{\varepsilon}

\renewcommand{\kappa}{\varkappa}

\DeclareMathOperator{\Spt}{\mathbf{Spt}}
\DeclareMathOperator{\spec}{Spec}

 \DeclareMathOperator{\charr}{char}
\DeclareMathOperator{\hocolim}{hocolim}
\DeclareMathOperator{\sd}{sd} 
 
 \DeclareMathOperator{\mot}{mot}

\DeclareMathOperator{\Hom}{Hom} 
 
 \DeclareMathOperator{\id}{id}

 \DeclareMathOperator{\Mor}{Mor}
 \DeclareMathOperator{\colim}{colim}
\DeclareMathOperator{\Ho}{Ho}

 \DeclareMathOperator{\nis}{nis}

\newcommand{\lra}[1]{\bl{#1}\longrightarrow\relax}
\newcommand{\bl}[1]{\buildrel #1\over}
\newcommand{\cc}{\mathcal}
\newcommand{\bb}{\mathbb}

\newcommand{\op}{{\textrm{\rm op}}}

\newcommand{\wt}{\widetilde}

\newcommand{\A}{\mathbb{A}}
\newcommand{\Fr}{\operatorname{Fr}}

\newtheorem{thm}{Theorem}[section]
\newtheorem{prop}[thm]{Proposition}
\newtheorem*{sublem}{Sublemma}
\newtheorem{cor}[thm]{Corollary}
\newtheorem{lem}[thm]{Lemma}

\newtheorem{conj}{Conjecture}
\newtheorem{rem}[thm]{Remark}

\newtheorem{defs}[thm]{Definition}
\newtheorem{constr}[thm]{Construction}
\newtheorem{notn}[thm]{Notation}
\newtheorem*{framework}{General Framework}

\makeatother

\begin{document}

\footskip30pt


\title{Framed motives of algebraic varieties (after V.~Voevodsky)}
\author{Grigory Garkusha}
\address{Department of Mathematics, Swansea University, Singleton Park, Swansea SA2 8PP, United Kingdom}
\email{g.garkusha@swansea.ac.uk}

\author{Ivan Panin}
\address{St. Petersburg Branch of V. A. Steklov Mathematical Institute,
Fontanka 27, 191023 St. Petersburg, Russia}

\address{St. Petersburg State University, Department of Mathematics and Mechanics, Universitetsky prospekt, 28, 198504,
Peterhof, St. Petersburg, Russia}

\email{paniniv@gmail.com}

\dedicatory{In memory of Vladimir Voevodsky}

\thanks{The second author was supported by the Russian Science Foundation (grant no. 14-21-00035).}

\begin{abstract}
Using the theory of framed correspondences developed by
Voevodsky~\cite{Voe2}, we introduce and study framed motives
of algebraic varieties. They are the major computational tool for constructing
an explicit quasi-fibrant motivic replacement of the suspension $\bb P^1$-spectrum
of any smooth scheme $X\in Sm/k$. Moreover, it is shown that the bispectrum
   $$(M_{fr}(X),M_{fr}(X)(1),M_{fr}(X)(2),\ldots),$$
each term of which is a twisted framed motive of $X$, has motivic
homotopy type of the suspension bispectrum of $X$. Furthermore, an explicit
computation of infinite $\mathbb P^1$-loop motivic spaces is given in terms of spaces with
framed correspondences. We also introduce big framed motives of bispectra and
show that they convert the classical Morel--Voevodsky motivic
stable homotopy theory into an equivalent local theory of framed bispectra.
{\it As a topological application}, it is proved
that the framed motive $M_{fr}(pt)(pt)$ of the point $pt=\spec k$
evaluated at $pt$ is a quasi-fibrant model of the classical sphere
spectrum whenever the base field $k$ is algebraically closed of
characteristic zero.
\end{abstract}



\maketitle

\thispagestyle{empty} \pagestyle{plain}

\newdir{ >}{{}*!/-6pt/@{>}} 

\tableofcontents

\section{Introduction}In~\cite{Voe2} Voevodsky develops the theory of (pre-)sheaves with
framed correspondences. One of its aims was to suggest another
framework for stable motivic homotopy theory more amenable to
explicit calculations (see his Nordfjordeid Lectures~\cite[Remark
2.15]{DLORV} or his unpublished notes~\cite{Voe2}). Recall that
Voevodsky~\cite[Section~2]{Voe2} invented a category of framed
correspondences $Fr_*(k)$ whose objects are those of $Sm/k$ and
morphisms sets $Fr_*(X,Y)=\sqcup_{n\geq 0} Fr_n(X,Y)$ are defined by
means of certain geometric data (see Section~\ref{ohoho} below).
Following Voevodsky~\cite{Voe2}, we put for every $Y\in Sm/k$ (see
Section~\ref{ohoho} below):
\[
Fr(X,Y):=\colim(Fr_0(X,Y)\xrightarrow{\sigma_Y}Fr_1(X,Y)\xrightarrow{\sigma_Y}\dots
\xrightarrow{\sigma_Y} Fr_n(X,Y)\xrightarrow{\sigma_Y}\dots)
\]
and refer to it as the \textit{set of stable framed
correspondences}. Replacing $Y$ by a simplicial object $Y^\bullet$
in $Sm/k$, we get a simplicial set $Fr(X,Y^\bullet)$. Finally, one
can take the diagonal of the pointed bisimplicial set
$Fr(\Delta^\bullet \times X,Y^\bullet)$. {\it Voevodsky
conjectured\/} that if the motivic space $Fr(\Delta^\bullet \times
-,Y^\bullet)$ is locally connected in the Nisnevich topology, then
it is isomorphic in $H_{\mathbb A^1}(k)$ to the motivic space
$\Omega^{\infty}_{\mathbb P^1}\Sigma^{\infty}_{\mathbb
P^1}(Y^\bullet_+)$. This shows that the theory of framed
correspondences can give a machinery for computing/producing motivic
infinite loop spaces.

In this paper {\it we prove the conjecture in affirmative\/} (see
Theorem \ref{infloopspaces}) providing that the base field is
infinite perfect of characteristic different from 2. This paper {\it
shows\/} that the theory of framed correspondences {\it indeed gives
a machinery for computing/producing motivic infinite loop spaces.}
To prove the conjecture, we introduce and study the theory of framed
motives as well as the theory of big framed motives. The main aim of
the  machinery of framed motives is to find an explicit $\bb
A^1$-local replacement of the functor
   $$\Omega^\infty_{\bb G}\Sigma^\infty_{\bb G}\Sigma^\infty_{S^1}: H_{\bb A^1}(k)\to SH_{S^1}(k).$$
To this end, we firstly regard $SH_{S^1}(k)$ as {\it a full
subcategory\/} of the ordinary stable homotopy category
$SH^{\nis}_{S^1}(k)$ consisting of $\mathbb A^1$-local spectra. Then
in Theorem \ref{infloopspectra}(2) we construct an explicit functor
$M_{fr}: H_{\bb A^1}(k)\to SH_{S^1}(k)$ together with a functor
isomorphism
   $$\alpha: \Omega^{\infty}_{\mathbb G}\Sigma^\infty_{\bb G}\Sigma^{\infty}_{S^1} \to M_{fr}.$$
To formulate the main goal of the theory of big framed motives,
consider a full subcategory $SH^{fr}_{\nis}(k)$ of $SH(k)$
consisting of framed bispectra $E$ such that for any $i,j\geq 0$ the
simplicial framed sheaf $E_{i,j}$ is $\bb A^1$-local regarded as an
ordinary motivic space and $\{E^f_{i,j}\}$ regarded as the ordinary
bispectrum is stably motivically fibrant in the stable motivic model
structure. Here ``$f$" refers to a {\it fibrant\/} replacement in
the local model structure on $sShv_\bullet(Sm/k)$.

The the main result of the theory of big framed motives are
Theorems~\ref{a1localspace1} and~\ref{a1localspace2}.
Theorem~\ref{a1localspace1} says that an explicitly constructed
functor
   $$\cc M^b_{fr}: SH(k)\to SH_{\nis}^{fr}(k)$$
converts classical Morel--Voevodsky's stable motivic homotopy theory
$SH(k)$ into an equivalent local homotopy theory of $\bb A^1$-local
framed bispectra from $SH_{\nis}^{fr}(k)$, {\it and thus producing a
new approach to stable motivic homotopy theory}. The main ingredient
of this equivalent local homotopy theory is framed motivic spaces of
the form $C_*Fr(-,Y)$ with $Y\in\Delta^{\op}Fr_0(k)$ a simplicial
scheme as well as their framed motives $M_{fr}(Y)$.
Theorem~\ref{a1localspace2} states that morphisms in $SH(k)$ between
two bispectra $E$ and $E'$ is the set $\pi_0(E^c,\cc
M_{fr}^b(E')^f)$ of ordinary morphisms between bispectra $E^c$ and
$\cc M_{fr}^b(E')_f$ modulo the naive homotopy.

Let us indicate some applications of the theory of framed motives and big framed motives.
The fact that $\alpha$ is a functor isomorphism yields the following statement: if the base field $k$
is perfect infinite with $\charr k\neq2$, then
for any $X\in Sm/k$ and any simplicial object $Y^\bullet$ in $Sm/k$ one has a canonical isomorphism
\begin{equation}\label{KeyAdjunction}
SH(k)(\Sigma^{\infty}_{\bb G} \Sigma^{\infty}_{S^1}X_+,
\Sigma^{\infty}_{\bb G} \Sigma^{\infty}_{S^1}Y^\bullet_+[n])=
SH_{S^1}^{\nis}(k)(\Sigma^{\infty}_{S^1}X_+,M_{fr}(Y^\bullet)[n]),\quad n\geq 0,
\end{equation}
(see Theorem~\ref{applic}).
In particular, the isomorphism $\alpha$ yields that for any
simplicial object $X^{\bullet}$ in $Sm/k$ the projection
$X^{\bullet}\times \mathbb A^1 \to X^{\bullet}$ induces an isomorphism
$M_{fr}(X^{\bullet}\times \mathbb A^1)\cong M_{fr}(X^{\bullet})$.
Furthermore, the functor $M_{fr}$ converts
any elementary distinguished Nisnevich square of $k$-smooth varieties to
the ordinary Mayer--Vietoris exact triangle (see Theorem~\ref{mvproperty}).
Another important property of
framed motives is as follows: given a morphism $\phi: Y^\bullet \to
Z^\bullet$ of simplicial objects in $Sm/k$ such that
the morphism $\Sigma^{\infty}_{\bb G} \Sigma^{\infty}_{S^1}(\phi)$
is an isomorphism in $SH(k)$, then the morphism $M_{fr}(\phi)$ is a
local equivalence.

By definition, the framed motive of a smooth scheme $X\in Sm/k$ over
a field $k$ is a motivic $S^1$-spectrum $M_{fr}(X)$ whose terms are
certain explicit motivic spaces with framed correspondences (see
Definition~\ref{frmotive}). We use framed motives to construct an
explicit quasi-fibrant motivic replacement (i.e. an
$\Omega$-spectrum in positive degrees) of the suspension $\bb
P^1$-spectrum $\Sigma_{\bb P^1}^\infty X_+$ in
Theorem~\ref{Segal_Thm_II} (here $\bb P^1$ is pointed at $\infty$).
Another application is to show in
Theorem~\ref{Motivic_Segal_Thm_III} that an explicitly constructed
bispectrum
   $$M_{fr}^{\bb G}(X)=(M_{fr}(X),M_{fr}(X)(1),M_{fr}(X)(2),\ldots),$$
each term of which is a twisted framed motive of $X$, has motivic
homotopy type of the suspension bispectrum $\Sigma^{\infty}_{\bb G} \Sigma^{\infty}_{S^1}X_+$ of $X$. Moreover,
if we take the stable local fibrant replacement $M_{fr}(X)(n)_f$ of each twisted
framed motive then the bispectrum
   $$M_{fr}^{\bb G}(X)_f=(M_{fr}(X)_f,M_{fr}(X)(1)_f,M_{fr}(X)(2)_f,\ldots)$$
is motivically fibrant by~\cite[Theorem~A]{AGP}. These definitions and results equally hold for simplicial objects
in the category $Sm/k$. We should point out that for any simplicial object
$Y^\bullet$ in $Sm/k$ there is a canonical morphism of bispectra
   $$M_{fr}^{\bb G}(Y^\bullet)\to \cc M^b_{fr}(\Sigma^{\infty}_{\bb G} \Sigma^{\infty}_{S^1}Y^\bullet_+)$$
which is an isomorphism in $SH(k)$. Thus the composite morphism
$\Sigma^{\infty}_{\bb G} \Sigma^{\infty}_{S^1}Y^\bullet_+ \to M_{fr}^{\bb G}(Y^\bullet)
\to \cc M^b_{fr}(\Sigma^{\infty}_{\bb G} \Sigma^{\infty}_{S^1}Y^\bullet_+)$
is an isomorphism in $SH(k)$. This yields equalities (see Theorem~\ref{a1localspace2})
   $$SH(k)(\Sigma^{\infty}_{\bb G}\Sigma^{\infty}_{S^1}X_+, \Sigma^{\infty}_{\bb G} \Sigma^{\infty}_{S^1}Y^\bullet_+)
       =\pi_0(\cc M^b_{fr}(\Sigma^{\infty}_{\bb G} \Sigma^{\infty}_{S^1}Y^\bullet_+)_{0,0}^f(X)).$$
Here ``$f$" refers to a {\it fibrant\/} replacement in the local model structure on $sShv_\bullet(Sm/k)$.

Let us also give some applications of isomorphism~\eqref{KeyAdjunction}.
Since $M_{fr}(Y^\bullet)$ is a sheaf of Segal $S^1$-spectra, it
follows that for any $n<0$ the Nisnevich sheaf $\pi_{n,0}^{\bb
A^1}(\Sigma^{\infty}_{\bb G} \Sigma^{\infty}_{S^1}Y^\bullet_+)$
vanishes. By varying $Y^\bullet$ one gets a much stronger vanishing
property. Namely, for any $n<r$ one has $\pi_{n,r}^{\bb
A^1}(\Sigma^{\infty}_{\bb G} \Sigma^{\infty}_{S^1}Y^\bullet_+)=0$.
We can also compute $\pi_{n,r}^{\bb A^1}(\Sigma^{\infty}_{\bb G}
\Sigma^{\infty}_{S^1}Y^\bullet_+)$ for $n=r$ with $r\leq 0$ as
   $$\pi_{-n,-n}^{\bb A^1}(\Sigma^\infty_{S^1}\Sigma^\infty_{\bb G} Y^\bullet_+)(K)=H_0(\bb
     ZF(\Delta_K^\bullet,Y^\bullet\times\bb G_m^{\wedge n})),\quad n\geq 0.$$
Here $K/k$ is any field extension and $\bb
ZF(\Delta_K^\bullet,Y^\bullet \times\bb G_m^{\wedge n}))$ is an
explicit chain complex of free abelian groups. If $X=\spec(k)$,
$\charr k=0$, then using Neshitov's computation $H_0(\bb
ZF(\Delta_K^\bullet,X\times\bb G_m^{\wedge n}))=K^{MW}_n(K)$~\cite{Nesh}, we
recover the celebrated theorem of Morel~\cite{Mor1} for Milnor--Witt
$K$-theory for fields of characteristic zero.

We also give an explicit computation of the suspension functor (see Theorem~\ref{infloopspaces})
   $$\Sigma^\infty_{\bb P^1}:H_{\bb A^1}(k)\to SH(k).$$
It is isomorphic to an explicitly constructed functor
   $$M_{\bb P^{\wedge 1}}:H_{\bb A^1}(k)\to SH(k)$$
that takes a motivic space to a spectrum consisting of spaces with framed correspondences.
As an application, an explicit computation of the space $\Omega^{\infty}_{\bb P^1}\Sigma^{\infty}_{\bb P^1}(\cc X)$
is given for any motivic space $\cc X$ (see Theorem~\ref{infloopspaces}). Thus the machinery
of framed motives leads to explicit computations of infinite $\bb P^1$-loop spaces.

{\it As a topological application}, using the machinery of framed motives together with
a theorem of Levine~\cite{Levine}, one
shows in Theorem~\ref{corres} that the framed motive of the point $M_{fr}(pt)(pt)$,
$pt=\spec k$, evaluated at $pt$ is a quasi-fibrant model (i.e. an
$\Omega$-spectrum in positive degrees) of the classical sphere
spectrum if the base field $k$ is algebraically closed of
characteristic zero. In particular, for any $n\geq 1$ the simplicial set
$C_*Fr(pt,pt\otimes S^n)$ has the homotopy type of
$\Omega^\infty_{S^1}\Sigma^\infty_{S^1}S^n$.

Finally, we show that every $\bb P^1$-spectrum $E$ is actually isomorphic in $SH(k)$ to
a framed spectrum in a canonical way, functorially in $E$. Furthermore, there is an equivalence of categories
   $$SH(k)\lra{\cong} SH^{fr}(k),$$
where $SH^{fr}(k)$ is a full subcategory of framed $\bb P^1$-spectra (see
Theorem~\ref{shfrsp}).

The main result of~\cite{GP4} says that for any $\bb A^1$-invariant
quasi-stable radditive framed presheaf of Abelian groups $\cc F$,
the associated Nisnevich sheaf $\cc F_{\nis}$ is strictly $\bb
A^1$-invariant whenever the base field $k$ is infinite perfect of
characteristic different from 2. The assumption on the
characteristic is needed to prove the ``Surjective \'{e}tale
excision theorem" in~\cite{GP4}. {\it It is for this reason that the
main results for framed spaces, framed motives or framed
spectra/bispectra are formulated and proven in this paper under the
assumption that the base field $k$ is infinite perfect of
characteristic different from 2\/} (if $\charr k=2$ we formulate and
prove the same results for spaces, framed motives or framed
spectra/bispectra ``with 1/2-coefficients" only). However, as soon
as the mentioned result of~\cite{GP4} about strict $\bb
A^1$-invariance of $\cc F_{\nis}$ is extended to (infinite) perfect
fields of any characteristic, then the main results of this paper
will automatically be extended to such fields as well.

{\it The authors dedicate the paper to the memory of Vladimir Voevodsky
who had kindly given his unpublished notes on framed
correspondences~\cite{Voe2} to us in 2010.
}

The authors are very grateful to Andrei Suslin for many helpful
discussions. This paper was partly written during the visit of the
authors in summer 2014 to the University of Duisburg--Essen (Marc
Levine's Arbeitsgruppe). The authors also made major revisions of
the original text during their visit to the Mittag-Leffler Institute
in Stockholm in 2017. They would like to thank the University of
Duisburg--Essen and the Mittag-Leffler Institute for the kind
hospitality and support.

Throughout the paper we denote by $Sm/k$ the category of smooth
separated schemes of finite type over the base field $k$.

\section{Voevodsky's framed correspondences}\label{ohoho}

In this section we collect basic facts for framed correspondences
and framed functors in the sense of Voevodsky~\cite{Voe2}. We start
with preparations.

Let $S$ be a scheme and $Z$ be a closed subscheme. Recall that an
{\it \'{e}tale neighborhood of $Z$ in $S$\/} is a triple
$(W',\pi':W'\to S,s': Z\to W')$ satisfying the conditions:

\indent (i) $\pi'$ is an \'{e}tale morphism;

\indent (ii) $\pi'\circ s'$ coincides with the inclusion
$Z\hookrightarrow S$ (thus $s'$ is a closed embedding);

\indent (iii) $(\pi')^{-1}(Z)=s'(Z)$

A morphism between two \'{e}tale neighborhoods
$(W',\pi',s')\to(W'',\pi'',s'')$ of $Z$ in $S$
is a morphism $\rho:W'\to W''$ such that $\pi''\circ\rho=\pi'$ and
$\rho\circ s'=s''$. Note that such $\rho$ is automatically \'etale
by~\cite[VI.4.7]{LNM146}.

\begin{defs}
[Voevodsky~\cite{Voe2}]{\rm




For $k$-smooth schemes $X, Y$ and $n\geq 0$ an {\it explicit framed
correspondence\/}  $\Phi$ of level $n$ consists of the following
data:

\begin{enumerate}
\item a closed subset $Z$ in $\mathbb A^n_X$ which is finite over $X$;
\item an etale neighborhood $p:U\to\mathbb A^n_X$ of $Z$ in $\mathbb A^n_X$;
\item a collection of regular functions $\phi=(\phi_1,\ldots,\phi_n)$ on $U$
such that $\cap_{i=1}^n \{\phi_i=0\}=Z$;
\item a morphism $g:U\to Y$.
\end{enumerate}
The subset $Z$ will be referred to as the {\it support\/} of the
correspondence. We shall also write triples $\Phi=(U,\phi,g)$ or
quadruples $\Phi=(Z,U,\phi,g)$ to denote explicit framed
correspondences.

Two explicit framed correspondences $\Phi$ and $\Phi'$ of level $n$
are said to be {\it equivalent\/} if they have the same support and
there exists an open neighborhood $V$ of $Z$ in $U \times_{\bb
A^n_X} U'$ such that on $V$, the morphism $g\circ pr$ agrees with
$g' \circ pr'$ and $\phi \circ pr$ agrees with $\phi' \circ pr'$. A
{\it framed correspondence of level $n$} is an equivalence class of
explicit framed correspondences of level $n$.

}\end{defs}

Let $Fr_n(X,Y)$ denote the set of framed correspondences from
$X$ to $Y$. We consider it as a pointed set with the basepoint
being the class $0_n$ of the explicit correspondence with
$U=\emptyset$.

As an example, the sets $Fr_0(X,Y)$ coincide
with the set of pointed morphisms $X_+\to Y_+$. In particular, for a
connected scheme $X$ one has\label{fr0}
   $$Fr_0(X,Y)=\text{Hom}_{Sm/k}(X,Y)\sqcup\{0_0\}.$$

If $f:X'\to X$ is a morphism of schemes and $\Phi=(U,\phi,g)$ an
explicit correspondence from $X$ to $Y$ then
   $$f^*(\Phi):=(U'=U\times_X X',\phi\circ pr,g\circ pr)$$
is an explicit correspondence from $X'$ to $Y$.

\begin{rem}\label{short_Notation}
{\rm
Let $\Phi=(Z,\mathbb A^n_X \xleftarrow{p} U,\phi: U \to \mathbb
A^n_k,g: U\to Y) \in Fr_n(X,Y)$ be an {\it explicit framed
correspondence  of level n}. It can more precisely be written in the
form
\[
((\alpha_1,\alpha_2,\dots,\alpha_n),f,Z,U,(\varphi_1,\varphi_2,\dots,\varphi_n),g)\in Fr_n(X,Y)
\]
where
\begin{itemize}
\item[$\diamond$]
$Z\subset \A^n_X$ is a closed subset finite over $X$,
\item[$\diamond$]
an etale neighborhood
$(\alpha_1,\alpha_2,\dots,\alpha_n),f)=p: U \to \mathbb A^n_k \times X$ of $Z$,
\item[$\diamond$]
a collection of regular functions $\phi=(\phi_1,\ldots,\phi_n)$ on
$U$ such that $\cap_{i=1}^n \{\phi_i=0\}=Z$;
\item[$\diamond$] a morphism $g:U\to Y$.
\end{itemize}
We shall usually drop $((\alpha_1,\alpha_2,\dots,\alpha_n),f)$ from
notation and just write
   \[(Z,U,(\varphi_1,\varphi_2,\dots,\varphi_n),g)=((\alpha_1,\alpha_2,\dots,\alpha_n),
    f,Z,U,(\varphi_1,\varphi_2,\dots,\varphi_n),g).\]
}\end{rem}

The following definition is to describe compositions of framed
correspondences.

\begin{defs}\label{cat_Fr_+}{\rm
Let $X,Y$ and $S$ be $k$-smooth schemes and let
\begin{gather*}
a=((\alpha_1,\alpha_2,\dots,\alpha_n),f,Z,U,(\varphi_1,\varphi_2,\dots,\varphi_n),g)
\end{gather*}
be an explicit correspondence of level $n$ from $X$
to $Y$ and let
\begin{gather*}
b=((\beta_1,\beta_2,\dots,\beta_m),f',Z',U',(\psi_1,\psi_2,\dots,\psi_m),g')\in Fr_m(Y,S)
\end{gather*}
be an explicit correspondence of level $m$ from $Y$ to $S$. We
define their composition as an explicit correspondence of level
$n+m$ from $X$ to $S$ by
\[
((\alpha_1,\alpha_2,\dots,\alpha_n,\beta_1,\beta_2,\dots,\beta_m),f,Z\times_Y Z',U\times_Y U',(\varphi_1,\varphi_2,\dots,\varphi_n,\psi_1,\psi_2,\dots,\psi_m),g').
\]
Clearly, the composition of explicit correspondences respects the
equivalence relation on them and defines associative maps
   \begin{equation*}\label{compos}
    Fr_n(X,Y)\times Fr_m(Y,S)\to Fr_{n+m}(X,S).
   \end{equation*}
}\end{defs}

Given $X, Y\in Sm/k$, denote by $Fr_+(X,Y)$ the set $\bigvee_n
Fr_n(X,Y)$. The composition of framed correspondences defined above
gives a category $Fr_+(k)$. Its objects are those of $Sm/k$ and the
morphisms are given by the sets $Fr_+(X,Y)$, $X, Y\in Sm/k$. Since
the naive morphisms of schemes can be identified with certain framed
correspondences of level zero, we get a canonical functor
   $$Sm/k\to Fr_+(k).$$
{\it The category $Fr_+(k)$ has the zero object.
It is the empty scheme.
}
One can easily see that for a framed correspondence $\Phi:X\to Y$
and a morphism $f:X'\to X$, one has $f^*(\Phi)=\Phi\circ f$.

\begin{defs}\label{d:boxpairing}{\rm
Let $X,Y,S$ and $T$ be smooth schemes. There is an \textit{external
product}
\[
Fr_n(X,Y)\times Fr_m(S,T) \xrightarrow{-\boxtimes -}
Fr_{n+m}(X\times S, Y\times T)
\]
given by
\begin{gather*}
((\alpha_1,\alpha_2,\dots,\alpha_n),f,Z,U,(\varphi_1,\varphi_2,\dots,\varphi_n),g)\boxtimes
((\beta_1,\beta_2,\dots,\beta_m),f',Z',U',(\psi_1,\psi_2,\dots,\psi_m),g')=
\\
((\alpha_1,\alpha_2,\dots,\alpha_n,\beta_1,\beta_2,\dots,\beta_m),f\times
f',Z\times Z',U\times
U',(\varphi_1,\varphi_2,\dots,\varphi_n,\psi_1,\psi_2,\dots,\psi_m),g\times
g').
\end{gather*}

For the constant morphism $c\colon \A^1\to pt$, we set (following
Voevodsky~\cite{Voe2})
\[
\Sigma=-\boxtimes (t,c,\{0\},\A^1,t,c)\colon Fr_n(X,Y)\to Fr_{n+1}(X,Y)
\]
and refer to it as the \textit{suspension}.

Also, following Voevodsky~\cite{Voe2}, one puts
\[
Fr(X,Y)=\colim(Fr_0(X,Y)\xrightarrow{\Sigma}Fr_1(X,Y)\xrightarrow{\Sigma}\dots
\xrightarrow{\Sigma} Fr_n(X,Y)\xrightarrow{\Sigma}\dots)
\]
and refer to it as the \textit{set stable framed correspondences}.
The above external product induces external products
\begin{gather*}
Fr_n(X,Y)\times Fr(S,T) \xrightarrow{-\boxtimes -} Fr(X\times S, Y\times T),\\
Fr(X,Y)\times Fr_0(S,T) \xrightarrow{-\boxtimes -} Fr(X\times S, Y\times T).
\end{gather*}

}\end{defs}

\begin{defs}\label{def:FrY/Y-S}{\rm
(I) Let $Y$ be a $k$-smooth scheme and $S\subset Y$ be a closed
subset and let $U\in Sm/k$. An {\it explicit framed correspondence
of level $m\geq 0$ from $U$ to $Y/(Y-S)$} consists of the tuples:
   $$(Z,W,\phi_1,\ldots,\phi_{m};g:W\to Y),$$
where $Z$ is a closed subset of $U\times\bb A^m$, finite over $U$,
$W$ is an \'{e}tale neighborhood of $Z$ in $U\times\bb A^m$,
$\phi_1,\ldots,\phi_{m}$ are regular functions on $W$, $g$ is a
regular map such that $Z=Z(\phi_1,\ldots,\phi_{m})\cap g^{-1}(S)$.
The set $Z$ is called the {\it support\/} of the explicit framed
correspondence. We shall also write quadruples $\Phi = (Z,W,\phi;g)$
to denote explicit framed correspondences.

(II) Two explicit framed correspondences $(Z,W,\phi;g)$ and
$(Z',W',\phi';g')$ of level $m$ are said to be {\it equivalent\/} if
$Z=Z'$ and there exists an \'{e}tale neighborhood $W''$ of $Z$ in
$W\times_{\bb A^m_U}W'$ such that $\phi\circ pr$ agrees with
$\phi'\circ pr'$ and the morphism $g\circ pr$ agrees with $g'\circ
pr'$ on $W''$.

(III) A {\it framed correspondence of level $m$ from $U$ to
$Y/(Y-S)$\/} is the equivalence class of an explicit framed
correspondence of level $m$ from $U$ to $Y/(Y-S)$. We write
$Fr_m(U,Y/(Y-S))$ to denote the set of framed correspondences of
level $m$ from $U$ to $Y/(Y-S)$. We regard it as a pointed set whose
distinguished point is the class $0_{Y/(Y-S),m}$ of the explicit
correspondence $(Z,W,\phi;g)$ with $W=\emptyset$.

(IV) If $S=Y$ then the pointed set $Fr_m(U,Y/(Y-S))$
coincides with the pointed set
$Fr_m(U,Y)$ of framed correspondences of
level $m$ from $U$ to $Y$.
}\end{defs}

\begin{defs}{\rm
A {\it framed presheaf\/} $\cc F$ on $Sm/k$ is a contravariant
functor from $Fr_+(k)$ to the category of sets. A {\it framed
functor\/} $\cc F$ on $Sm/k$ is a contravariant functor from
$Fr_+(k)$ to the category of pointed sets such that $\cc
F(\emptyset)=pt$ and $\cc F(X\sqcup Y)=\cc F(X)\times\cc F(Y)$.

A framed Nisnevich sheaf on $Sm/k$ is a framed presheaf $\cc F$ such that
its restriction to $Sm/k$ is a Nisnevich sheaf.

}\end{defs}

Note that the representable presheaves on $Fr_+(k)$ are not framed functors.

\begin{constr}\label{c:framed_presh_Y_Y-S}{\rm
We set $Fr_+(-,Y/(Y-S)):=\bigvee_{m\geq 0} Fr_m(-,Y/(Y-S))$ and define the structure of a framed presheaf on it as follows.
Let $X,Y$ and $S$ be $k$-smooth schemes and let
\begin{equation*}
\Psi=(Z',\A^k\times V\xleftarrow{(\alpha,\pi')}
W',\psi_1,\psi_2,\dots,\psi_k;g:W'\to U)\in Fr_k(V,U)
\end{equation*}
be an explicit correspondence of level $k$ from $V$ to $U$. Suppose
\begin{equation*}
\Phi=(Z,\A^m\times U\xleftarrow{(\beta,\pi)}
W,\varphi_1,\varphi_2,\dots,\varphi_m;g':W\to Y)\in Fr_m(U,Y/(Y-S))
\end{equation*}
is an explicit correspondence of level $m$ from $U$ to $Y/(Y-S)$. We
define $\Psi^{*}(\Phi)$ as an explicit correspondence of level $k+m$
from $V$ to $Y/(Y-S)$ as\footnotesize
$$(Z\times_U Z',\A^{k+m}\times V\xleftarrow{(\alpha,\beta,\pi')} W'\times_U W,\psi_1,\psi_2,\dots,\psi_k,\varphi_1,\varphi_2,\dots,\varphi_m,,g'\circ pr_W)\in
Fr_{k+m}(V,Y/(Y-S)).$$
\normalsize Clearly, the pullback operation $(\Psi,\Phi)\mapsto \Psi^{*}(\Phi)$
of explicit correspondences respects the equivalence relation on
them. We get a pairing
   \begin{equation}\label{eq:compos}
    Fr_k(V,U)\times Fr_m(U,Y/(Y-S))\to Fr_{k+m}(V,Y/(Y-S))
   \end{equation}
making $Fr_+(-,Y/(Y-S))$ a $Fr_+(k)$-presheaf.
}
\end{constr}

Denote by
\begin{equation}\label{eq:Sigma}
\sigma_{Y/(Y-S)}: Fr_m(U,Y/(Y-S))\to Fr_{m+1}(U,Y/(Y-S))
\end{equation}
a map, which takes $\Phi = (Z,W,\phi;g)$ to $(Z\times \{0\},W\times \A^1,\phi\circ pr_W,pr_{\A^1};g)$.
\smallskip

Following Voevodsky~\cite{Voe2} we give the following

\begin{defs}\label{def:st_fr_corr}{\rm
We shall refer to the set
\begin{multline*}
Fr(U,Y/(Y-S)):= \\
=\colim(Fr_0(U,Y/(Y-S))\xrightarrow{\sigma_{Y/(Y-S)}} Fr_1(U,Y/(Y-S))
\xrightarrow{\sigma_{Y/(Y-S)}} Fr_2(U,Y/(Y-S)) \cdots)
\end{multline*}
as the \textit{set of stable framed correspondences from $U$ to $Y/(Y-S)$}.
}\end{defs}

\begin{rem}\label{def:Fr_Y_Y-S_as_framed_presh}{\rm
It is straightforward to check that
for any framed correspondences $\Psi\in Fr_n(U',U)$ and $\Phi\in Fr_m(U,Y/(Y-S)$
one has $\sigma_{Y/(Y-S)}(\Psi^{*}(\Phi))=\Psi^*(\sigma_{Y/(Y-S)}(\Phi))$.
This shows that
the assignment
$U\mapsto Fr(U,Y/(Y-S))$ from Definition
\ref{def:st_fr_corr}
is {\it a framed presheaf}.
}
\end{rem}

For a scheme $X$ we let $Et/X$ denote the category of schemes separated and \'{e}tale over $X$.

\begin{thm}[Voevodsky \cite{Voe2}] \label{Fr_is_sheaf}
Let $X$ be a $k$-smooth scheme.
Then for any scheme Y the functor $U\mapsto Fr_n(U,Y )$ from $Et/X$ to $Sets_\bullet$ is a sheaf in the etale topology.
\end{thm}

The proof of the latter theorem given in \cite{Voe2} yields the following

\begin{cor} \label{Fr_is_sheaf_2}
Given $Y\in Sm/k$ and any closed subset $S$ in $Y$, the presheaf $Fr_n(-,Y/(Y-S))$ on $Sm/k$ is a pointed Nisnevich sheaf.
Also, the framed presheaf $Fr(-,Y/(Y-S))$ is a framed Nisnevich sheaf.
\end{cor}

\section{The Voevodsky Lemma}\label{voevlem}

In this section we discuss the Voevodsky lemma computing framed correspondences in terms of
morphisms of associated Nisnevich sheaves. It is crucial in our analysis.
Corollary~\ref{cor:sheaf_and_geometry} and a sketch of its proof was communicated to us by A.~Suslin.
It very much helped the authors in understanding Voevodsky's notes \cite{Voe2}.

\begin{constr}\label{c:Fr_and_hom_2}{\rm
Given an explicit framed correspondence
$\alpha=(Z,W,g: W\to Y)$ from $X$ to $Y/(Y-S)$ of level zero, consider an elementary distinguished square of the form
$$\xymatrix{W-Z\ar[r]^{in}\ar[d]_{(\rho)|_{W-Z}}& W\ar[d]^{\rho}\\
             X-Z    \ar[r]_(.55){in}& X}$$
where $\rho:W\to X$ is an \'{e}tale neighborhood of $Z$. Let
$q: Y \to F:=Y/(Y-S)$ be the canonical morphism of Nisnevich sheaves. Take a morphism of sheaves
$q\circ g: W\to F$ and a morphism of sheaves $c: X-Z \to F$ sending
$X-Z$ to the distinguished point of $F$.
These two morphisms agree on $W-Z$. Thus there is a unique morphism of Nisnevich sheaves
   $$s_{(Z,W,g)}: X \to F$$
such that $s_{(Z,W,g)}\circ in=c$ and
$s_{(Z,W,g)}\circ \rho=q\circ g$. Clearly, the sheaf morphism
$s_{(Z,W,g)}$ depends only on the equivalence class of $(Z,W,g)$ in $Fr_0(U,Y/(Y-S))$.
The assignment $(Z,W,g)\mapsto s_{(Z,W,g)}$ defines a map of pointed sets
   $$a_{X,Y/(Y-S)}: Fr_0(X,Y/(Y-S)) \to \Mor_{Shv}(X,Y/(Y-S)).$$
The map $a_{X,Y/(Y-S)}$ is natural in $X$ with respect to morphisms of smooth varieties.
Hence $a_{Y/(Y-S)}: Fr_0(-,Y/(Y-S)) \to \Mor_{Shv}(-,Y/(Y-S))$
is a morphism of presheaves on the category $Sm/k$. Using Corollary~\ref{Fr_is_sheaf_2}, the morphism}
   $$a_{Y/(Y-S)}: Fr_0(-,Y/(Y-S)) \to \Mor_{Shv}(-,Y/(Y-S))$$
{\it is a morphism of Nisnevich sheaves on $Sm/k$}.
\end{constr}

\begin{lem}[Voevodsky's Lemma]\label{l:sheaf_and_geometry}
Let $Y$ be a $k$-smooth scheme and $S\subset Y$ be a closed subset.
The morphism of pointed Nisnevich sheaves
\begin{equation}
a_{Y/(Y-S)}: Fr_0(-,Y/(Y-S)) \to \Mor_{Shv}(-,Y/(Y-S))
\end{equation}
is an isomorphism.
\end{lem}

\begin{proof}
Since $a_{Y/(Y-S)}$ is a morphism of Nisnevich sheaves, it suffices to check that for any
essentially $k$-smooth local Henzelian $U$ the map $a_{X,Y/(Y-S)}$ is a bijection.
Let $U$ be local essentially smooth Henselian with the closed point $u\in U$.
Since $U$ is local Henzelian the following holds: for any non-empty closed subset
$Z$ in $U$ the henzelization $U^h_Z$ of $U$ at $Z$ coincides with $U$ itself.
This shows that
   $$Fr_0(U,Y/(Y-S))-0_0=\{(Z,U,f: U\to Y)\mid Z=f^{-1}(S), Z\neq \emptyset \}=\{f: U\to Y|f(u)\in S \}.$$
Here $0_0$ is the distinguished point in $Fr_0(U,Y/(Y-S))$.
The map $a_{U,Y/(Y-S)}$ takes a triple $(Z,U,f)$ to the morphism $q\circ f: U\to Y/(Y-S)$,
where $q: Y\to Y/(Y-S)$ is the quotient map.

Let $Y_u(U)\subset Y(U)$ be the subset of $U$-points of $Y$ consisting of $g\in Y(U)$ with $g(u)\in S$.
Then the map $Y(U)=\Mor_{Shv}(U,Y)\to \Mor_{Shv_\bullet}(U,Y/(Y-S))$ taking $f\in Y(U)$ to $q\circ f$
identifies $Y_u(U)$ with the subset $\Mor_{Shv_\bullet}(U,Y/(Y-S))-*$. In fact,
\begin{equation}\label{eq:U_to_Y/Y-S}
\Mor_{Shv}(U,Y/(Y-S))=
(Y_u(U)\sqcup (Y-S)(U))/(Y-S)(U))=
Y_u(U)\sqcup *,
\end{equation}
where $*$ is a singleton. Hence the map
   $$a_{U,Y/(Y-S)}|_{Fr_0(U,Y/(Y-S))-\{0_0\}}: Fr_0(U,Y/(Y-S))-\{0_0\} \to \Mor_{Shv_\bullet}(U,Y/(Y-S))\setminus *$$
is a bijection. Thus the map $a_{U,Y/(Y-S)}$ is a bijection, too.
\end{proof}

\begin{cor}\label{cor:sheaf_and_geometry}
Let $Y$ be a $k$-smooth scheme and $S\subset Y$ be a closed subset. Let $X$ be a $k$-smooth variety and $B\subset X$ its closed subset.
Suppose $Fr_0(X/B,Y/(Y-S))$ is the subset of $Fr_0(X,Y/(Y-S))$ consisting of framed correspondences
$(Z,W,g)$ with $Z\cap B=\emptyset$. Then the map of pointed sets
  \begin{equation*}
   a_{X/B,Y/(Y-S)}: Fr_0(X/B,Y/(Y-S)) \to \Mor_{Shv_\bullet}(X/B,Y/(Y-S))
  \end{equation*}
is a bijection.
\end{cor}

\begin{proof}
Consider a commutative diagram
$$\xymatrix{Fr_0(X/B,Y/(Y-S)) \ar[rr]^{a_{X/B,Y/(Y-S)}}\ar[d]_{in}&& \Mor_{Shv_\bullet}(X/B,Y/(Y-S)) \ar[d]^{r^*}\\
             Fr_0(X,Y/(Y-S))    \ar[rr]_(.55){a_{X,Y/(Y-S)}}&& \Mor_{Shv}(X,Y/(Y-S))},$$
where $r^*$ is induced by the quotient map $r: X\to X/(X-S)$. Since $r$ is an epimorphism the map $r^*$ is
injective. The map $in$ is an inclusion by the definition of $Fr_0(X/B,Y/(Y-S))$.
By Lemma~\ref{l:sheaf_and_geometry} the map $a_{X,Y/(Y-S)}$ is bijective. Thus the map
$a_{X/B,Y/(Y-S)}$ is injective. It remains to check its surjectivity.

Let $g: X\to Y/(Y-S)$ be a Nisnevich sheaf morphism. It is in $\Mor_{Shv_\bullet}(X/B,Y/(Y-S))$
if and only if $g(B)$ is the distinguished point $*$ in $Y/(Y-S)$.

Let $g: X\to Y/(Y-S)$ be a Nisnevich sheaf morphism such that $g(B)=*$.
We claim that $g$ is in the image of $a_{X/B,Y/(Y-S)}$.
By Lemma~\ref{l:sheaf_and_geometry} there is an explicit framed correspondence
$(Z,W,\tilde g:W\to Y)$ from $X$ to $Y/(Y-S)$ such that
$g=a_{X,Y/(Y-S)}((Z,W,\tilde g: W\to Y))$.
The latter equality means that the morphism $g$ is unique
such that $g(X-Z)=*$ and $q\circ \tilde g=g\circ \rho$.
If $g(B)=*$, then $B\cap Z=\emptyset$. Indeed, if $b$ is a closed point of the closed subset
$B\cap Z$ then $*\not=\tilde g(b)\in S$. On the other hand, $\tilde g(b)=g(b)=*$. We see that
$B\cap Z=\emptyset$, $(Z,W,\tilde g:W\to Y)$
is in $Fr_0(X/B,Y/(Y-S))$, and $g=a_{X/B,Y/(Y-S)}((Z,W,\tilde g:W\to Y))$ as claimed.
\end{proof}

\begin{rem}\label{rem:Fr_0_and_Fr_n}{\rm
Let $n>0$ be an integer. Let $B_n\subset (\bb P^1)^n$ be a closed subset which is the union
of all subsets of the form $\bb P^1\times ... \times \{\infty\} \times ... \times \bb P^1$.
Set $B_0=\{\infty\}$. For any $X,Y\in Sm/k$ and any $n\geq 0$ the inclusion
$$Fr_n(X,Y)\subset Fr_0(X\times (\bb P^1)^n/X\times B_n,Y\times \bb A^n/Y\times (\bb A^n-\{0\}))$$
{\it is an equality}.
It suffices to check that any element $(Z,W,f:W\to Y\times \bb A^n)$ from
$Fr_0(X\times (\bb P^1)^n/X\times B_n,Y\times \bb A^n/Y\times (\bb A^n-\{0\}))$
is contained in $Fr_n(X,Y)$. Since $Z\cap (X\times B_n)=\emptyset$, it follows that
$Z\subset X\times \bb A^n$. Since $Z$ is closed in $X\times (\bb P^1)^n$,
then $Z$ is projective over $X$. Since $Z$ is also affine over $X$, it is finite over $X$.
Giving a morphism $f:W\to Y\times \bb A^n$ is the same as
giving $n$ functions $\phi_1,...,\phi_n$ and a morphism $g: W\to Y$.
The condition $Z=f^{-1}(Y\times \{0\})$ is equivalent to that of
$Z=\{\phi_1=...=\phi_n=0\}$. The desired equality is checked.
}\end{rem}

The preceding remark and Corollary~\ref{cor:sheaf_and_geometry} imply the following

\begin{prop}[Voevodsky]\label{voevych}
For any $X,Y\in Sm/k$ and any $n\geq 0$ the map
\begin{multline*}\label{vformula}
    a_{n,X,Y}=a_{X\times (\bb P^1)^n/X\times B_n,Y\times \bb A^n/Y\times (\bb A^n-\{0\})}:Fr_n(X,Y) \to \\
    \to \Hom_{Shv^{\nis}_{\bullet}(Sm/k)}(X_+\wedge(\bb P^1,\infty)^{\wedge n},Y_+\wedge(\bb A^1/(\bb A^1-0))^n)=\\
    =\Hom_{Shv^{\nis}_{\bullet}(Sm/k)}(X_+\wedge(\bb P^1,\infty)^{\wedge n},Y_+\wedge T^n).
   \end{multline*}
is a bijection. In what follows we shall write $\bb P^{\wedge n}$ for $(\bb P^1,\infty)^{\wedge n}$ and
$\Hom(X_+\wedge \bb P^{\wedge n},Y_+\wedge T^n)$ instead of $\Hom_{Shv^{\nis}_{\bullet}(Sm/k)}(X_+\wedge \bb P^{\wedge n},Y_+\wedge T^n)$.
We shall also write $\cc Fr_n(X,Y)$ to denote $\Hom(X_+\wedge \bb P^{\wedge n},Y_+\wedge T^n)$.
\end{prop}

Consider two categories $Fr_+(k)$ and $\cc Fr_+(k)$,
where the objects in both categories are those of $Sm/k$.
The category $Fr_+(k)$ is defined in \ref{cat_Fr_+}. The morphisms between $X$ and $Y$ in
$\cc Fr_+(k)$ are defined as $\bigvee_{n\geq 0} \cc Fr_n(X,Y)$.
The composition is defined as follows. Given two morphisms
$\alpha: X_+\wedge \bb P^{\wedge m} \to Y_+\wedge T^m$ and
$\beta: Y_+\wedge \bb P^{\wedge n} \to S_+\wedge T^n$,
define a morphism $\beta \circ \alpha \in \cc Fr_{m+n}(X,S)$
as the composite
   $$X_+\wedge \bb P^{\wedge m}\wedge \bb P^{\wedge n} \xrightarrow{\alpha\wedge\id} Y_+\wedge T^m\wedge \bb P^{\wedge n}\cong T^m\wedge Y_+\wedge \bb P^{\wedge n}
       \xrightarrow{\id\wedge \beta} T^m\wedge Y_+\wedge T^n\cong Y_+\wedge T^m\wedge T^n.$$
It is straightforward to check commutativity of the diagram
$$\xymatrix{Fr_m(X,Y)\times Fr_m(Y,S)\ar[r]^(.6){\circ}\ar[d]_{a\times a}& Fr_{m+n}(X,S) \ar[d]^{a}\\
            \cc Fr_m(X,Y)\times \cc Fr_n(Y,S)    \ar[r]^(.6){\circ}&
            \cc Fr_{m+n}(X,S).  }$$

These observations imply the following

\begin{cor}\label{Fr_+_and_Fr_hom_+}
The functor
   $$a: Fr_+(k)\to \cc Fr_+(k)$$
is an isomorphism of categories.
\end{cor}

It is also worth to make the following

\begin{rem}\label{rem:Fr_0_and_Fr_n_2}{\rm
One has that
   $$Fr_n(X,Y/(Y-S))\subset Fr_0(X\times (\bb P^1)^n/X\times B_n,Y\times \bb A^n/(Y\times \bb A^n - S\times \{0\}))$$
{\it is an equality}. It suffices to check that any element $(Z,W,f:W\to Y\times \bb A^n)$ from
$Fr_0(X\times (\bb P^1)^n/X\times B_n,Y\times \bb A^n/(Y\times \bb A^n - S\times \{0\}))$
is contained in $Fr_n(X,Y(Y-S))$. Since $Z\cap (X\times B_n)=\emptyset$, then
$Z\subset X\times \bb A^n$. Since $Z$ is closed in $X\times (\bb P^1)^n$,
then $Z$ is projective over $X$. Since $Z$ is affine over $X$, it is also finite over $X$.
Giving a morphism $f:W\to Y\times \bb A^n=\bb A^n\times Y$ is the same as
giving $n$ functions $\phi_1,...,\phi_n$ and a morphism $g: W\to Y$.
The condition $Z=f^{-1}(S\times \{0\})$ is equivalent to that of
$Z=\{\phi_1=...=\phi_n=0\} \cap g^{-1}(S)$. The equality is checked.
}
\end{rem}

The previous remark and Corollary~\ref{cor:sheaf_and_geometry} imply the following

\begin{prop}[Voevodsky]\label{voevych_2}
For any $X,Y\in Sm/k$ and any $n\geq 0$, the map
$$a_{n,X,Y/(Y-S)}:Fr_n(X,Y/(Y-S)) \to
    \Hom_{Shv_\bullet}(X_+\wedge \bb P^{\wedge n},Y\times \bb A^n/(Y\times \bb A^n-S\times \{0\}))=$$
$$=\Hom_{Shv_\bullet}(X_+\wedge \bb P^{\wedge n},Y/(Y-S)\wedge T^n).$$
is a bijection. For brevity, we shall write
$\cc Fr_n(X,Y/(Y-S))$ to denote $\Hom_{Shv_\bullet}(X_+\wedge \bb P^{\wedge n},Y/(Y-S)\wedge T^n)$.
\end{prop}

\begin{rem}\label{lem: Fr_hom F)}{\rm
For any pointed Nisnevich sheaf $\cc F$, the presheaf
$\cc Fr_+(\mathcal F):=\bigvee_{n\geq 0} \underline{\Hom}(\bb P^{\wedge n},\mathcal F\wedge T^n)$
is a framed presheaf. Indeed,
define for any $U,X\in Sm/k$ and any $m,n$ a map of pointed sets
$$\text{Hom}(U_+\wedge \bb P^{\wedge m},X_+\wedge T^m) \times \text{Hom}(X_+\wedge \bb P^{\wedge n},\mathcal F\wedge T^n) \to
\text{Hom}(U_+\wedge \bb P^{\wedge (m+n)},\mathcal F\wedge T^{m+n}).$$
If $\alpha: U_+\wedge \bb P^{\wedge m} \to X_+\wedge T^m$ and
$s: X_+\wedge \bb P^{\wedge n} \to \mathcal F\wedge T^n$
are morphisms of pointed Nisnevich sheaves, then we define $\alpha^*(s)$ as the composite morphism
$$U_+\wedge\bb P^{\wedge m}\wedge \bb P^{\wedge n} \xrightarrow{\alpha\wedge\id} X_+\wedge T^m\wedge \bb P^{\wedge n}\cong
T^m\wedge X_+\wedge \bb P^{\wedge n} \xrightarrow{\id\wedge s} T^m\wedge \mathcal F\wedge T^n\cong\mathcal F\wedge T^m\wedge T^n.$$
By Corollary \ref{Fr_+_and_Fr_hom_+} the categories
$Fr_+(k)$ and $\cc Fr_+(k)$ are isomorphic.
The category isomorphism makes $\cc Fr_+(\mathcal F)$ a framed presheaf.
In the special case when $\mathcal F=Y/(Y-S)$ the bijections $a_{n,X,Y/(Y-S)}$
induce {\it isomorphisms of framed presheaves\/}
$a_{Y/(Y-S)}:Fr_+(-,Y/(Y-S))\to \cc Fr_+(-,Y/(Y-S))$, where $Fr_+(-,Y/(Y-S)):=\bigvee_{n\geq 0}Fr_n(-,Y/(Y-S))$.
}
\end{rem}

\begin{defs}\label{def: cc_Fr_F}{\rm
For a pointed Nisnevich sheaf $\cc F$ set
   $$\cc Fr(-,\cc F)=\colim(\mathcal F\lra{\sigma} \underline{\Hom}(\bb P^{\wedge 1},\mathcal F\wedge T)\lra{\sigma}
   \underline{\Hom}(\bb P^{\wedge 2},\mathcal F\wedge T^2)\lra{\sigma}\cdots),$$
where
${\sigma}(\Phi:U_+\wedge\bb P^{\wedge n}\to\cc F\wedge T^n)=
(U_+\wedge\bb P^{\wedge n+1}\xrightarrow{\Phi\wedge 1_{\bb P^{\wedge 1}}}\cc F\wedge T^{n}\wedge\bb P^{\wedge 1}
\xrightarrow{1\wedge\sigma}\cc F\wedge T^{n+1})$.
Observe that $\cc Fr(-,\cc F)$ is a framed Nisnevich sheaf.
}
\end{defs}

In the special case $\mathcal F=Y/(Y-S)$ one has that
\begin{equation}\label{a_Fr_to_Fr_hom}
a_{Y/(Y-S)}=\colim_n {a_{n,Y/(Y-S)}}: Fr(-,Y/(Y-S))\to \cc Fr(-,Y/(Y-S)).
\end{equation}
is {\it an isomorphism of framed sheaves}.

In what follows we shall identify the isomorphic framed sheaves $Fr(-,Y/(Y-S))$ and $\cc Fr(-,Y/(Y-S))$.
If we write $Fr(-,Y/(Y-S))$ then we use the geometric description of the sheaf. In turn,
the use of $\cc Fr(-,Y/(Y-S))$ will mostly refer to the equivalent categorical description.
The reader should always keep in mind the equivalent desriptions of both framed sheaves thanks to Voevodsky's Lemma.

\section{Motivic version of Segal's theorem}\label{segal}

After collecting necessary facts about framed correspondences in
previous sections, we can formulate the main computational result of
the paper. It is reminiscent of Segal's theorem computing the
suspension spectrum $\Sigma^\infty_{S^1}X$ of a topological space
$X$ as the Segal spectrum of an associated $\Gamma$-space
$B\Sigma_X$ (see~\cite[Section~3]{S} for more details). The motivic
counterpart of the Segal theorem computes the suspension $\bb
P^1$-spectrum $\Sigma^\infty_{\bb P^1}X_+$ of a smooth algebraic
variety $X$ in terms of associated motivic spaces with framed
correspondences. In a certain sense the theory of framed
correspondences gives rise to an infinite $\bb P^1$-loop space
machine. In order to formulate the theorem, we need some
preparations.

In Section~\ref{ohoho} we introduced Nisnevich sheaves
$Fr_s(-,Y/Y-S)$ and $Fr(-,Y/Y-S)$. In the special case when
$Y=X\times\bb A^n$ and $S=X\times 0$ we shall write $Fr_s(-,X\times
T^n)$ and $Fr(-,X\times T^n)$ to denote the Nisnevich sheaves
$Fr_s(-,X\times\bb A^n/(X\times\bb A^n-X\times 0))$ and
$Fr(-,X\times\bb A^n/(X\times\bb A^n-X\times 0))$ respectively.
Recall that an element of $Fr_s(-,X\times T^n)$ can be written as a
tuple $(Z,W,(\phi_1,...,\phi_s;\psi_1,...,\psi_n):W\to\bb
A^{s+n};g:W\to X)$ such that the support $Z=Z(\phi_1,...,\phi_s)\cap
Z(\psi_1,...,\psi_n)= Z(\phi_1,...,\phi_s,\psi_1,...,\psi_n)$.

For $X\in Sm/k$ and integers $s,n\geq 0$ consider a morphism of pointed Nisnevich sheaves
   \begin{equation}\label{sigmasn}
    \sigma_{s,n}: Fr_{s}(-,X\times T^n)\to \underline{\Hom}({\bb P^{\wedge 1}},Fr_{s}(-,X\times T^{n+1})).
   \end{equation}
It takes $(Z,W,\phi_1,...,\phi_s;\psi_1,...,\psi_n;g)$ to $(Z\times 0,W\times\bb A^1,\phi_1,...,\phi_{s};\psi_1,...,\psi_n,pr_{\bb A^1};g)$.
On the other hand, we have canonical maps~\eqref{eq:Sigma}
   \begin{equation}\label{sigmasntn}
    \sigma_{s,X\times T^n}: Fr_{s}(-,X\times T^n)\to Fr_{s+1}(-,X\times T^{n})
   \end{equation}
taking $(Z,W,\phi_1,...,\phi_s;\psi_1,...,\psi_n;g)$ to $(Z\times 0,W\times\bb A^1,\phi_1,...,\phi_{s},pr_{\bb A^1};\psi_1,...,\psi_n;g)$.

We have a commutative diagram
   $$\xymatrix{Fr_{s}(-,X\times T^n)\ar[r]^(.4){\sigma_{s,n}}\ar[d]_{\sigma_{s,X\times T^n}}
                       &\underline{\Hom}({\bb P^{\wedge 1}},Fr_{s}(-,X\times T^{n+1}))\ar[d]^{(\sigma_{s,X\times T^{n+1}})_*}\\
                       Fr_{s+1}(-,X\times T^n)\ar[r]^(.4){\sigma_{s+1,n}}&\underline{\Hom}({\bb P^{\wedge 1}},Fr_{s+1}(-,X\times T^{n+1})).}$$
Passing to colimits in the $s$-direction, we get a morphism of pointed sheaves
   $$\sigma_{n}: Fr(-,X\times T^n)\to \underline{\Hom}({\bb P^{\wedge 1}},Fr(-,X\times T^{n+1})).$$
We can form a $\bb P^{1}$-spectrum
   \begin{equation}\label{frp1t}
   Fr_{\bb P^{\wedge 1},T}(X)=(Fr(-,X),Fr(-,X\times T),Fr(-,X\times T^2), ... )
   \end{equation}
with structure morphisms given by $\sigma_n$-s.

We can now take the Suslin complex of each motivic space of the spectrum $Fr_{\bb P^{\wedge 1},T}(X)$
to form a $\bb P^{1}$-spectrum
   $$M_{\bb P^{\wedge 1}}(X)=(C_*Fr(-,X),C_*Fr(-,X\times T),C_*Fr(-,X\times T^2), ... )$$
with structure maps defined by $C_*(\sigma_n)$-s.

There is a canonical morphism of $\bb P^1$-spectra
   $$\kappa: \Sigma^{\infty}_{\bb P^{1}}X_+\to M_{\bb P^{\wedge 1}}(X)$$
given by the section $\id_X\in Fr_0(X,X)$ (recall that a morphism
from the suspension spectrum of a variety $X\in Sm/k$ to any other
spectrum is fully determined by a section of the zeroth motivic
space of the spectrum at $X$).

By~\cite[2.7]{Jar1} the category of simplicial Nisnevich sheaves on
$Sm/k$ has the injective local model structure with cofibrations
monomorphisms and local weak equivalences. Take a fibrant
replacement $C_*Fr(-,X\times T^n)_f$ of every motivic space
$C_*Fr(-,X\times T^n)$ within the injective local model structure.
We then arrive at a $\bb P^1$-spectrum
   $$M_{\bb P^{\wedge 1}}(X)_f=(C_*Fr(-,X)_f,C_*Fr(-,X\times T)_f,C_*Fr(-,X\times T^2)_f, ... ).$$
Notice that $M_{\bb P^{\wedge 1}}(X)_f$ is a fibrant replacement of
the $\bb P^1$-spectrum $M_{\bb P^{\wedge 1}}(X)$ within the level
injective local model structure of $\bb P^1$-spectra.

If the characteristic of the base field is different from 2, then we shall prove (see Theorem~\ref{Segal_Thm_II})
that $M_{\bb P^{\wedge 1}}(X)_f$ is a positively fibrant $\bb P^1$-spectrum. In turn, if $\charr k=2$ then we are able to prove a
similar result only ``after inverting 2". Precisely, there are explicitely constructed motivic spaces
$C_*Fr(-,X\times T^n)\{1/2\}$, $n>0$, such that the canonical morphism of spaces
   $$\tau:C_*Fr(-,X\times T^n)\to C_*Fr(-,X\times T^n)\{1/2\}$$
induces an isomorphism of sheaves $\pi_*^{\nis}(\tau)\otimes\bb Z[1/2]$ (see Proposition~\ref{frac12}).

We can form a $\bb P^{1}$-spectrum
   $$M_{\bb P^{\wedge 1}}(X)\{1/2\}:=(C_*Fr(-,X),C_*Fr(-,X\times T)\{1/2\},C_*Fr(-,X\times T^2)\{1/2\}, ... )$$
with structure maps defined as for $M_{\bb P^{\wedge 1}}(X)$. The
reader may have observed that we do not modify the zeroth space
$C_*Fr(-,X)$. The reason is that we only deal with spaces in
positive degrees in the proof of Theorem~\ref{Segal_Thm_II}. Take a
fibrant replacement $C_*Fr(-,X\times T^n)\{1/2\}_f$ of every motivic
space $C_*Fr(-,X\times T^n)\{1/2\}$ within the injective local model
structure. We then arrive at a $\bb P^1$-spectrum
   $$M_{\bb P^{\wedge 1}}(X)\{1/2\}_f=(C_*Fr(-,X)_f,C_*Fr(-,X\times T)\{1/2\}_f,C_*Fr(-,X\times T^2)\{1/2\}_f, ... ).$$
Notice that $M_{\bb P^{\wedge 1}}(X)\{1/2\}_f$ is a fibrant replacement of the $\bb P^1$-spectrum
$M_{\bb P^{\wedge 1}}(X)\{1/2\}$ within the level injective local model structure of $\bb P^1$-spectra.

Let
   $$\kappa_f:\Sigma^{\infty}_{\bb P^1}X_+\to M_{\bb P^{\wedge 1}}(X)\to M_{\bb P^{\wedge 1}}(X)_f$$
denote the composite morphism. We also have a morphism of spectra
   $$\tau_f\circ\kappa_f:\Sigma^{\infty}_{\bb P^1}X_+\to M_{\bb P^{\wedge 1}}(X)_f\to M_{\bb P^{\wedge 1}}(X)\{1/2\}_f.$$
A morphism of $\bb P^1$-spectra is called a {\it $2^{-1}$-stable motivic equivalence\/} if it becomes an isomorphism
in the category $SH(k)[1/2]$.

A motivic counterpart of the Segal theorem says:

\begin{thm}\label{Segal_Thm_II}
Let $k$ be an infinite perfect field. Then the following statements are true:
\begin{enumerate}
\item The morphism $\kappa_f: \Sigma^{\infty}_{\bb P^1}X_+\to M_{\bb P^{\wedge 1}}(X)_f$
is a stable motivic equivalence of $\bb P^1$-spectra and the morphism
$\tau_f\circ\kappa_f:\Sigma^{\infty}_{\bb P^1}X_+\to M_{\bb P^{\wedge 1}}(X)\{1/2\}_f$
is a $2^{-1}$-stable motivic equivalence of $\bb P^1$-spectra.
\item If $\charr k\ne 2$ (respectively $k$ is of any characteristic) then the
$\bb P^1$-spectrum $M_{\bb P^{\wedge 1}}(X)_f$ (respectively $M_{\bb
P^{\wedge 1}}(X)\{1/2\}_f$) is a motivically fibrant
$\Omega$-spectrum in posisitive degrees. This means that for every
positive integer $n>0$ each motivic space $C_*(Fr(-,X\times
T^{n}))_f$ (respectively $C_*(Fr(-,X\times T^{n}))\{1/2\}_f$) is
motivically fibrant in the Morel--Voevod\-sky~\cite{MV} motivic
model category of simplicial Nisnevich sheaves and the structure map
   $$C_*(Fr(-,X\times T^{n}))_f\to \Omega_{\bb P^1}(C_*(Fr(-,X\times T^{n+1}))_f)$$
(respectively  $C_*(Fr(-,X\times T^{n}))\{1/2\}_f\to \Omega_{\bb P^1}(C_*(Fr(-,X\times T^{n+1}))\{1/2\}_f)$)
is a  weak equivalence scheme\-wise.
\end{enumerate}
\end{thm}

We shall also extend Theorem~\ref{Segal_Thm_II} to directed colimits of
simplicial schemes (see Theorem~\ref{Segal_Thm_Simpl}).
The next two corollaries are immediate consequences of the preceding theorem.

\begin{cor}\label{segalsledstvie1}
Let $k$ be an infinite perfect field of characteristic different from 2.
Then for any positive integer $m>0$ the natural morphism of $\bb P^1$-spectra\footnotesize
   $$\kappa_{f}: \Sigma^{\infty}_{\bb P^1}(X_+\wedge\bb P^{\wedge m})\to M_{\bb P^{\wedge 1}}(X\times T^m)_f:=
       (C_*Fr(-,X\times T^m)_f,C_*Fr(-,X\times T^{m+1})_f,C_*Fr(-,X\times T^{m+2})_f, ... )$$
\normalsize is a fibrant replacement of the suspension $\bb P^1$-spectrum $\Sigma_{\bb P^1}^\infty(X_+\wedge\bb P^{\wedge m})$
in the stable motivic model structure of $\bb P^1$-spectra in the sense of Jardine~\cite{Jar}.
\end{cor}

Given a $\bb P^1$-spectrum $E$, let $\cc E$ be an $\Omega$-spectrum
stably equivalent to $E$. By $\Omega^{\infty}_{\bb P^1}(E)$ we mean
the zeroth motivic space $\cc E_0$. If $E=\Sigma_{\bb P^1}^\infty\cc
X$ is the suspension $\bb P^1$-spectrum of a pointed motivic space
$\cc X$, we shall write $\Omega^{\infty}_{\bb
P^1}\Sigma^{\infty}_{\bb P^1}(\cc X)$ to denote
$\Omega^{\infty}_{\bb P^1}(E)$.\label{ominfty}

\begin{cor}\label{segalsledstvie2}
Let $k$ be an infinite perfect field of characteristic different from 2.
Then for any positive integer $m>0$ the natural morphism of motivic spaces
   $$C_*(Fr(X\times T^m))\to\Omega^{\infty}_{\bb P^1}\Sigma^{\infty}_{\bb P^1}(X_+\wedge\bb P^{\wedge m})$$
is a stalkwise weak equivalence for the Nisnevich topology.
Particularly, for any field extension $K/k$ the natural morphism
of simplicial sets
$$Fr(\Delta^{\bullet}_K,X\times T^m)\to\Omega^{\infty}_{\bb P^1}\Sigma^{\infty}_{\bb P^1}(X_+\wedge\bb P^{\wedge m})(K)$$
is a weak equivalence.
\end{cor}

The proof of Theorem~\ref{Segal_Thm_II} is lengthy and is postponed.
Although it states something for motivic spaces, the main strategy
to prove it is to use the machinery of framed motives introduced and
studied in this paper. By definition, the framed motive of a variety
or a sheaf is a $S^1$-spectrum of simplicial Nisnevich sheaves, and
hence may have nothing to do with Theorem~\ref{Segal_Thm_II} at the
first glance. {\it But this is not the case! It is the theory of
framed motives that allows to prove Theorem~\ref{Segal_Thm_II}.}

The proof also depends on a theorem of~\cite{GP4} about homotopy invariant presheaves with framed correspondences and
further two papers~\cite{AGP,GPN}, in which the Cancellation Theorem for framed motives of algebraic varieties is proved and
framed motives of relative motivic spheres are computed.

\section{Framed motives}

As we have mentioned above, framed motives give the main technical
tool to prove Theorem~\ref{Segal_Thm_II}. Before introducing them,
we fix the following useful

\begin{framework}\label{genfr}{\rm
Let $(\cc V,\otimes)$ be a closed symmetric monoidal category and
$\cc C$ is a bicomplete category which is tensored and cotensored
over $\cc V$. Then for every $V\in\cc V$ and $X\in\cc C$ there are
defined objects $V\otimes X$, $X\otimes V$, $\underline{\Hom}(V,X)$
of $\cc C$. They are all functorial in $V$ and $X$. Moreover, for
every morphism $u:V\to V'$ in $\cc V$ the square
   \begin{equation}\label{actionV}
    \xymatrix{X\ar[rr]^{-\otimes V}\ar[d]_{-\otimes V'}&&\underline{\Hom}(V,X\otimes V)\ar[d]^{u_*}\\
    \underline{\Hom}(V',X\otimes V')\ar[rr]_{u^*}&&\underline{\Hom}(V,X\otimes V')}
   \end{equation}
is commutative.

As an important example, $\cc V$ is the category
$(sShv_\bullet(Sm/k),\wedge)$ of Nisnevich sheaves of pointed
simplicial sets and $\cc C$ is either $sShv_\bullet(Sm/k)$ or the
category of $S^1$-spectra of simplicial Nisnevich sheaves.

Another example is the category $(Fr_0(k),\times, pt)$ and the category $Fr_+(k)$
The functor $$Fr_+(k)\times Fr_0(k) \to Fr_+(k)$$
takes $(X,Y)$ to $X\times Y$.
}\end{framework}

Let $\Gamma^{\op}$ be the category of finite pointed sets and
pointed maps. Its skeleton has objects $n^+=\{0,1,\ldots,n\}$. We
shall also regard each finite pointed set as a pointed smooth
scheme. For example, we identify $n^+$ with the pointed scheme
$(\sqcup_{1}^n\spec k)_+$ with the distinguished point +
corresponding to $0\in n^+$. Note that $0^+=\emptyset_+$. A {\it
$\Gamma$-space\/} is a covariant functor from $\Gamma^{\op}$ to the
category of simplicial sets taking $0^+$ to a one point simplicial
set. A morphism of $\Gamma$-spaces is a natural transformation of
functors. A $\Gamma$-space $X$ is called {\it special\/} if the map
$X((k + l)^+):X(k^+)\times X(l^+)$ induced by the projections from
$(k+l)^+\cong k^+\vee l^+$ to $k^+$ and $l^+$ is a weak equivalence
for all $k$ and $l$. $X$ is called {\it very special\/} if it is
special and the monoid $\pi_0(X(1^+))$ is a group.

{\it In what follows we shall regard $\Gamma^{\op}$ as a full
subcategory in $sShv_\bullet(Sm/k)$ by means of the identification
$K\in\Gamma^{\op}$ with the pointed scheme $(\spec
k\sqcup\ldots\sqcup\spec k)_+$, where the coproduct is indexed by
the non-based elements in $K$}.

By the General Framework above, for every $\cc F,\cc G\in
sShv_\bullet(Sm/k)$ the association
   $$K\in\Gamma^{\op}\mapsto\Hom_{sShv_\bullet(Sm/k)}(\cc F,\cc G\wedge K)$$
gives rise to a $\Gamma$-space, where the right hand side is
regarded as a discrete simplicial set. In particular, if $\cc
F=X_+\wedge\bb P^{\wedge n}$ and $\cc G=\cc H\wedge T^n$ with $X\in
Sm/k$, $\cc H\in sShv_\bullet(Sm/k)$, we have that the association
   $$K\in\Gamma^{\op}\mapsto\cc Fr_n(X,\cc H\wedge K):=
     \Hom_{sShv_\bullet(Sm/k)}(X_+\wedge\bb P^{\wedge n},\cc H\wedge T^n\wedge K)$$
is a $\Gamma$-space. Taking the colimit over $n$, we get that
   $$K\in\Gamma^{\op}\mapsto\cc Fr(X,\cc H\wedge K)=\colim_n(\cc Fr_n(X,\cc H\wedge K))$$
is a $\Gamma$-space as well.

Using the geometric description of framed correspondences for $\cc
H=Y_+$, $Y\in Sm/k$, the $\Gamma$-spaces $K\in\Gamma^{\op}\mapsto\cc
Fr_n(X,\cc H\wedge K)$ and $K\in\Gamma^{\op}\mapsto\cc Fr(X,\cc
H\wedge K)$ can equivalently be defined as
   $$K\in\Gamma^{\op}\mapsto Fr_n(X,Y\otimes K)\quad\textrm{and}\quad Fr(X,Y\otimes K)$$
respectively. Here $Y\otimes K:=Y\sqcup\ldots\sqcup Y$ with the
coproduct indexed by the non-based elements in $K$. Observe that
$\emptyset\otimes K=\emptyset$ and $X\otimes *=\emptyset$. These
$\Gamma$-spaces are functorial in $X$ and $Y$ in framed
correspondences of level zero. The second $\Gamma$-space is
furthermore a framed functor in $X$.

\begin{defs}\label{def:SmOpFr_0}{\rm
We define a category $SmOp(Fr_0(k))$, which will be often used in
our constructions. Its objects are pairs $(X,U)$, where $X\in Sm/k$
and $U\subset X$ is its open subset. A morphism between $(X,U)$ and
$(X',U')$ in $SmOp(Fr_0(k))$ is a morphism $f\in Fr_0(X,X')$ such
that $f(U)\subset U'$. We shall also identify $X\in Sm/k$ with the
pair $(X,\emptyset)\in SmOp(Fr_0(k))$.

}\end{defs}

The category $SmOp(Fr_0(k))$ is symmetric monoidal with the monoidal
product $\wedge$ given by
   $$(X,U)\wedge (Y,V):=(X\times Y,X\times V\cup U\times Y).$$
The point $pt$ is its monoidal unit. Also, by $(X,U)\sqcup (Y,V)$ we shall mean $(X \sqcup Y, U \sqcup V)$.

Let $\Delta^{\op} SmOp(\Fr_0(k))$ be the category of simplicial
objects in $SmOp(Fr_0(k))$. There is an obvious functor $spc:
SmOp(Fr_0(k)) \to Shv_\bullet(Sm/k)$ sending an object $(X,U)\in
SmOp(Fr_0(k))$ to the Nisnevich sheaf $X/U$. Observe that this
functor is a strict symmetric monoidal functor. It induces a functor
   $$spc:\Delta^{\op} SmOp(Fr_0(k))\to sShv_\bullet(Sm/k),$$
taking an object $[n]\mapsto (Y_n,U_n)$ to the simplicial Nisnevich
sheaf $[n]\mapsto (Y_n/U_n)$.

Given $Y\in SmOp(Fr_0(k))$ there is a $\Gamma$-space
   $$K\in\Gamma^{\op}\mapsto Fr(X,Y\otimes K):=\cc Fr(X,spc(Y)\otimes K).$$
Notice that the right hand side has an explicit geometric
description thanks to Voevodsky's lemma.

\begin{defs}\label{frmotive}{\rm
(1) The {\it framed motive $\cc M_{fr}(\cc G)$ of a pointed
Nisnevich simplicial sheaf $\cc G$\/} is the Segal $S^1$-spectrum
$(C_*\cc Fr(-,\cc G),C_*\cc Fr(-,\cc G\wedge S^1),C_*\cc Fr(-,\cc
G\wedge S^2),\ldots)$ associated with the
$\Gamma$-space $K\in\Gamma^{\op}\mapsto C_*\cc Fr(-,\cc G\wedge
K)=\cc Fr(\Delta^\bullet_+\wedge-,\cc G\wedge K)$. More precisely,
each structure map
   $$C_*\cc Fr(-,\cc G\wedge S^m)\wedge S^1\to C_*\cc Fr(-,\cc G\wedge S^{m+1})$$
is given as follows. For any $r$ and $m$, it coincides termwise with
the natural morphisms
   $$\bigvee\cc Fr(\Delta^r_+\wedge-,\cc G\wedge S^m)\to\cc Fr(\Delta^r_+\wedge-,\bigvee(\cc G\wedge S^m)),$$
where coproducts are indexed by non-basepoint elements of
$S^1_n=n^+$.

(2) The {\it framed motive $M_{fr}(Y)$ of $Y\in\Delta^{\op}
SmOp(Fr_0(k))$\/} is the framed motive $\cc M_{fr}(spc(Y))$. It is
the Segal $S^1$-spectrum $(C_*Fr(-,Y),C_*Fr(-,Y\otimes
S^1),C_*Fr(-,Y\otimes S^2),\ldots)$ associated
with the $\Gamma$-space $K\in\Gamma^{\op}\mapsto C_*Fr(-,Y\otimes
K)=Fr(\Delta^\bullet\times-,Y\otimes K)$.

(3) In particular, by the {\it framed motive $M_{fr}(Y)$ of a smooth
algebraic variety $Y\in Sm/k$\/} we mean the framed motive of
$(Y,\emptyset)\in SmOp(Fr_0(k))$.
}
\end{defs}

\begin{rem}\label{M_and_cal_M}{\rm
(1) We should point out that it is framed motives $M_{fr}(Y^\bullet)$ of simplicial $k$-varieties
which are used to construct the functor $M_{fr}: H_{\mathbb A^1}(k)\to SH_{S^1}(k)$ in Section~\ref{furtherappl}.

(2) Framed motives of pointed Nisnevich simplicial sheaves are not suitable for constructing such a functor.
Particularly, we do not expect that the functor $\cc G \mapsto \cc M_{fr}(\cc G)$ from Definition~\ref{frmotive}(1)
preserves motivic equivalences. Therefore we do not expect that for a general motivic space $\cc G$ the value
of the functor $M_{fr}$ on $\cc G$ constructed in Section~\ref{furtherappl} has
stable motivic homotopy type of $\cc M_{fr}(\cc G)$ from Definition~\ref{frmotive}(1).

(3) It is for this reason that the main result of~\cite{GPN} stating that the natural motivic equivalence
$X\times (\A^1// \mathbb G_m)^{\wedge n}\to X\times T^n$
of motivic spaces induces for any $n\geq 1$ a motivic equivalence (and even a level Nisnevich local weak equivalence)
$M_{fr}(X\times (\A^1// \mathbb G_m)^{\wedge n}) \simeq \mathcal M_{fr}(X_+\wedge T^n)$
of $S^1$-spectra is not obvious at all. In other words, the framed motive of the sheaf $X_+\wedge T^n$
is computed as the framed motive of the associated simplicial scheme. This result is
necessary to prove Theorem~\ref{Segal_Thm_II}. It is also of independent interest.

(4) More generally, we can raise a problem asking for which motivic spaces $\cc G$ the framed motive
$\cc M_{fr}(\cc G)$ from Definition~\ref{frmotive}(1) is isomorphic in $SH_{S^1}(k)$ to
its image under the functor $M_{fr}: H_{\mathbb A^1}(k)\to SH_{S^1}(k)$ constructed in Section~\ref{furtherappl}.

(5) However, the functor $M_{fr}: \Delta^{\op}(Fr_0(k)) \to Sp_{S^1}(k)$ does preserve motivic equivalences
(see Corollary~\ref{Motivic_Segal_Thm_Cor}).

(6) Framed motives of the form $M_{fr}(Y)$ with $Y\in\Delta^{\op} SmOp(Fr_0(k))$
are of great utility to prove Theorem~\ref{Segal_Thm_II}.

}\end{rem}

The framed motive $M_{fr}(Y)$ or $\cc M_{fr}(\cc G)$ is a symmetric
semistable $S^1$-spectrum, because it is the value of the $\Gamma$-space $C_*Fr(-,Y)$
or $C_*\cc Fr(-,\cc G)$ at the sphere spectrum $\bb S=(S^0,S^1,S^2,\ldots)$. Since we also
deal with framed spectra/spaces whose presheaves of homotopy groups are pre\-sheaves
of $\bb Z[1/2]$-modules, we have to introduce ``framed motives with $1/2$-coefficients".
For this consider $2\in\pi_1(S^1)=[S^1,Ex^\infty S^1]=\bb Z$ (here square brackets refer to
the maps up to simplicial homotopy). There is the smallest integer $k>0$ such that any
representative of $2\in\pi_1(S^1)$ factors as
   $$S^1\to \sd^k S^1\to Ex^\infty S^1.$$
Here $\sd^k S^1$ stands for the $k$th subdivision of $S^1$. If
there is no likelihood of confusion, we denote the left arrow by $2:S^1\to \sd^k S^1$ and fix a representative
in its homotopy class from now on.
We also denote the finite pointed simplicial set $\sd^k S^1$ by $\cc S^1$. We also set
$\cc S^n:=\cc S^1\wedge\bl n\cdots\wedge\cc S^1$.

Let $\bb S[1/2]$ be the spectrum $(S^0,\cc S^1,\cc S^2,\ldots)$. Its structure maps are defined as
$\cc S^n\wedge S^1\xrightarrow{\id\wedge 2}\cc S^n\wedge\cc S^1=\cc S^{n+1}$.

\begin{defs}\label{frmotive1/2}{\rm
(1) The {\it framed motive with $1/2$-coefficients $\cc M_{fr}(\cc G)\{1/2\}$ of a pointed
Nisnevich simplicial sheaf $\cc G$\/} is the value of the $\Gamma$-space $C_*\cc Fr(-,\cc G)$
at the spectrum $\bb S[1/2]$. Explicitly, it is the symmetric $S^1$-spectrum
$(C_*\cc Fr(-,\cc G),C_*\cc Fr(-,\cc G\wedge\cc S^1),C_*\cc Fr(-,\cc
G\wedge\cc S^2),\ldots)$ with each structure map given by the composition
   $$C_*\cc Fr(-,\cc G\wedge\cc S^m)\wedge S^1\to C_*\cc Fr(-,\cc G\wedge\cc S^{m}\wedge S^1)
       \xrightarrow{2_*} C_*\cc Fr(-,\cc G\wedge\cc S^{m+1}).$$

(2) The {\it framed motive with $1/2$-coefficients $M_{fr}(Y)\{1/2\}$ of $Y\in\Delta^{\op}
SmOp(Fr_0(k))$\/} is the framed motive $\cc M_{fr}(spc(Y))\{1/2\}$. It is
the symmetric $S^1$-spectrum $(C_*Fr(-,Y),C_*Fr(-,Y\otimes
\cc S^1),C_*Fr(-,Y\otimes\cc S^2),\ldots)$ with structure maps defined as above.

(3) In particular, by the {\it framed motive $M_{fr}(Y)\{1/2\}$ of a smooth
algebraic variety $Y\in Sm/k$\/} we mean the framed motive with $1/2$-coefficients of
$(Y,\emptyset)\in SmOp(Fr_0(k))$.

}\end{defs}

Our next goal is to show that the framed motive $M_{fr}(Y)$ of
$Y\in\Delta^{\op} SmOp(Fr_0(k))$ is a positively fibrant
$\Omega$-spectrum. To this end we need to prove the ``Additivity
Theorem".

\section{Additivity Theorem}\label{additivity_section}

In this section we prove the Additivity Theorem. It is reminiscent
of the Additivity Theorem in algebraic $K$-theory. We shall use it
to produce special $\Gamma$-spaces in the sense of Segal~\cite{S}
for associated motivic spaces with framed correspondences. In
particular, Segal's machine then implies that the framed motive of a
variety or, more generally, $Y\in\Delta^{\op}SmOp(Fr_0(k))$ is a
positively fibrant $S^1$-spectrum.

Following \cite{Voe2} for any $X\in Sm/k$ denote by $m$ the explicit correspondence from $X$ to
$X\sqcup X$ with $U=(\bb A^1 -\{0\}\sqcup\bb A^1 -\{1\})_X$, $\phi=
(t-1)\sqcup t$ where $t:\bb A^1_X\to X$ is the projection and
$g:(\bb A^1-\{0\}\sqcup \bb A^1-\{1\})_X\to X\sqcup X$. For any
framed functor $\cc F$ it defines a map
   $$\cc F(X)\times\cc F(X)= \cc F(X\sqcup X)\xrightarrow{m^*} \cc F(X).$$

\begin{defs}{\rm
Let $\mathcal{F}$ and $\mathcal{G}$ be two presheaves of sets on the
category of $k$-smooth schemes and let $\phi_0,\phi_1: \mathcal{F}
\rightrightarrows \mathcal{G}$ be two morphisms. An {\it
$\A^1$-homotopy\/} between $\phi_0$ and $\phi_1$ is a morphism $H:
\mathcal{F}\to \text{\underline {Hom}}(\A^1,\mathcal{G})$ such that
$H_0=\phi_0$ and $H_1=\phi_1$. We write $\phi_0\sim \phi_1$ if
there is an $\A^1$-homotopy between $\phi_0$ and $\phi_1$.
We say that $\phi_0,\phi_n: \cc F \rightrightarrows \mathcal{G}$
are {\it $\bb A^1$-homotopic}, if there is a chain of morphisms
$\phi_1,...,\phi_n$ such that $\phi_i\sim \phi_{i+1}$ for $i=0,1,..., n-1$.
}
\end{defs}

We now want to discuss matrices actions on framed correspondences
and $\bb A^1$-homotopies associated to them. Let $Y$ be a $k$-smooth scheme and let $A\in GL_n(k)$ be a
matrix. Then $A$ defines an automorphism $\phi_{A\bot Id_m}: Fr_{n+m}(-,Y) \to Fr_{n+m}(-,Y)$
of the presheaf $Fr_{n+m}(-,Y)$ in the following way. Given $W\in Sm/k$ and
an {\it explicit framed
correspondence  of level n}
$$\Phi=(\mathbb A^{n+m}_X \xleftarrow{p} U,\phi: U \to \mathbb A^{n+m}_k,g: U\to Y) \in Fr_{n+m}(X,Y),$$
set
$\phi_{A\bot Id_m}(\Phi)=((A\bot Id_m)\circ p,(A\bot Id_m)\circ \phi, g)$.
Within the notation of Remark~\ref{short_Notation}
   $$\phi_{A\bot Id_m}(Z,U,(\varphi_1,\varphi_2,\dots,\varphi_{n+m}),g))
       := ((A\bot Id_m)(Z),U,(A\bot Id_m)\circ(\varphi_1,\varphi_2,\dots,\varphi_{n+m}),g),$$
where $A\bot Id_m$ is a linear automorphism of $\A^{n+m}_k$.
In more details, if
$(U,\mathbb A^{n+m}_X \xleftarrow{p} U, s: Z\to U)$
is an \'{e}tale neighbourhood of $Z$ in $\mathbb A^{n+m}_X$, then we take
$$(U,\mathbb A^{n+m}_X \xleftarrow{(A\bot Id_m)\circ p} U, s\circ ((A\bot Id_m)^{-1}|_{(A\bot Id_m)(Z)}):(A\bot Id_m)(Z)\to U)$$
as an \'{e}tale neighbourhood of $(A\bot Id_m)(Z)$ in $\mathbb A^{n+m}_X$.
Clearly, $\sigma\circ \phi_{A\bot Id_m}=\phi_{A\bot Id_{m+1}}\circ \sigma$.
Hence the maps $\phi_{A\bot Id_m}$ give rise to a unique automorphism of {\it presheaves on $Sm/k$}
   \begin{equation}\label{phi_A}
    \phi_A: Fr(-,Y) \to Fr(-,Y)
   \end{equation}
such that for any $m\geq 0$ one has $\phi_A|_{Fr_{n+m}(-,Y)}=\phi_{A\bot Id_m}$.

\begin{defs} \label{d:homotopy_1}{\rm
Let $A\in SL_n(k)$. Choose a matrix $A_s\in SL_n(k[s])$ such that
$A_0=id$ and $A_1=A$. The matrices $A_s\bot Id_m\in SL_{n+m}(k[s])$, regarded as morphisms
$\A^{n+m} \times \A^1 \to \A^{n+m}$, give rise to an $\A^1$-homotopy $h$ between the automorphisms
$id$ and $\phi_A$ of $Fr(-,Y)$ as follows. Given
$a=(Z,U,(\varphi_1,\varphi_2,...,\varphi_{n+m}),g)
\in Fr_n(-,Y)$, one sets
   $$h(a)=(Z\times \A^1,U\times \A^1, (A_s\bot Id_m)\circ (\varphi\times id_{\A^1}),g\circ pr_U) \in Fr_n(W\times \A^1,Y).$$
In this way we get a morphism $h: Fr(-,Y)\to Fr(-\times \bb A^1,Y)$ such that
$h_0=id$ and $h_1=\phi_A$. We see that $h$ is an $\bb A^1$-homotopy between the identity and $\phi_A$.
}
\end{defs}

\begin{defs}\label{d:homotopy_2}{\rm
Let $\tau \in \Sigma_n$ be an even permutation regarded as a matrix in $SL_n(k)$.
Let $A_s\in SL_n(k[s])$ be such that
$A_0=id$ and $A_1=\tau$. Then the morphism $h$ from Definition~\ref{d:homotopy_1}
defines an $\bb A^1$-homotopy between the automorphisms $\phi_{id}$ and $\phi_{\tau}$ of $Fr(-,Y)$.
}
\end{defs}

We are now in a position to prove the Additivity Theorem. The category
$SmOp(Fr_0(k))$ was introduced in Definition~\ref{def:SmOpFr_0}. Given a simplicial object $Y$ of $SmOp(Fr_0(k))$,
by $C_*Fr(-,Y)$ we mean as usual the diagonal of the bisimplicial presheaf $(m,n)\mapsto Fr(\Delta^m\times-,Y_n)$.

\begin{thm}[Additivity]\label{additivity}
For any two simplicial objects $Y_1, Y_2$ in $SmOp(Fr_0(k))$, the natural map
$\alpha: Fr(-,Y_1\sqcup Y_2)\to Fr(-,Y_1)\times Fr(-,Y_2)$, given by the morphisms
$\id\sqcup\emptyset:Y_1\sqcup Y_2\to Y_1,\emptyset\sqcup\id:Y_1\sqcup Y_2\to Y_2$, induces a map
of simplicial framed presheaves
   $$C_*(\alpha):C_*Fr(-,Y_1\sqcup Y_2)\to C_*Fr(-,Y_1)\times C_*Fr(-,Y_2),$$
which is a schemewise weak equivalence.
\end{thm}

\begin{proof}
Since the realization functor takes simplicial weak equivalences of simplicial sets to weak equivalences,
it is enough to prove the theorem for any objects $Y_1,Y_2$ in $SmOp(Fr_0(k))$.
Moreover, it is sufficient to prove that for any $X\in Sm/k$ and any finite unpointed simplicial set $K$, the map
    $$C_*(\alpha)(X,K): [K,C_*Fr(X,Y_1\sqcup Y_2)]\to [K,C_*Fr(X,Y_1)\times C_*Fr(X,Y_2)]$$
between the Hom-sets in the homotopy category $\Ho(sSets)$ of unpointed simplicial sets is a bijection.
Furthermore, for the simplicity of the exposition we shall assume that $Y_1,Y_2$ are just $k$-smooth varieties.
Let $Y=Y_1\sqcup Y_2$. Define a morphism
   $$\beta:Fr(-,Y_1)\times Fr(-,Y_2)\to Fr(-,Y)$$
of presheaves on $Sm/k$ by the following commutative diagram:
   $$\xymatrix{Fr(X,Y_1)\times Fr(X,Y_2)\ar[d]_{(i_1)_*\times(i_2)_*}\ar[rr]^\beta &&Fr(X,Y)\\
               Fr(X,Y)\times Fr(X,Y)\\
               Fr(X\sqcup X,Y)\ar[u]^{(j_1)^*\times(j_2)^*}_\cong\ar@/_2pc/[uurr]_{m^*}}$$
Here $i_\epsilon:Y_\epsilon\to Y$, $\epsilon=1,2$, is the
corresponding embedding.

We claim that $(\beta\alpha)|_{Fr_{2n}(-,Y)}$ is $\bb A^1$-homotopic to the inclusion
$in_{2n}: Fr_{2n}(-,Y)\to Fr(-,Y)$ and $(\alpha\beta)|_{Fr_{2n}(-,Y_1)\times Fr_{2n}(-,Y_2)}$ is $\bb A^1$-homotopic to the inclusion
   $$inc^1_{2n}\times inc^2_{2n}: Fr_{2n}(-,Y_1)\times Fr_{2n}(-,Y_2) \to Fr(-,Y_1)\times Fr(-,Y_2).$$

The first of these $\bb A^1$-homotopies will imply that $C_*(\beta)\circ C_*(\alpha)|_{C_*(Fr_{2n}(-,Y))}$
is simplicially homotopic to the inclusion $C_*(in_{2n})$, because $C_*(-)$ converts $\bb A^1$-homotopies
into simplicial ones. The second of these $\bb A^1$-homotopies will imply that
$(C_*(\alpha)\circ C_*(\beta))|_{C_*(Fr_{2n}(-,Y_1)\times C_*(Fr_{2n}(-,Y_1))}$
is simplicially homotopic to the inclusion $C_*(inc^1_{2n}\times inc^2_{2n})$. It will follow that
for any $X\in Sm/k$ and any finite simplicial set $K$ the map $C_*(\alpha)(X,K)$ is bijective. Indeed,
one should use the fact that the functor $[K,-]:\Ho(sSets)\to\Ho(sSets)$ commutes with sequential colimits
whenever $K$ is finite. It therefore remains to prove the claim.

Firstly, let us focus on the morphism $\alpha\beta$. The map $\alpha\beta$ is of the form
   $$\rho_1\times\rho_2: Fr(X,Y_1)\times Fr(X,Y_2)\to Fr(X,Y_1)\times Fr(X,Y_2).$$
Here $\rho_1$ takes a framed correspondence $(Z_1,W_1,{\underline \phi^{(1)};g_1})$
of level $n$ to the framed correspondence
$(0\times Z_1,\bb A^1\times W_1,t_1,{\underline \phi^{(1)}};g_1)$ of level $n+1$,
$\rho_2$ takes a framed correspondence
$(Z_2,W_2,{\underline \phi^{(2)};g_2})$ of level $n$
to the framed correspondence
$(1\times Z_2,\bb A^1\times W_2,t_0-1,{\underline \phi^{(2)}};g_2)$ of level $n+1$.
We first observe that the morphism $\rho_2$ is $\bb A^1$-homotopic to the
morphism $\rho^0_2: Fr(-,Y_2)\to Fr(-,Y_2)$ taking
a framed correspondence  $(Z_2,W_2,{\underline \phi^{(2)};g_2})$
of level $n$ to the framed correspondence
$(0\times Z_2,\bb A^1\times W_2,t_0,{\underline \phi^{(2)}};g_2)$ of level $n+1$.
To see this, send a framed correspondence $(Z_2,W_2,{\underline \phi^{(2)};g_2})$ of level $n$
to the framed correspondence
   $$(\Delta\times Z_2,\bb A^1\times \bb A^1\times W_2,t_0-\lambda,{\underline \phi^{(2)}};g_2)$$
of level $n+1$. Evaluating the latter framed correspondence at $\lambda=1$, we get
$\rho_2(Z_2,W_2,{\underline \phi^{(2)};g_2})$. Evaluating the same framed correspondence at $\lambda=0$,
we get $(0\times Z_2,\bb A^1\times W_2,t_0,{\underline \phi^{(2)}};g_2)$.

Let $n>0$ be an {\it even\/} integer and let $\tau\in \Sigma_{n+1}$ be the even permutation
$(n+1,1,2,...,n)$. Let $h_1$ denote the associated $\bb A^1$-homotopy from Definition~\ref{d:homotopy_2}
between the automorphisms $\phi_{id}$ and $\phi_{\tau}$ of $Fr(-,Y_1)$.
Then $h_1\circ inc^1_{n}$ is an $\bb A^1$-homotopy between
$inc^1_{n}$ and $\phi_{\tau}\circ inc^1_{n}=\rho_1|_{Fr_n(-,Y_1)}: Fr_n(-,Y_1)\to Fr(-,Y_1)$.
Let $h_2$ be the associated $\bb A^1$-homotopy from Definition \ref{d:homotopy_2} between the automorphisms
$\phi_{id}$ and $\phi_{\tau}$ of $Fr(-,Y_2)$. Then
$h_2\circ inc^2_{n}$ is an $\bb A^1$-homotopy between
$inc^2_{n}$ and $\phi_{\tau}\circ inc^2_{n}=\rho^0_2|_{Fr_n(-,Y_2)}: Fr_n(-,Y_2)\to Fr(-,Y_2)$.
Thus $(\alpha\beta)|_{Fr_{2n}(-,Y_1)\times Fr_{2n}(-,Y_2)}$ is $\bb A^1$-homotopic to the inclusion
$inc^1_{2n}\times inc^2_{2n}$.

Next, let us focus on the morphism $\beta\alpha$. Since $Y=Y_1\sqcup Y_2$ then every
framed correspondence of level $n$ from $X$ to $Y$ is of the form
$a=(Z_1\sqcup Z_2, W_1\sqcup W_2, {\underline \phi^{(1)}}\sqcup {\underline \phi^{(2)}};g_1\sqcup g_2)$.
One has,
   $$\beta\alpha(a)=(0\times Z_1\sqcup 1\times Z_2,\bb A^1\times W_1\sqcup \bb A^1\times W_2),
       (t_0,{\underline \phi^{(1)}})\sqcup (t_0-1,{\underline \phi^{(2)}});g_1\sqcup g_2).$$
Firstly, the morphism $\beta\alpha$ is $\bb A^1$-homotopic to the morphism $\rho^0: Fr(-,Y)\to Fr(-,Y)$ taking a
framed correspondence $(Z,W,{\underline \phi;g})$ to the framed correspondence
$(0\times Z,\bb A^1\times W,t_0,{\underline \phi};g_2)$ of level $n+1$.
To see this, send a framed correspondence
$a=(Z_1\sqcup Z_2, W_1\sqcup W_2, {\underline \phi^{(1)}}\sqcup {\underline \phi^{(2)}};g_1\sqcup g_2)$ of level $n$
to the framed correspondence of level $n+1$
   $$(\bb A^1\times Z_1\sqcup \Delta\times Z_2,\bb A^1\times \bb A^1\times (W_1\sqcup W_2),
       (t_0,{\underline \phi^{(1)}})\sqcup ((t_0-\lambda),{\underline \phi^{(2)}});g_1\sqcup g_2).$$
Evaluating the latter framed correspondence of level $n+1$ at $\lambda=1$, we get $\beta\alpha(a)$.
Evaluating the same framed correspondence $\lambda=0$, we get
$\rho^0(a)$. Furthermore, using the associated homotopy from Definition~\ref{d:homotopy_2},
we see that $\rho^0|_{Fr_{2n}(-,Y)}$ is $\bb A^1$-homotopic to the inclusion $in_{2n}$.
Hence $\beta\alpha|_{Fr_{2n}(-,Y)}$ is $\bb A^1$-homotopic to the inclusion
$in_{2n}$, as was to be proved.
\end{proof}

Now Additivity Theorem~\ref{additivity} together with the Segal
machine~\cite{S} imply the following

\begin{thm}\label{addconseq}
Let $Y\in\Delta^{\op} SmOp(Fr_0(k))$. Then the $\Gamma$-space
$K\in\Gamma^{\op}\mapsto C_*Fr(-,Y\otimes K)$ is sectionwise
special. As a consequence, the framed motive $M_{fr}(Y)$ of $Y$ is
sectionwise a positively fibrant $\Omega$-spectrum, which is sectionwise
(respectively locally in the Nisnevich topology) an $\Omega$-spectrum
whenever the motivic space $C_*Fr(-,Y)$ is sectionwise
(respectively locally in the Nisnevich topology) connected.
\end{thm}

\begin{rem}\label{exinfty}{\rm
Whenever we say that $M_{fr}(Y)$ is a (positively) fibrant
$\Omega$-spectrum we tacitly assume that each of its spaces
$C_*Fr(-,Y\otimes S^n)$ is replaced with $Ex^\infty(C_*Fr(-,Y\otimes
S^n))$, where $Ex^\infty$ refers to Kan's complex. The spaces
$Ex^\infty(C_*Fr(-,Y\otimes S^n))$ are then spaces with framed
correspondences and sectionwise fibrant simplicial sets. A detailed
description of the spaces will be given in
Section~\ref{mgfrfunctor}. We can equally take any naive sectionwise
fibrant resolution functor in the category of spaces with framed
correspondences (which exist by using standard arguments) in place
of $Ex^\infty$.

}\end{rem}

It is worth to mention that the latter theorem is a kind of the
``Cancellation Theorem for framed motives in the $S^1$-direction"
(whose meaning will become clear in the proof of
Theorem~\ref{Segal_Thm_II}(2)). One should also stress  the motivic
spaces $C_*Fr(-,X\times T^n)_f$ from Theorem~\ref{Segal_Thm_II} are
zero spaces of sheaves of $S^1$-spectra $M_{fr}(X\times T^n)_f$
(these are level local Nisnevich replacements of the framed motives
$M_{fr}(X\times T^n)$). So each space $C_*Fr(-,X\times T^n)_f$ is
part of the $S^1$-spectrum $M_{fr}(X\times T^n)_f$. If we can prove
that each $S^1$-spectrum $M_{fr}(X\times T^n)_f$, $n>0$, is
motivically fibrant over a field of characteristic different from 2,
then each space $C_*Fr(-,X\times T^n)_f$ becomes motivically fibrant
(what is claimed in Theorem~\ref{Segal_Thm_II}).
{\it In turn, with a
little extra work the same is true for spaces $C_*Fr(-,X\times
T^n)\{1/2\}_f$ we are going to define now.
}

Suppose $Y$ is a simplicial object in $SmOp(Fr_0(k))$, then
$M_{fr}(Y)\{1/2\}$ is not a positive $\Omega$-spectrum in contrast
to $M_{fr}(Y)$. The problem is that each structure map
$C_*Fr(Y\otimes\cc S^n)\to\Omega C_*Fr(Y\otimes\cc S^{n+1})$ of the
spectrum $M_{fr}(Y)\{1/2\}$ factors as
   $$C_*Fr(Y\otimes\cc S^n)\to\Omega C_*Fr(Y\otimes\cc S^{n}\otimes S^1)\xrightarrow{2_*}\Omega C_*Fr(Y\otimes\cc S^{n+1}).$$
The left map is a weak equivalence of spaces for all $n>0$, but the
right arrow is such only ``after inverting 2" explained below.

It is important to observe that the composite map
   $$C_*Fr(Y\otimes\cc S^{n}\otimes S^1)\xrightarrow{2_*} C_*Fr(Y\otimes\cc S^{n+1})
     \xrightarrow{\iota^*}C_*Fr(Y\otimes\cc S^{n}\otimes S^1),$$
where $\iota^*$ is a weak equivalence induced by the ``last vertex
weak equivalence" $\iota:\cc S^1=\sd^k S^1\to S^1$, is such that it
is equal to the homomorphism $2\times\id_{\pi_*(C_*Fr(Y\otimes\cc
S^{n}\otimes S^1))}$ on presheaves of homotopy groups. This easily
follows from the fact that these presheaves of homotopy groups are
the same with the stable homotopy groups of their Segal
$S^1$-spectra and basic facts of $\Gamma$-spaces.

In particular, there is a canonical isomorphism of presheaves of
stable homotopy groups
   $$\pi_*(M_{fr}(Y)\{1/2\})\cong\pi_*(M_{fr}(Y))\otimes\bb Z[1/2]$$
and hence the canonical map of spectra $\mu:M_{fr}(Y)\to
M_{fr}(Y)\{1/2\}$ is a ``$2^{-1}$-stable weak equivalence" in the
sense that $\pi_*(\mu)\otimes\bb Z[1/2]$ is an isomorphism.

The inversion of the map $2_*$ above is achieved if we take the space
   $$C_*Fr(Y\otimes\cc S^n)\{1/2\}:=\colim_k\Omega^kC_*Fr(Y\otimes\cc S^{n+k}).$$
It follows from above that each map of motivic spaces
   \begin{equation*}\label{inv2}
    \tau:C_*Fr(Y\otimes\cc S^n)\to C_*Fr(Y\otimes\cc S^n)\{1/2\},\quad n>0,
   \end{equation*}
induces an isomorphism on homotopy presheaves after tensoring with
$\bb Z[1/2]$ and the homotopy presheaves of $C_*Fr(Y\otimes\cc
S^n)\{1/2\}$ are already presheaves of $\bb Z[1/2]$-modules. In turn,
if $C_*Fr(Y)$ is locally connected, say if $Y= X\times T^n$ with $X\in Sm/k$ and $n>0$
(see~\cite[A.1]{GPN}), then $\tau:C_*Fr(Y)\to C_*Fr(Y)\{1/2\}$
induces isomorphisms on sheaves of homotopy groups with $\bb Z[1/2]$-coefficients
$\pi_*^{\nis}(\tau)\otimes\bb Z[1/2]$ (for this we use Theorem~\ref{addconseq}). Moreover,
   $$\pi_*(C_*Fr(Y\otimes\cc S^n)\{1/2\})=\pi_{*-n}(M_{fr}(Y)\{1/2\})\cong\pi_{*-n}(M_{fr}(Y))\otimes\bb Z[1/2].$$

We shall also consider a spectrum
   $$\wt{M}_{fr}(Y)\{1/2\}:=\Omega^\infty M_{fr}(Y)\{1/2\}=(C_*Fr(Y)\{1/2\},C_*Fr(Y\otimes\cc S^1)\{1/2\},\ldots).$$
Then the natural morphism of spectra
$\ell:{M}_{fr}(Y)\{1/2\}\to\wt{M}_{fr}(Y)\{1/2\}$ is a stable
equivalence by construction and $\wt{M}_{fr}(Y)\{1/2\}$ becomes an
$\Omega$-spectrum. By $\wt{M}_{fr}(Y)\{1/2\}_f$ we shall mean a
fibrant replacement of the spectrum $\wt{M}_{fr}(Y)\{1/2\}$ in the
level injective Nisnevich local model structure of spectra. Its
spaces are of the form $CFr(Y\otimes\cc S^n)\{1/2\}_f$, $n\geq 0$.
More generally, for every simplicial sheaf $\cc G$ we shall write
$\wt{\cc M}_{fr}(\cc G)\{1/2\}$ to denote the spectrum
$\Omega^\infty \wt{\cc M}_{fr}(\cc G)\{1/2\}$. Its spaces will be
denoted by $C_*\cc Fr(\cc G\otimes\cc S^n)\{1/2\}$, $n\geq 0$.

We document the above arguments as follows.

\begin{prop}\label{frac12}
Suppose $Y$ is a simplicial object in $SmOp(Fr_0(k))$, then the
presheaves of homotopy groups of $\wt{M}_{fr}(Y)\{1/2\}$ and
$C_*Fr(Y\otimes\cc S^n)\{1/2\}$, $n\geq 0$, are presheaves of $\bb
Z[1/2]$-modules and the canonical morphisms
   $$\ell\circ\mu:M_{fr}(Y)\to\wt{M}_{fr}(Y)\{1/2\},\quad\tau:C_*Fr(Y\otimes\cc S^n)\to C_*Fr(Y\otimes\cc S^n)\{1/2\},\quad n> 0,$$
induce isomorphisms on presheaves of homotopy groups with $\bb
Z[1/2]$-coefficients $\pi_*(\ell\circ\mu)\otimes\bb Z[1/2]$ and
$\pi_*(\tau)\otimes\bb Z[1/2]$ respectively. Furthermore, if
$C_*Fr(Y)$ is locally connected (e.g. if $Y= X\times T^n$ with $X\in
Sm/k$ and $n>0$), then $\tau:C_*Fr(Y)\to C_*Fr(Y)\{1/2\}$ induces
isomorphisms on sheaves of homotopy groups with $\bb
Z[1/2]$-coefficients $\pi_*^{\nis}(\tau)\otimes\bb Z[1/2]$.
\end{prop}

If we can prove that each $S^1$-spectrum $M_{fr}(X\times T^n)$ or
$\wt{M}_{fr}(X\times T^n)\{1/2\}_f$, $n>0$, is motivically fibrant,
then each space $C_*Fr(-,X\times T^n)_f$ or $C_*Fr(-,X\times
T^n)\{1/2\}_f$ becomes motivically fibrant (what is claimed in
Theorem~\ref{Segal_Thm_II}). Therefore our next goal is to
investigate this kind of fibrant motivic spaces coming from relevant
$S^1$-spectra in more detail.

\section{Fibrant motivic spaces generated by $S^1$-spectra}

Let $sShv_\bullet(Sm/k)$ denote the category of Nisnevich sheaves of
pointed simplicial sets. It has the injective model
structure~\cite{Jar1} in which cofibrations are the monomorphisms
and weak equivalences are stalkwise weak equivalences of simplicial
sets. The category of $S^1$-spectra $Sp_{S^1}(sShv_\bullet(Sm/k))$ associated with
$sShv_\bullet(Sm/k)$ will also be called the {\it category of
ordinary $S^1$-spectra of simplicial Nisnevich sheaves}. It has
level and stable model structures (the standard references here
are~\cite{H,Jar}). In this section we describe a class of motivic
spaces coming from ordinary $S^1$-spectra of simplicial Nisnevich
sheaves, which are fibrant in the motivic model category
$sShv_\bullet(Sm/k)_{\mot}$ of Morel--Voevodsky~\cite{MV}. This
class occurs in our analysis. Recall that
$sShv_\bullet(Sm/k)_{\mot}$ is obtained from $sShv_\bullet(Sm/k)$ by
Bousfield localization with respect to the projections $p:X\times\bb
A^1\to X$, $X\in Sm/k$. As above, the level/stable model category of
$S^1$-spectra associated with $sShv_\bullet(Sm/k)_{\mot}$ will also
be called the {\it level/stable injective model category of
$S^1$-spectra}.

\begin{prop}\label{prop:A1_fibrant}
Let $E$ be an $S^1$-spectrum in the category of simplicial Nisnevich
sheaves such that each space $E_n\in sShv_\bullet(Sm/k)$ of the
spectrum is fibrant in $sShv_\bullet(Sm/k)$. Suppose $E$ is
sectionwise an $\Omega$-spectrum in the category of ordinary
$S^1$-spectra of pointed simplicial sets. Suppose $E$ is locally
$(-1)$-connected in the Nisnevich topology. Finally suppose that for
any integer $n$, the Nisnevich sheaf $\pi^{\nis}_n(E)$ is strictly
homotopy invariant. Then the following statements are true:
\begin{enumerate}
\item every motivic space $E_n$ of $E$ is motivically fibrant;
\item $E$ is fibrant in the
stable injective motivic model structure of $S^1$-spectra.
\end{enumerate}
\end{prop}

\begin{proof}
(1). Since $E$ is sectionwise an $\Omega$-spectrum, then every
$E_n$, $n\geq 0$, is sectionwise fibrant. Therefore it suffices to
prove that $E_n$ is $\mathbb A^1$-local. So we have to check that
for any smooth variety $X$ the projection $p:X\times\mathbb A^1\to
X$ induces a weak equivalence of simplicial sets $p^*: E_n(X)\to
E_n(X\times \mathbb A^1)$. Since $E$ is sectionwise an
$\Omega$-spectrum it suffices to check that the pull-back map
$p^*:E(X)\to E(X\times \mathbb A^1)$ is a stable equivalence of
ordinary $S^1$-spectra. So it is sufficient to verify that for any
integer $r$ the map $p^*: \pi_r(E(X)) \to \pi_r(E(X\times \mathbb
A^1))$ is an isomorphism. Consider two convergent spectral sequences
   $$H^p_{Nis}(X,\pi^{\nis}_q(E))\Rightarrow \pi_{q-p}(E(X))\quad \text{and}
   \quad H^p_{Nis}(X\times \mathbb A^1,\pi^{\nis}_q(E))\Rightarrow\pi_{q-p}(E(X\times \mathbb A^1)).$$
The projection $p$ induces a pull-back morphism between these two
spectral sequences. This morphism is an isomorphism on the second
page, because each Nisnevich sheaf $\pi^{\nis}_n(E)$ is strictly
homotopy invariant by assumption. Hence $p^*: \pi_r(E(X)) \to
\pi_r(X\times \mathbb A^1)$ is an isomorphism.
Assertion (2) easily follows from assertion (1).
\end{proof}

Below we shall need the following

\begin{lem}\label{puchk}
Let $F$ be a radditive framed presheaf of Abelian groups. Then the
associated sheaf in the Nisnevich topology has a unique structure of a
framed presheaf such that the map $F\to F_{\nis}$ is a map of framed presheaves.
\end{lem}

\begin{proof}
This is proved in~\cite[4.5]{Voe2}.
\end{proof}

\begin{rem}{\rm
We should stress that~\cite[4.5]{Voe2} used in the proof of the preceding lemma
is not true unless $F$ is radditive.
}\end{rem}

\begin{cor}\label{cor:A1_fibrant}
Let $k$ be an infinite perfect field and $E$ be an $S^1$-spectrum of simplicial Nisnevich sheaves with
framed correspondences. Suppose $E$ is locally an $\Omega$-spectrum
in the Nisnevich topology. Suppose it is $(-1)$-connected locally in
the Nisnevich topology. Finally suppose that if $\charr k\ne 2$ (respectively $\charr k=2$)  then for any integer $n$,
the Nisnevich presheaf $\pi_n(E)$ is homotopy invariant,
quasi-stable and radditive (respectively a homotopy invariant,
quasi-stable and radditive presheaf of $\bb Z[1/2]$-modules). Let $E\to E^f$ be a fibrant replacement
of $E$ in the level injective model structure of ordinary sheaves of
$S^1$-spectra. Then the following statements are true:
\begin{enumerate}
\item each motivic space $E_{n}^f$ is motivically fibrant;
\item the spectrum $E^f$ is fibrant in the
stable injective motivic model category of $S^1$-spectra.
\end{enumerate}
\end{cor}

\begin{proof}
The Nisnevich presheaf $\pi_n(E)$ is a radditive framed presheaf.
Hence the associated Nisnevich sheaf $\pi^{\nis}_n(E)$ is equipped
with a unique structure of a framed presheaf such that the canonical
morphism $\pi_n(E)\to \pi^{\nis}_n(E)$ is a morphism of framed
presheaves by Lemma~\ref{puchk}. By~\cite[1.1]{GP4} the Nisnevich sheaf $\pi^{\nis}_n(E)$
is strictly homotopy invariant, hence so is the Nisnevich sheaf
$\pi^{\nis}_n(E^f)$ of $E^f$. Our statement now follows from the
previous proposition.
\end{proof}

\begin{cor}\label{cor:A1_fibrant_2}
Let $k$ be an infinite perfect field and $Y$ be a simplicial object
in $SmOp(Fr_0(k))$. Suppose the simplicial Nisnevich sheaf
$C_*Fr(Y)$ is locally connected in the Nisnevich topology. Let
$M_{fr}(Y) \to M_{fr}(Y)_f$ and $\wt{M}_{fr}(Y)\{1/2\}
\to\wt{M}_{fr}(Y)\{1/2\}_f$ be fibrant replacements in the level
injective model structure of ordinary sheaves of $S^1$-spectra.
Then:
\begin{enumerate}
\item if $\charr k\ne 2$ (respectively of any characteristic) then
$M_{fr}(Y)_f$ (respectively $\wt{M}_{fr}(Y)\{1/2\}_f$) is fibrant in
the stable injective motivic model category of $S^1$-spectra;

\item if $\charr k\ne 2$ (respectively of any characteristic) then
for any $n\geq 0$ and any fibrant relpacement $C_*(Fr(-,Y\otimes
{S}^n))\to C_*(Fr(-,Y\otimes {S}^n))_{f}$ in $sShv_\bullet(Sm/k)$
(respectively $C_*(Fr(-,Y\otimes {\cc S}^n))\{1/2\}\to
C_*(Fr(-,Y\otimes {\cc S}^n))\{1/2\}_{f}$), the space
$C_*(Fr(-,Y\otimes {S}^n))_{f}$ (respectively $C_*(Fr(-,Y\otimes{\cc S}^n))\{1/2\}_{f}$) is fibrant in $sShv_\bullet(Sm/k)_{\mot}$.
\end{enumerate}
\end{cor}

\begin{proof}
Suppose $\charr k\ne 2$. The zeroth space $C_*Fr(Y)$ of the framed
spectrum $M_{fr}(Y)$ is locally connected. Hence the framed spectrum
$M_{fr}(Y)$ is locally an $\Omega$-spectrum by the Segal
machine~\cite{S} and Theorem~\ref{addconseq}. The presheaves
$\pi_n(M_{fr}(Y))$ are homotopy invariant, quasi-stable and
radditive framed presheaves. Corollary~\ref{cor:A1_fibrant} implies
assertion (1). To prove the second one, note that any two fibrant
replacements of $C_*Fr(Y)$ in $sShv_\bullet(Sm/k)$ are sectionwise
weakly equivalent. Hence Corollary~\ref{cor:A1_fibrant}(1) implies
the second assertion. The statement for $\wt{M}_{fr}(Y)\{1/2\}_f$
and $C_*(Fr(-,Y\otimes {\cc S}^n))\{1/2\}_{f}$ is proved in a
similar fashion if we use the fact that the presheaves of (stable)
homotopy groups of $\wt{M}_{fr}(Y)\{1/2\}$ and $C_*(Fr(-,Y\otimes
{\cc S}^n))\{1/2\}$ are homotopy invariant, quasi-stable and
radditive framed presheaves of $\bb Z[1/2]$-modules (see Proposition~\ref{frac12} as well).
\end{proof}

Under the notation of the preceding corollary we can now prove the
following

\begin{cor}\label{cor:A1_fibrant_3}
Let $k$ be an infinite perfect field. If $\charr k\ne 2$ then the
following statements are true:

\begin{enumerate}
\item For any integer $n> 0$, the $S^1$-spectrum $M_{fr}(X\times T^n)_f$ is motivically fibrant
and the motivic space $C_*Fr(X\times T^n)_f$ is motivically fibrant.

\item For any integer $n\geq 0$, the $S^1$-spectrum $M_{fr}(X\times
T^n\times \bb G_m^{\wedge 1}\otimes { S}^1)_f$ is motivically fibrant and the
motivic space $C_*(Fr(X\times T^n\times \bb G_m^{\wedge 1}\otimes {S}^1))_f$ is
motivically fibrant.

\item For any integer $n\geq 0$, the $S^1$-spectrum $M_{fr}(X\times
T^n\times (\A^1//\bb G_m))_f$ is motivically fibrant and the motivic
space $C_*Fr(X\times T^n\times (\A^1//\bb G_m))_f$ is motivically
fibrant.
\end{enumerate}

In turn, in the case of any characteristic of $k$ the following
statements are true:

\begin{enumerate}
\item For any integer $n> 0$, the $S^1$-spectrum $\wt{M}_{fr}(X\times T^n)\{1/2\}_f$ is motivically fibrant
and the motivic space $C_*Fr(X\times T^n)\{1/2\}_f$ is motivically
fibrant.

\item For any integer $n\geq 0$, the $S^1$-spectrum $\wt{M}_{fr}(X\times
T^n\times \bb G_m^{\wedge 1}\otimes { S}^1)\{1/2\}_f$ is motivically
fibrant and the motivic space $C_*(Fr(X\times T^n\times \bb
G_m^{\wedge 1}\otimes {S}^1))\{1/2\}_f$ is motivically fibrant.

\item For any integer $n\geq 0$, the $S^1$-spectrum $\wt{M}_{fr}(X\times
T^n\times (\A^1//\bb G_m))\{1/2\}_f$ is motivically fibrant and the
motivic space $C_*Fr(X\times T^n\times (\A^1//\bb G_m))\{1/2\}_f$ is
motivically fibrant.
\end{enumerate}
\end{cor}

\begin{proof}
Suppose $\charr k\ne 2$. By~\cite[A.1]{GPN} the spaces
$C_*Fr(X\times T^n)$, $C_*Fr(X\times T^n\times (\A^1//\bb G_m))$ of
the corollary are locally connected in the Nisnevich topology. The
space $C_*(Fr(X\times T^n\times \bb G_m^{\wedge 1}\otimes {S}^1))$
is, moreover, sectionwise connected. Now our assertions follow from
Corollary~\ref{cor:A1_fibrant_2}. In the case of any characteristic
of $k$ the second part of our statement is proved in a similar fashion.
\end{proof}

We should stress that the previous corollary is of great utility in
the proof of Theorem~\ref{Segal_Thm_II}.

\section{Comparing framed motives}\label{compmotives}

One of the key properties of framed motives is that they convert
motivic equivalences between certain motivic spaces to Nisnevich
local weak equivalences. Some of such motivic equivalences are
discussed in this section. Its main result, Theorem~\ref{compframes}, is an essential step in
proving Theorem~\ref{Segal_Thm_II}. We start with preparations.

Every category $\cc A$ with coproducts and zero object 0 has a natural action of finite pointed sets.
For example, $\cc A=Fr_0(k)$ or, more generally, $\cc A=SmOp(Fr_0(k))$.
Precisely, if $A\in\cc A$ and $(K,*)$ is a finite pointes set, then we set $A\otimes K:=A\sqcup\ldots\sqcup A$,
where the coproduct is taken over non-base elements of $K$. Clearly, $A\otimes K$ is functorial
in $A$ and $K$. Note that $A\otimes *=0$ and $0\otimes K=0$.

This action is extended to an action of finite pointed simplicial sets on the category $\Delta^{\op}\cc A$ of
simplicial objects in $\cc A$. Let $(I,1)$ denote the pointed simplicial set $\Delta[1]$ with
basepoint 1. The {\it cone\/} of $A\in\cc A$ is the simplicial object
$A\otimes I$ in $\cc A$. There is a natural morphism $i_0:A\to
A\otimes I$ in $\Delta^\op\cc A$. Given a morphism $f:A\to B$ in
$\cc A$, denote by $B//_f A$ a simplicial object in $\cc A$ which is
obtained from the pushout in $\Delta^\op\cc A$ of the diagram
   $$B\xleftarrow f A\bl{i_0}\hookrightarrow A\otimes I$$
We can think of $B//_f A$ as a cone of $f$. In practice if $A$ is a subobject of $B$,
we shall also write $B// A$ to denote the simplicial object $B//_\iota A$ in
$\cc A$ with $\iota:A\to B$ the inclusion.
We have a sequence of simplicial objects in $\cc A$
   $$A\xrightarrow f B\to B//_f A\to A\otimes S^1.$$
In practice, this sequence is a typical ``triangle" of an associated triangulateed category
(see, e.g., the proof of Theorem~\ref{compframes}).

\begin{notn}\label{edem}{\rm
(1) In the particular example when $\cc A=Fr_0(k)$ and $(X,x)$ is a pointed
smooth variety, we shall write $X^{\wedge 1}$ to denote the cone
$X//x$ of the inclusion $x\hookrightarrow X$. The most common example
is $\bb G_m^{\wedge 1}$ given by the pointed scheme $(\bb G_m,1)$.
Regarding $\Delta^{\op}Fr_0(k)$ as a full subcategory of the symmetric
monoidal category $\Delta^{\op}SmOp(Fr_0(k))$, we can take the $n$th monoidal power
of $X//x$ for every $n>0$, which we shall denote by $X^{\wedge n}$.
The most common example is $\bb G_m^{\wedge n}$.

(2) If $X$ is an open subset of $Y\in Sm/k$, we shall denote by $(Y//X)^{\wedge n}$
the $n$th monoidal power of $Y//X\in\Delta^{\op}Fr_0(k)$. The most common
example will be $(\bb A^1//\bb G_m)^{\wedge n}$.

}\end{notn}

If  $\cc A=SmOp(Fr_0(k))$ then the symmetric monoidal product on $SmOp(Fr_0(k))$
defined above gives rise to a natural pairing
   $$SmOp(Fr_0(k))\times\Delta^{\op}Fr_0(k)\to\Delta^{\op}SmOp(Fr_0(k)).$$
Composing it with the framed motive functor, we get a functor
   $$M_{fr}:SmOp(Fr_0(k))\times\Delta^{\op}Fr_0(k)\to Sp_{S^1}(sShv_\bullet(Sm/k)).$$
Taking pairings of $(X\times\bb A^n,X\times\bb A^n-X\times0)\in SmOp(Fr_0(k))$
with $\A^1//\bb G_m\in\Delta^{\op}Fr_0(k)$ and $\bb G_m^{\wedge 1}\otimes {S}^1\in\Delta^{\op}Fr_0(k)$,
we get the framed motives $M_{fr}(X\times T^n\times (\A^1//\bb G_m))$
and $M_{fr}(X\times T^{n}\times\bb G_m^{\wedge 1}\otimes {S}^1)$ respectively.

Consider a commutative diagram in $\Delta^{\op}Fr_0(k)$
   \begin{equation}\label{ska-vs-dinamo}
    \xymatrix{\bb G_m\ar[r]\ar[d]&\bb A^1\ar[r]\ar[d]&\bb A^1//\bb G_m\ar[d]^{\alpha}\\
               \bb G_m^{\wedge 1}\ar[r]\ar@{=}[d]&\bb A^{\wedge 1}\ar[r]\ar[d]
               &\bb A^{\wedge 1}//\bb G_m^{\wedge 1}\ar[d]^\beta\\
               \bb G_m^{\wedge 1}\ar[r]&\emptyset\ar[r]&\bb G_m^{\wedge 1}\otimes {S}^1}
   \end{equation}
It induces a morphism of framed motives
   $$\beta_*\alpha_*:M_{fr}(X\times T^n\times (\A^1//\bb G_m))\to M_{fr}(X\times T^{n}\times\bb G_m^{\wedge 1}\otimes {S}^1),\quad n\geq 0,$$
and
   $$\beta_*\{1/2\}\alpha_*\{1/2\}:M_{fr}(X\times T^n\times (\A^1//\bb G_m))\{1/2\}\to
       M_{fr}(X\times T^{n}\times\bb G_m^{\wedge 1}\otimes {S}^1)\{1/2\},\quad n\geq 0.$$
The main result of this section is as follows.

\begin{thm}\label{compframes}
The morphism $\beta_*\alpha_*$ (respectively $\beta_*\{1/2\}\alpha_*\{1/2\}$) is a stable Nisnevich local weak
equivalence of $S^1$-spectra whenever the base field $k$ is of characteristic different from 2 (respectively of any characteristic).
\end{thm}

We postpone the proof of the theorem. It requires the language of
``linear framed motives".

\begin{defs}\label{stab}{\rm
Let $X$ and $Y$ be smooth schemes. Denote by
\begin{itemize}
\item[$\diamond$]
${\bb Z}Fr_n(X,Y):=\widetilde{\mathbb{Z}}[Fr_n(X,Y)]=\mathbb{Z}[Fr_n(X,Y)]/\mathbb{Z}\cdot
0_n$, i.e the free abelian group generated by the set $Fr_n(X,Y)$
modulo $\mathbb{Z}\cdot 0_n$;

\item[$\diamond$]
${\bb Z}F_n(X,Y):={\bb Z}Fr_n(X,Y)/A$, where $A$ is a subgroup
generated by the elementts
\begin{multline*}
(Z\sqcup Z', U,(\varphi_1,\varphi_2,\dots,\varphi_n),g) - \\
-(Z, U\setminus
Z',(\varphi_1,\varphi_2,\dots,\varphi_n)|_{U\setminus
Z'},g|_{U\setminus Z'}) - (Z',{U\setminus
Z},(\varphi_1,\varphi_2,\dots,\varphi_n)|_{U\setminus
Z},g|_{U\setminus Z}).
\end{multline*}
\end{itemize}
We shall also refer to the latter relation as the {\it additivity
property for supports}. In other words, it says that a framed
correspondence in ${\bb Z}F_n(X,Y)$ whose support is a disjoint
union $Z\sqcup Z'$ equals the sum of the framed correspondences with
supports $Z$ and $Z'$ respectively. Note that ${\bb Z}F_n(X,Y)$ is
$\mathbb{Z}[Fr_n(X,Y)]$ modulo the subgroup generated by the
elements as above, because $0_n=0_n+0_n$ in this quotient group,
hence $0_n$ equals zero. Indeed, it is enough to observe that the
support of $0_n$ equals $\emptyset\sqcup\emptyset$ and then apply
the above relation to this support.

The elements of ${\bb Z}F_n(X,Y)$ are called {\it linear framed
correspondences of level $n$} or just {\it linear framed
correspondences}. It is worth to mention that ${\bb Z}F_n(X,Y)$
is a free abelian group generated by
the elements of $Fr_n(X,Y)$ with connected support.

Denote by ${\bb Z}F_*(k)$ an additive category whose objects are
those of $Sm/k$ with $\text{Hom}$-groups defined as
   $$\text{Hom}_{{\bb Z}F_*(k)}(X,Y)=\bigoplus_{n\geq 0}{\bb Z}F_n(X,Y).$$
The composition is induced by the composition in the category
$Fr_*(k)$.

There is a functor $Sm/k \to {\bb Z}F_*(k)$ which is the identity on
objects and which takes a regular morphism $f: X\to Y$ to the linear
framed correspondence $1\cdot(X,X\times \A^0,pr_{\A^0},f\circ pr_X)
\in {\bb Z}F_0(k)$.

}\end{defs}

\begin{defs}\label{d:boxpairing_linear}{\rm
Let $X,Y,S$ and $T$ be schemes. The external product from
Definition~\ref{d:boxpairing} induces a unique external product
\[
{\bb Z}F_n(X,Y)\times{\bb Z}F_m(S,T) \xrightarrow{-\boxtimes -} {\bb
Z}F_{n+m}(X\times S, Y\times T)
\]
such that for any elements $a \in Fr_n(X,Y)$ and $b\in Fr_m(S,T)$
one has $1\cdot a\boxtimes 1\cdot b=1\cdot(a\boxtimes b)\in {\bb
Z}F_{n+m}(X\times S, Y\times T)$.

}\end{defs}

For the constant morphism $c\colon \bb A^1\to pt$, we set
\[\Sigma:=-\boxtimes 1\cdot(\{0\},\A^1,t,c)\colon{\bb Z}F_n(X,Y)\to {\bb Z}F_{n+1}(X,Y)\]
and refer to it as the \textit{suspension}.

\begin{defs}\label{d:stable_linear_fr_corr}{\rm
For any $k$-smooth variety $Y$ there is a presheaf ${\bb Z}F_*(-,Y)$
on the category ${\bb Z}F_*(k)$ represented by $Y$. We also have a
${\bb Z}F_*(k)$-presheaf
   $${\bb Z}F(-,Y):=\colim({\bb Z}F_0(-,Y)\xrightarrow{\Sigma}{\bb Z}F_1(-,Y)
     \xrightarrow{\Sigma}\cdots\xrightarrow{\Sigma}{\bb Z}F_n(-,Y)\xrightarrow{\Sigma}\cdots).$$
For a $k$-smooth variety $X$, the elements of ${\bb Z}F(X,Y)$ are
also called {\it stable linear framed correspondences}. Stable
linear framed correspondences {\it do not form} morphisms of a
category.}
\end{defs}

\begin{rem}{\rm
For any $X,Y$ in $Sm/k$ one has ${\bb Z}F_*(-,X\sqcup Y)={\bb
Z}F_*(-,X)\oplus{\bb Z}F_*(-,Y)$ and ${\bb Z}F(-,X\sqcup Y)={\bb
Z}F(-,X)\oplus{\bb Z}F(-,Y)$.

}\end{rem}

For every $(Y,Y-S)\in SmOp(Fr_0(k))$, ${\bb Z}F_*(k)$-presheaves
${\bb Z}F_*(-,Y/Y-S)$ and ${\bb Z}F(-,Y/Y-S)$ are defined in a similar fashion.
Namely, each ${\bb Z}F_n(X,Y/Y-S)$ is the free abelian group generated by
the elements of $Fr_n(X,Y/Y-S)$ with connected support. Then we set
${\bb Z}F_*(X,Y/Y-S)=\oplus_{n\geq 0}{\bb Z}F_n(X,Y/Y-S)$ and ${\bb Z}F(X,Y/Y-S)$
is obtained from ${\bb Z}F_*(X,Y/Y-S)$ by stabilization in the $\Sigma$-direction.

For every $Y\in\Delta^{\op}SmOp(Fr_0(k))$ there is a $\Gamma$-space
   $$(K,*)\in\Gamma^{\op}\mapsto {\bb Z}F(-,(Y/Y-S)\otimes K).$$

\begin{defs}\label{lfrmotive}{\rm
The {\it linear framed motive $LM_{fr}(Y)$ of $Y\in\Delta^{\op}
SmOp(Fr_0(k))$\/} is the Segal $S^1$-spectrum $(C_*\bb ZF(-,Y),C_*\bb ZF(-,Y\otimes
S^1),C_*\bb ZF(-,Y\otimes S^2),\ldots)$ of spaces in $sShv_\bullet(Sm/k)$ associated
with the $\Gamma$-space $K\in\Gamma^{\op}\mapsto C_*\bb ZF(-,Y\otimes
K)=\bb ZF(\Delta^\bullet\times-,Y\otimes K)$.

The {\it linear framed motive $LM_{fr}(Y)\{1/2\}$ of $Y\in\Delta^{\op}
SmOp(Fr_0(k))$ with $1/2$-coefficients\/} is the $S^1$-spectrum $(C_*\bb ZF(-,Y),C_*\bb ZF(-,Y\otimes\cc
S^1),C_*\bb ZF(-,Y\otimes\cc S^2),\ldots)$ of motivic spaces in $sShv_\bullet(Sm/k)$. Equivalently,
it is the value of the $\Gamma$-space $K\in\Gamma^{\op}\mapsto C_*\bb ZF(-,Y\otimes
K)$ at the spectrum $\bb S[1/2]$.

}\end{defs}

Note that $LM_{fr}(Y)$ is the Eilenberg--Mac~Lane spectrum associated with the
complex of Nisnevich sheaves $C_*\bb ZF(-,Y)$ (we often identify simplicial abelian groups
with their normalized complexes by the Dold--Kan correspondence). Therefore $\pi_*(LM_{fr}(Y))=
H_*(C_*\bb ZF(-,Y))$.

The following lemma is proved similar to Proposition~\ref{frac12}.

\begin{lem}\label{fraclin12}
Suppose $Y$ is a simplicial object in $SmOp(Fr_0(k))$, then the presheaves of homotopy
groups of ${LM}_{fr}(Y)\{1/2\}$ are presheaves of $\bb Z[1/2]$-modules and the canonical morphism
   $$\mu:LM_{fr}(Y)\to{LM}_{fr}(Y)\{1/2\}$$
induces isomorphisms on presheaves of homotopy groups with $\bb Z[1/2]$-coefficients
$\pi_*(\mu)\otimes\bb Z[1/2]$.
\end{lem}

Denote by $LM_{fr}(Y)[1/2]$ the $S^1$-spectrum
   $$(C_*\bb ZF(-,Y)\otimes\bb Z[1/2],C_*\bb ZF(-,Y\otimes
        S^1)\otimes\bb Z[1/2],C_*\bb ZF(-,Y\otimes S^2)\otimes\bb Z[1/2],\ldots).$$
It is the Segal $S^1$-spectrum associated with the $\Gamma$-space $K\in\Gamma^{\op}\mapsto C_*\bb ZF(-,Y\otimes
K)\otimes\bb Z[1/2]:=\bb ZF(\Delta^\bullet\times-,Y\otimes K)\otimes\bb Z[1/2]$. By the very construction,
$\pi_*(LM_{fr}(Y)[1/2])=\pi_*(LM_{fr}(Y))\otimes\bb Z[1/2]$.
We also denote by $\wt{LM}_{fr}(Y)\{1/2\}:=\Omega^{\infty}{LM}_{fr}(Y)\{1/2\}$. The canonical morphism
$\zeta:{LM}_{fr}(Y)\{1/2\}\to\wt{LM}_{fr}(Y)\{1/2\}$ is a stable equivalence. Notice that ${LM}_{fr}(Y)[1/2],\wt{LM}_{fr}(Y)\{1/2\}$
are sectionwise $\Omega$-spectra.

\begin{cor}\label{fraclincor12}
The composite map $LM_{fr}(Y)\bl\mu\to{LM}_{fr}(Y)\{1/2\}\bl\zeta\to\wt{LM}_{fr}(Y)\{1/2\}$ factors as
   $$LM_{fr}(Y)\bl\rho\to{LM}_{fr}(Y)[1/2]\bl\xi\to\wt{LM}_{fr}(Y)\{1/2\}$$
and $\xi$ is a level weak equivalence of spectra.
\end{cor}

\begin{proof}[Proof of Theorem~\ref{compframes}]
Suppose $\charr k\ne 2$. Since the spectra $M_{fr}(X\times T^n\times (\A^1//\bb G_m))$ and
$M_{fr}(X\times T^{n}\times\bb G_m^{\wedge 1}\otimes {S}^1)$
are sectionwise connected, $\beta_*\alpha_*$ is a stable Nisnevich local equivalence of spectra
if and only if so is the induced map on homology
   $$\beta_*\alpha_*:\bb ZM_{fr}(X\times T^n\times (\A^1//\bb G_m))\to\bb ZM_{fr}(X\times T^{n}\times\bb G_m^{\wedge 1}\otimes {S}^1).$$
Here both linear spectra are defined by taking free Abelian groups of every entry of
$M_{fr}(X\times T^n\times (\A^1//\bb G_m))$ and $M_{fr}(X\times T^{n}\times\bb G_m^{\wedge 1}\otimes {S}^1)$
respectively. By~\cite[1.2]{GPN} the latter arrow is schemewise stably equivalent to the map
of linear framed motives
   $$\beta_*\alpha_*:LM_{fr}(X\times T^n\times (\A^1//\bb G_m))\to LM_{fr}(X\times T^{n}\times\bb G_m^{\wedge 1}\otimes {S}^1).$$
We see that the map of the theorem is a stable Nisnevich local equivalence of spectra
if and only if the morphism of complexes of Nisnevich sheaves
   $$\beta_*\alpha_*:C_*\bb ZF(-,X\times T^n\times (\A^1//\bb G_m))\to C_*\bb ZF(-,X\times T^{n}\times\bb G_m^{\wedge 1}\otimes {S}^1)$$
is a quasi-isomorphism.

Since $C_*\bb ZF(-,Y_1\sqcup Y_2)=C_*\bb ZF(-,Y_1)\oplus C_*\bb ZF(-,Y_2)$ for any $Y_1,Y_2\in\Delta^{\op}SmOp(Fr_0(k))$,
the digram~\eqref{ska-vs-dinamo} induces a commutative diagram of triangles of complexes of Nisnevich sheaves\footnotesize
   \begin{equation}\label{induceddiagr}
    \xymatrix{C_*\bb ZF(-,X\times T^n\times\bb G_m)\ar[r]\ar[d]&C_*\bb ZF(-,X\times T^n\times\bb A^1)\ar[r]\ar[d]&C_*\bb ZF(-,X\times T^n\times(\bb A^1//\bb G_m))\ar[d]^{\alpha}\ar[r]^(.8)+&{}\\
               C_*\bb ZF(-,X\times T^n\times\bb G_m^{\wedge 1})\ar[r]\ar@{=}[d]&C_*\bb ZF(-,X\times T^n\times\bb A^{\wedge 1})\ar[r]\ar[d]
               &C_*\bb ZF(-,X\times T^n\times(\bb A^{\wedge 1}//\bb G_m^{\wedge 1}))\ar[d]^\beta\ar[r]^(.8)+&{}\\
               C_*\bb ZF(-,X\times T^n\times\bb G_m^{\wedge 1})\ar[r]&0\ar[r]&C_*\bb ZF(-,X\times T^n\times\bb G_m^{\wedge 1}\otimes {S}^1)\ar[r]^(.8)+&{}}
   \end{equation}
\normalsize Firstly, we claim that $\alpha$ is a schemewise quasi-isomorphism of complexes of presheaves.
Indeed, we have a map of two triangles of complexes of presheaves\footnotesize
    $$\xymatrix{C_*\bb ZF(-,X\times T^n\times\bb G_m)\ar[r]\ar[d]&C_*\bb ZF(-,X\times T^n\times\bb G_m^{\wedge 1})\ar[r]\ar[d]
               &C_*\bb ZF(-,X\times T^n\times(pt\otimes S^1))\ar@{=}[d]\ar[r]^(.8)+&{}\\
               C_*\bb ZF(-,X\times T^n\times\bb A^1)\ar[r]&C_*\bb ZF(-,X\times T^n\times\bb A^{\wedge 1})\ar[r]
               &C_*\bb ZF(-,X\times T^n\times(pt\otimes S^1))\ar[r]^(.8)+&{}}$$
\normalsize We see that the left square of the diagram is Mayer--Vietoris, hence
$\alpha$ is a schemewise quasi-isomorphisms of complexes of presheaves.

Secondly, we claim that $\beta$ is a local quasi-isomorphism of complexes of Nisnevich sheaves.
This is equivalent to showing that the complex $C_*\bb ZF(-,X\times T^n\times\bb A^{\wedge 1})$
of diagram~\eqref{induceddiagr} is locally quasi-isomorphic to zero. To prove the latter,
consider a map of two triangles of complexes of sheaves\footnotesize
    $$\xymatrix{C_*\bb ZF(-,X\times T^n\times pt)\ar[r]\ar@{=}[d]&C_*\bb ZF(-,X\times T^n\times\bb A^{1})\ar[r]\ar@{=}[d]
               &C_*\bb ZF(-,X\times T^n\times\bb A^{\wedge 1})\ar[d]\ar[r]^(.8)+&{}\\
               C_*\bb ZF(-,X\times T^n\times pt)\ar[r]&C_*\bb ZF(-,X\times T^n\times\bb A^{1})\ar[r]
               &C_*\bb ZF(-,X\times T^n\times\bb A^1)/C_*\bb ZF(-,X\times T^n\times pt)\ar[r]^(.85)+&{}}$$
\normalsize The cohomology sheaves of the lower right complex are homotopy invariant and
quasi-stable framed presheaves. By~\cite[1.1]{GP4} these cohomology sheaves are strictly homotopy invariant. The
terms of this complex are contractible sheaves. Now the proof of~\cite[1.10.2]{SV1} yields
the local acyclicity of the complex. It follows that $C_*\bb ZF(-,X\times T^n\times\bb A^{\wedge 1})$
is locally acyclic, and hence $\beta$ is locally a quasi-isomorphism.

The proof that $\beta_*\{1/2\}\alpha_*\{1/2\}$ is a stable Nisnevich local weak
equivalence repeats the above proof word for word if we notice that homology of framed motives
with $1/2$-coefficients is computed as $LM_{fr}(Y)[1/2]$ (see Lemma~\ref{fraclin12} and Corollary~\ref{fraclincor12} as well).
This completes the proof of the theorem.
\end{proof}

In fact, the proof of Theorem~\ref{compframes} also shows the following fact.

\begin{cor}\label{a1comparison}
Suppose $\charr k\ne 2$ (respectively $k$ is of any characteristic). Then for every $n\geq 0$ and $X\in Sm/k$, the natural maps
$M_{fr}(X\times T^n\times\A^1)\to M_{fr}(X\times T^n)$ and $LM_{fr}(X\times T^n\times\A^1)\to LM_{fr}(X\times T^n)$
(respectively $M_{fr}(X\times T^n\times\A^1)\{1/2\}\to M_{fr}(X\times T^n)\{1/2\}$ and $
LM_{fr}(X\times T^n\times\A^1)\{1/2\}\to LM_{fr}(X\times T^n)\{1/2\}$) are
stable local weak equivalences of $S^1$-spectra.
\end{cor}

Let us take the $n$th power $(\beta\alpha)^{\wedge n}:(\A^1//\bb G_m)^{\wedge n}\to (\bb G_m^{\wedge 1}\otimes {S}^1)^{\wedge n}$
of the morphism $\beta\alpha$ in the symmetric monoidal category $\Delta^{\op}SmOp(Fr_0(k))$. Below we shall
also need the following

\begin{cor}\label{potom1}
Suppose $\charr k\ne 2$ (respectively $k$ is of any characteristic). For every $n\geq 1$ and $X\in Sm/k$, the map
$(\beta\alpha)^{\wedge n}_*:C_*Fr(-,X\times(\A^1//\bb G_m)^{\wedge n})\to C_*Fr(-,X\times(\bb G_m^{\wedge 1}\otimes {S}^1)^{\wedge n})$
(respectively $(\beta\alpha)^{\wedge n}_*:C_*Fr(-,X\times(\A^1//\bb G_m)^{\wedge n})\{1/2\}
\to C_*Fr(-,X\times(\bb G_m^{\wedge 1}\otimes {S}^1)^{\wedge n})\{1/2\}$)
is a local Nisnevich weak equivalence of motivic spaces.
\end{cor}

\begin{proof}
Suppose $\charr k\ne 2$. The space $C_*Fr(-,X\times(\bb G_m^{\wedge 1}\otimes {S}^1)^{\wedge n})$
is plainly sectionwise connected. By~\cite[A.1]{GPN}
the space $C_*Fr(-,X\times(\A^1//\bb G_m)^{\wedge n})$ is locally connected.
Therefore $M_{fr}(X\times(\bb G_m^{\wedge 1}\otimes {S}^1)^{\wedge n})$
is sectionwise an $\Omega$-spectrum and $M_{fr}(X\times(\A^1//\bb G_m)^{\wedge n})$
is locally an $\Omega$-spectrum by Theorem~\ref{addconseq}. Therefore our assertion would
follow if we showed that the map $(\beta\alpha)^{\wedge n}_*:
M_{fr}(X\times(\A^1//\bb G_m)^{\wedge n})\to M_{fr}(X\times(\bb G_m^{\wedge 1}\otimes {S}^1)^{\wedge n})$
is a local Nisnevich weak equivalence of $S^1$-spectra. The latter follows by using induction in $n$, Theorem~\ref{compframes}
and the fact that the realization of Nisnevich local weak equivalences is a local Nisnevich weak equivalence.
If $k$ is of any characteristic then our proof is similar to the proof above.
\end{proof}

We finish the section by proving the following useful

\begin{thm}\label{mvproperty}
Suppose $\charr k\ne 2$ (respectively $k$ is of any characteristic).
For every $n\geq 0$ and every elementary Nisnevich square
   $$\xymatrix{U'\ar[r]\ar[d]&X'\ar[d]\\
                       U\ar[r]&X}$$
the square of $S^1$-spectra
   $$\xymatrix{M_{fr}(U'\times T^n)\ar[r]\ar[d]&M_{fr}(X'\times T^n)\ar[d]\\
                       M_{fr}(U\times T^n)\ar[r]&M_{fr}(X\times T^n)}$$
(repectively the square
   $$\xymatrix{M_{fr}(U'\times T^n)\{1/2\}\ar[r]\ar[d]&M_{fr}(X'\times T^n)\{1/2\}\ar[d]\\
                        M_{fr}(U\times T^n)\{1/2\}\ar[r]&M_{fr}(X\times T^n)\{1/2\}}$$
of framed motives with $1/2$-coefficients)
is homotopy cartesian locally in the Nisnevich topology. The same is also true for linear framed motives
and linear framed motives with $1/2$-coefficients.
\end{thm}

\begin{proof}
Suppose $\charr k\ne 2$. Since we deal with connected $S^1$-spectra, the proof of Theorem~\ref{compframes}
shows that our statement suffuces to verify for linear framed motives.
It follows from the proof of~\cite[4.4]{Voe2} that the sequence of presheaves
   $$0\to\bb ZF_s(-,U'\times T^n)\to\bb ZF_s(-,U\times T^n)\oplus\bb ZF_s(-,X'\times T^n)\to\bb ZF_s(-,X\times T^n)\to 0$$
is locally exact for every $s\geq 0$. Passing to the colimit over $s$, the sequence of presheaves
   $$0\to\bb ZF(-,U'\times T^n)\to\bb ZF(-,U\times T^n)\oplus\bb ZF(-,X'\times T^n)\to\bb ZF(-,X\times T^n)\to 0$$
is locally exact as well. It follows that the sequence of Eilenberg--Mac~Lane spectra
   $$EM(\bb ZF(-,U'\times T^n))\to EM(\bb ZF(-,U\times T^n))\times EM(\bb ZF(-,X'\times T^n))\to EM(\bb ZF(-,X\times T^n))$$
is locally a homotopy fibre sequence. It follows that the sequence
   \begin{equation}\label{yui}
    LM_{fr}(U'\times T^n)\to LM_{fr}(U\times T^n)\times LM_{fr}(X'\times T^n)\to LM_{fr}(X\times T^n)
   \end{equation}
is a homotopy fibre sequence in the motivic model structure of $S^1$-spectra, and hence so is the sequence
   $$LM_{fr}(U'\times T^n)_f\to LM_{fr}(U\times T^n)_f\times LM_{fr}(X'\times T^n)_f\to LM_{fr}(X\times T^n)_f,$$
where ``$f$" as in Corollary~\ref{cor:A1_fibrant}. It follows from Corollary~\ref{cor:A1_fibrant} that the latter sequence is
a sequence of fibrant objects in the stable injective motivic model structure of $S^1$-spectra.
Therefore this sequence is also locally a homotopy fibre sequence, and hence so is sequence~\eqref{yui}.
If $k$ is of any characteristic then our proof is similar to the proof above (recall that homology of
framed motives with $1/2$-coefficients is computed as linear framed motives tensored with $\bb Z[1/2]$).
\end{proof}

\section{Proof of Theorem~\ref{Segal_Thm_II}}\label{dokazatelstvo}

In this section we prove Theorem~\ref{Segal_Thm_II}. We first give a
proof for the second statement of the theorem followed by the first
statement.

\subsection{Proof of Theorem~\ref{Segal_Thm_II}(2)}

Recall that $A\in sShv_\bullet(Sm/k)$ is {\it finitely
presentable\/} if the functor $\Hom_{sShv_\bullet(Sm/k)}(A,-)$
preserves directed colimits. Using the General Framework on
p.~\pageref{genfr}, for every $A,L\in sShv_\bullet(Sm/k)$ with $A$
finitely presentable there are canonical morphisms in
$sShv_\bullet(Sm/k)$
   $$C_*\cc Fr(L)\to\underline{\Hom}(A,C_*\cc Fr(L\wedge A)),\quad
       C_*\cc Fr(L)\{1/2\}\to\underline{\Hom}(A,C_*\cc Fr(L\wedge A)\{1/2\})$$
as well as canonical morphisms of ordinary $S^1$-spectra of simplicial
Nisnevich sheaves
   \begin{equation}\label{smashell}
    \cc M_{fr}(L)\to\underline{\Hom}(A,\cc M_{fr}(L\wedge A)),\quad\wt{\cc M}_{fr}(L)\{1/2\}\to\underline{\Hom}(A,\wt{\cc M}_{fr}(L\wedge A)\{1/2\}).
   \end{equation}
They induce the corresponding morphisms of spaces and $S^1$-spectra respectively\footnotesize
   $$a_A:C_*\cc Fr(L)_f\to\underline{\Hom}(A,C_*\cc Fr(L\wedge A)_f),\quad
       a_A\{1/2\}:C_*\cc Fr(L)\{1/2\}_f\to\underline{\Hom}(A,C_*\cc Fr(L\wedge A)\{1/2\}_f)$$
\normalsize and\footnotesize
   $$\alpha_A:\cc M_{fr}(L)_f\to\underline{\Hom}(A,\cc M_{fr}(L\wedge A)_f),\quad
        \alpha_A\{1/2\}:\wt{\cc M}_{fr}(L)\{1/2\}_f\to\underline{\Hom}(A,\wt{\cc M}_{fr}(L\wedge A)\{1/2\}_f)$$
\normalsize where $C_*\cc Fr(L)_f,C_*\cc Fr(L)\{1/2\}_f$ (respectively $\cc M_{fr}(L)_f,\wt{\cc M}_{fr}(L)\{1/2\}_f$) are Nisnevich local
fibrant replacements of $C_*\cc Fr(L),C_*\cc Fr(L)\{1/2\}$ (respectively a level Nisnevich local
fibrant replacement of $\cc M_{fr}(L),\wt{\cc M}_{fr}(L)\{1/2\}$ in the category of ordinary $S^1$-spectra).

\begin{lem}\label{chempion}
Suppose $u:A\to B$ is a motivic weak equivalence in
$sShv_\bullet(Sm/k)$ between finitely presentable objects such that
the induced map $u_*:\cc M_{fr}(L\wedge A)\to\cc M_{fr}(L\wedge B)$
is a stable Nisnevich local weak equivalence of spectra. Suppose
$\cc M_{fr}(L)_f,\cc M_{fr}(L\wedge A)_f,\cc M_{fr}(L\wedge B)_f$
are all motivically fibrant $S^1$-spectra. Then $\alpha_A:\cc
M_{fr}(L)_f\to\underline{\Hom}(A,\cc M_{fr}(L\wedge A)_f)$ is a
sectionwise stable equivalence if and only if so is $\alpha_B:\cc
M_{fr}(L)_f\to\underline{\Hom}(B,\cc M_{fr}(L\wedge B)_f)$.

Suppose $u:A\to B$ is a motivic weak equivalence in $sShv_\bullet(Sm/k)$ between
finitely presentable objects such that
$u\{1/2\}_*:\wt{\cc M}_{fr}(L\wedge A)\{1/2\}\to\wt{\cc M}_{fr}(L\wedge B)\{1/2\}$ is a stable Nisnevich local
weak equivalence of spectra. Suppose
$\wt{\cc M}_{fr}(L)\{1/2\}_f,\wt{\cc M}_{fr}(L\wedge A)\{1/2\}_f,\wt{\cc M}_{fr}(L\wedge B)\{1/2\}_f$
are all motivically fibrant $S^1$-spectra. Then
$\alpha\{1/2\}_A:\wt{\cc M}_{fr}(L)\{1/2\}_f\to\underline{\Hom}(A,\wt{\cc M}_{fr}(L\wedge A)\{1/2\}_f)$
is a sectionwise stable equivalence if and only if so is
$\alpha\{1/2\}_B:\wt{\cc M}_{fr}(L)\{1/2\}_f\to\underline{\Hom}(B,\wt{\cc M}_{fr}(L\wedge B)\{1/2\}_f)$.
\end{lem}

\begin{proof}
Commutative square~\eqref{actionV} of the General Framework gives
rise to a commutative square
   $$\xymatrix{\cc M_{fr}(L)_f\ar[r]^{\alpha_A}\ar[d]_{\alpha_B}&\underline{\Hom}(A,\cc M_{fr}(L\wedge A)_f)\ar[d]^{u_*}\\
               \underline{\Hom}(B,\cc M_{fr}(L\wedge B)_f)\ar[r]^{u^*}&\underline{\Hom}(A,\cc M_{fr}(L\wedge B)_f)}$$
By assumption, $\cc M_{fr}(L\wedge A)_f,\cc M_{fr}(L\wedge B)_f$ are
motivically fibrant $S^1$-spectra, hence $u^*$ is sectionwise stable
equivalence. Since $\cc M_{fr}(L\wedge A)\to\cc M_{fr}(L\wedge B)$
is a stable Nisnevich local weak equivalence of spectra, it follows
that $\cc M_{fr}(L\wedge A)_f\to\cc M_{fr}(L\wedge B)_f$ is a
sectionwise stable weak equivalence of spectra. We see that the
right vertical arrow $u_*$ of the square is a sectionwise stable
weak equivalence of spectra. Our statement now follows from the
two-out-three property for weak equivalences (the statement for
$\alpha\{1/2\}_A,\alpha\{1/2\}_B$ repeats the above arguments word
for word).
\end{proof}

\begin{cor}\label{corchempion}
$(1)$ Under the assumptions of Lemma~\ref{chempion} the map of
spaces $a_A:C_*\cc Fr(L)_f\to\underline{\Hom}(A,C_*\cc Fr(L\wedge
A)_f)$ is a sectionwise weak equivalence if and only if so is
$a_B:C_*\cc Fr(L)_f\to\underline{\Hom}(B,C_*\cc Fr(L\wedge B)_f)$.

$(2)$ Under the assumptions of Lemma~\ref{chempion} the map of
spaces $a\{1/2\}_A:C_*\cc Fr(L)\{1/2\}_f\to\underline{\Hom}(A,C_*\cc
Fr(L\wedge A)\{1/2\}_f)$ is a sectionwise weak equivalence if and
only if so is $a\{1/2\}_B:C_*\cc
Fr(L)\{1/2\}_f\to\underline{\Hom}(B,C_*\cc Fr(L\wedge B)\{1/2\}_f)$.
\end{cor}

\begin{lem}\label{varlamov}
$(1)$ Suppose $u:A\to B$ is a motivic weak equivalence in
$sShv_\bullet(Sm/k)$ between finitely presentable objects. Suppose
$\cc M_{fr}(L)_f,\cc M_{fr}(L\wedge B)_f$ are motivically fibrant
$S^1$-spectra. Then the composite map of spaces $C_*\cc
Fr(L)_f\xrightarrow{a_B}\underline{\Hom}(B,C_*\cc Fr(L\wedge
B)_f)\xrightarrow{u^*} \underline{\Hom}(A,C_*\cc Fr(L\wedge B)_f)$
is a sectionwise weak equivalence if and only if $\alpha_B:\cc
M_{fr}(L)_f\to\underline{\Hom}(B,\cc M_{fr}(L\wedge B)_f)$ is a
sectionwise stable equivalence of spectra.

$(2)$ Suppose $u:A\to B$ is a motivic weak equivalence in
$sShv_\bullet(Sm/k)$ between finitely presentable objects. Suppose
$\wt{\cc M}_{fr}(L)\{1/2\}_f,\wt{\cc M}_{fr}(L\wedge B)\{1/2\}_f$
are motivically fibrant $S^1$-spectra. Then the composite map of
spaces
$$C_*\cc Fr(L)\{1/2\}_f\xrightarrow{a\{1/2\}_B}\underline{\Hom}(B,C_*\cc Fr(L\wedge B)\{1/2\}_f)\xrightarrow{u^*}
\underline{\Hom}(A,C_*\cc Fr(L\wedge B)\{1/2\}_f)$$
is a sectionwise weak equivalence
if and only if
$$\alpha\{1/2\}_B:\wt{\cc M}_{fr}(L)\{1/2\}_f\to\underline{\Hom}(B,\wt{\cc M}_{fr}(L\wedge B)\{1/2\}_f)$$
is a sectionwise stable equivalence of spectra.
\end{lem}

\begin{proof}
Our assumptions on spectra imply $C_*\cc Fr(L)_f,C_*\cc Fr(L\wedge B)_f$ are
motivically fibrant spa\-ces. It follows that
$C_*\cc Fr(L)_f\xrightarrow{a_B}\underline{\Hom}(B,C_*\cc Fr(L\wedge B)_f)\xrightarrow{u^*}
\underline{\Hom}(A,C_*\cc Fr(L\wedge B)_f)$ is a sectionwise weak equivalence if
and only if so is $C_*\cc Fr(L)_f\xrightarrow{a_B}\underline{\Hom}(B,C_*\cc Fr(L\wedge B)_f)$,
because $u:A\to B$ is a motivic weak equivalence (recall that all spaces
in the Morel--Voevodsy's model category $sShv_\bullet(Sm/k)_{\mot}$ are cofibrant).

Again because of our assumptions on spectra we have that
$C_*\cc Fr(L)_f\xrightarrow{a_B}\underline{\Hom}(B,C_*\cc Fr(L\wedge B)_f)$
is a sectionwise weak equivalence
if and only if $\alpha_B:\cc M_{fr}(L)_f\to\underline{\Hom}(B,\cc M_{fr}(L\wedge B)_f)$
is a sectionwise stable equivalence of spectra. The corresponding proof for maps
$a\{1/2\}_B$, $\alpha\{1/2\}_B$ is similar to that above.
\end{proof}

We are now in a position to prove the second statement of
Theorem~\ref{Segal_Thm_II}.

\begin{proof}[Proof of Theorem~\ref{Segal_Thm_II}(2)]
By Corollary~\ref{cor:A1_fibrant_3} for any integer $n> 0$,
the $S^1$-spectrum $M_{fr}(X\times T^n)_f$ is motivically fibrant
and the motivic space $C_*Fr(X\times T^n)_f$ is motivically fibrant.
Let $u:\bb P^{\wedge 1}=(\bb P^1,\infty)\to T$ be the canonical motivic weak equivalence
in $sShv_\bullet(Sm/k)$. It is also given by the framed correspondence of level one
$(\{0\},\bb A^1,t)\in Fr_1(pt,pt)$. By Lemma~\ref{varlamov} the map
   $$C_*Fr(X\times T^n)_f\to\underline{\Hom}(\bb P^{\wedge 1},C_*Fr(X\times T^{n+1})_f)$$
is a sectionwise weak equivalence if and only if
$\alpha_T:M_{fr}(X\times T^n)_f\to\underline{\Hom}(T,M_{fr}(X\times T^{n+1})_f)$
is a sectionwise stable equivalence of spectra.

Consider the zigzag of motivic weak equivalences
   $$T\xleftarrow{\sim}(-,\A^1//\bb G_m)_+\xrightarrow{\sim}(-,\bb G_m^{\wedge 1}\otimes { S}^1)_+,$$
where the right arrow is induced by $\beta\alpha$ of diagram~\eqref{ska-vs-dinamo}.
By Corollary~\ref{cor:A1_fibrant_3} for any integer $n\geq 0$, the $S^1$-spectra $M_{fr}(X\times
T^n\times \bb G_m^{\wedge 1}\otimes { S}^1)_f,M_{fr}(X\times
T^n\times (\A^1//\bb G_m))_f$ are motivically fibrant and
$C_*(Fr(X\times T^n\times \bb G_m^{\wedge 1}\otimes {S}^1))_f$,
$C_*Fr(X\times T^n\times (\A^1//\bb G_m))_f$ are motivically fibrant spaces.

By~\cite[8.1]{GPN} $M_{fr}(X\times T^n\times (\A^1//\bb G_m))\to
M_{fr}(X\times T^{n+1})$ is a stable Nisnevich local weak
equivalence of spectra. By Theorem~\ref{compframes} $M_{fr}(X\times
T^n\times (\A^1//\bb G_m))\to M_{fr}(X\times T^n\times \bb
G_m^{\wedge 1}\otimes {S}^1)$ is a stable Nisnevich local weak
equivalence of spectra. By Lemma~\ref{chempion}
$\alpha_T:M_{fr}(X\times T^n)_f\to\underline{\Hom}(T,M_{fr}(X\times T^{n+1})_f)$ is
a sectionwise stable equivalence of spectra if and only if so is the
map of spectra $\alpha_{\bb G_m^{\wedge 1}\otimes
{S}^1}:M_{fr}(X\times T^n)_f\to \underline{\Hom}((\bb G_m^{\wedge 1}\otimes
{S}^1)_+,M_{fr}(X\times T^{n}\times\bb G_m^{\wedge 1}\otimes {S}^1)_f)$.

Consider a commutative diagram
   $$\xymatrix{M_{fr}(X\times(\bb A^1//\bb G_m)^{\wedge n})_f\ar[r]^(.33){\alpha_{\bb G_m^{\wedge 1}\otimes {S}^1}}\ar[d]
                       &\underline{\Hom}((\bb G_m^{\wedge 1}\otimes {S}^1)_+,M_{fr}(X\times(\bb A^1//\bb G_m)^{\wedge n}\times\bb G_m^{\wedge 1}\otimes {S}^1)_f)\ar[d]\\
                       M_{fr}(X\times T^n)_f\ar[r]^(.3){\alpha_{\bb G_m^{\wedge 1}\otimes {S}^1}}&
                       \underline{\Hom}((\bb G_m^{\wedge 1}\otimes {S}^1)_+,M_{fr}(X\times T^{n}\times\bb G_m^{\wedge 1}\otimes {S}^1)_f)}$$
with $M_{fr}(X\times(\bb A^1//\bb G_m)^{n})_f$ a stable Nisnevich
local fibrant replacement of ordinary spectra and $(\bb A^1//\bb
G_m)^{\wedge n}$ from Notation~\ref{edem}. It follows from~\cite[1.1]{GPN}
that the left vertical arrow is a sectionwise stable weak
equivalence of spectra, hence so is the right vertical arrow. We see
that the lower arrow is a sectionwise stable weak equivalence of
spectra if and only if so is the upper one. But the upper arrow is a
sectionwise stable weak equivalence of spectra by the Cancellation
Theorem for framed motives of algebraic
varieties~\cite[Theorem~A]{AGP} and Theorem~\ref{addconseq}. The
proof for spaces $C_*Fr(X\times T^n)\{1/2\}_f$ repeats the above
proof word for word.
\end{proof}

The proof of Theorem~\ref{Segal_Thm_II}(2) and
Corollary~\ref{potom1} also implies the following

\begin{cor}\label{potom2}
For any $n\geq 1$, the map $C_*Fr(-,X\times(\A^1//\bb G_m)^{\wedge n})\to
C_*Fr(-,X\times T^n)$ (respectively $C_*Fr(-,X\times(\A^1//\bb
G_m)^{\wedge n})\{1/2\}\to C_*Fr(-,X\times T^n)\{1/2\}$) is a local Nisnevich
weak equivalence of motivic spaces whenever $\charr k\ne 2$
(respectively $k$ is of any characteristic).
\end{cor}

\subsection{Proof of Theorem~\ref{Segal_Thm_II}(1)}\label{P1_suspension_spectra}

In this section we finish the proof of Theorem~\ref{Segal_Thm_II}.
It remains to show part~(1) of the theorem. Denote by $\Spt^{\bb
P^1}(Sm/k)$ the category of $\bb P^1$-spectra with $\bb P^1$ pointed
at $\infty$. We shall work with the injective stable motivic model
structure on $\Spt^{\bb P^1}(Sm/k)$ (see~\cite{Jar} for details).
The weak equivalences in this model category will be referred to as
{\it stable equivalences}.

We define the {\it fake suspension functor\/} $\Sigma^\ell_{\bb
P^1}:\Spt^{\bb P^1}(Sm/k)\to\Spt^{\bb P^1}(Sm/k)$ by
$(\Sigma^\ell_{\bb P^1}\cc Z)_n=\cc Z_n\wedge\bb P^1$ and structure
maps
   $$(\Sigma^\ell_{\bb P^1}\cc Z_n)\wedge\bb P^1\xrightarrow{\sigma_n\wedge\bb P^1}\Sigma^\ell_{\bb P^1}\cc Z_{n+1},$$
where $\sigma_n$ is a structure map of $\cc Z$. The fake suspension
functor is left adjoint to the {\it fake loops functor\/}
$\Omega^\ell_{\bb P^1}:\Spt^{\bb P^1}(Sm/k)\to\Spt^{\bb P^1}(Sm/k)$
defined by $(\Omega^\ell_{\bb P^1}\cc Z)_n=\Omega_{\bb P^1}\cc
Z_n=\underline{\Hom}({\bb P^1},\cc Z_n)$ and structure maps adjoint
to
   $$\Omega_{\bb P^1}\cc Z_n\xrightarrow{\Omega_{\bb P^1}\wt\sigma_n}\Omega_{\bb P^1}(\Omega_{\bb P^1}\cc Z_{n+1}),$$
where $\wt\sigma_n$ is adjoint to the structure map $\sigma_n$ of
$\cc Z$.

Define the \emph{shift functors} $t:\Spt^{\bb
P^1}(Sm/k)\xrightarrow{}\Spt^{\bb P^1}(Sm/k)$ and $s:\Spt^{\bb
P^1}(Sm/k)\xrightarrow{}\Spt^{\bb P^1}(Sm/k)$ by $(s\cc Z)_{n}=\cc
Z_{n+1}$ and $(t\cc Z)_{n}=\cc Z_{n-1}$, $(t\cc Z)_{0}=pt$, with the
evident structure maps. Note that $t$ is left adjoint to $s$.

Define $\Theta:\Spt^{\bb P^1}(Sm/k)\to\Spt^{\bb P^1}(Sm/k)$ to be
the functor $s\circ\Omega^\ell_{\bb P^1}$, where $s$ is the shift
functor. Then we have a natural map $\iota_{\cc Z}:\cc Z\to\Theta\cc
Z$, and we define
   \begin{equation}\label{theta}
    \Theta^{\infty}\cc Z=\colim(\cc Z\xrightarrow{\iota_{\cc Z}}\Theta\cc Z\xrightarrow{\Theta\iota_{\cc
     Z}}\Theta^{2}\cc Z\xrightarrow{\Theta^{2}\iota_{\cc Z}}\cdots
     \xrightarrow{\Theta^{n-1}\iota_{\cc Z}}\Theta^{n}\cc
     Z\xrightarrow{\Theta^{n}\iota_{\cc Z}}\cdots).
   \end{equation}
Set $\eta_{\cc Z}:\cc Z\to\Theta^{\infty}\cc Z$ to be the obvious
natural transformation.

\begin{lem}\label{spektrell}
For every $\bb P^1$-spectrum $\cc L$ the natural map $\eta_{\cc Z}:\cc Z\to\Theta^{\infty}\cc Z$ is a
stable motivic weak equivalence.
\end{lem}

\begin{proof}
The assertion would follow from~\cite[4.11]{H} as soon as we find a weakly finitely generated
model structure on pointed simplicial presheaves $sPre_\bullet(Sm/k)$  in the
sense of~\cite{DRO} such that its model category
of $\bb P^1$-spectra is Quillen equivalent to the injective stable
motivic model structure of Jardine~\cite{Jar}. Such a model structure on $sPre_\bullet(Sm/k)$ is the flasque
motivic model structure of Isaksen~\cite{Is}. The fact that it is weakly finitely generated
follows from~\cite[3.10, 4.9, 5.1]{Is} and~\cite[2.2]{GP}.
\end{proof}

We are now in a position to prove Theorem~\ref{Segal_Thm_II}(1).

\begin{proof}[Proof of Theorem~\ref{Segal_Thm_II}(1)]
We shall construct the stable equivalence of the theorem not only for smooth algebraic
varieties but also for all pointed motivic spaces. This is of independent interest. Let $\cc X\in sPre_\bullet(Sm/k)$ be a pointed motivic space.
Consider its suspension spectrum
   $$\Sigma^\infty_{\bb P^1}\cc X=(\cc X,\cc X\wedge\bb P^1,\cc X\wedge\bb P^{\wedge 2},\ldots).$$
We set
   $$\Sigma^\infty_{\bb P^1,T}\cc X=(\cc X,\cc X\wedge T,\cc X\wedge T^2,\ldots)$$
be the $\bb P^1$-spectrum with structure maps defined by $(\cc
X\wedge T^n)\wedge\bb P^1\xrightarrow{\id\wedge\sigma}\cc X\wedge
T^{n+1}$, where $\sigma:\bb P^1\to T$ is the canonical motivic
equivalence of sheaves.

Since smash product of a motivic weak equivalence and a motivic
space is again a motivic weak equivalence, we get a level motivic
equivalence of spectra
   $$\sigma:\Sigma^\infty_{\bb P^1}\cc X\to\Sigma^\infty_{\bb P^1,T}\cc X.$$

Let us take the Nisnevich sheaf $(\cc X\wedge T^n)_{\nis}$ of each
space $\cc X\wedge T^n$ of $\Sigma^\infty_{\bb P^1,T}\cc X$. Then we
get a $\bb P^1$-spectrum $\Sigma^\infty_{\bb P^1,T}\cc X_{\nis}$ and
a level local weak equivalence of spectra
   $$\nu:\Sigma^\infty_{\bb P^1,T}\cc X\to\Sigma^\infty_{\bb P^1,T}\cc X_{\nis}.$$
By Lemma~\ref{spektrell} the natural map of spectra
   $$\eta:\Sigma^\infty_{\bb P^1,T}\cc X_{\nis}\to\Theta^\infty\Sigma^\infty_{\bb P^1,T}\cc X_{\nis}$$
is a stable equivalence.

Evaluation of the spectrum $\Sigma^\infty_{\bb P^1,T}\cc X_{\nis}$
at a smooth scheme $U$ has entries
   $$(\colim_n(\Hom_{Shv_{\bullet}(Sm/k)}(U_+\wedge\bb P^{\wedge n},\cc X_{\nis}\wedge T^n)),
     \colim_n(\Hom_{Shv_{\bullet}(Sm/k)}(U_+\wedge\bb P^{\wedge n},\cc X_{\nis}\wedge T^{n+1})),\ldots),$$
where each smash product $\cc X_{\nis}\wedge T^n$ is taken in the
category of pointed simplicial sheaves.

So the spectrum $\Theta^\infty\Sigma^\infty_{\bb P^1,T}\cc X_{\nis}$ can be written as
   $$(\cc Fr(-,\cc X_{\nis}),\cc Fr(-,\cc X_{\nis}\wedge T),
      \cc Fr(-,\cc X_{\nis}\wedge T^2),\ldots).$$
If we apply the functor $C_*$ to the spectrum
$\Theta^\infty\Sigma^\infty_{\bb P^1,T}\cc X_{\nis}$ levelwise,
we get a spectrum
   $$(C_*(\cc Fr(-,\cc X_{\nis})),
      C_*(\cc Fr(-,\cc X_{\nis}\wedge T)),C_*(\cc Fr(-,\cc X_{\nis}\wedge T^2)),\ldots).$$
By~\cite[2.3.8]{MV} the natural map of spectra
   $$\delta:\Theta^\infty\Sigma^\infty_{\bb P^1,T}\cc X_{\nis}\to C_*(\Theta^\infty\Sigma^\infty_{\bb P^1,T}\cc X_{\nis})$$
is a level motivic weak equivalence.

Taking a Nisnevich local fibrant replacement $C_*(\cc Fr(-,\cc X_{\nis}\wedge T^n))_f$ of every
motivic space $C_*(\cc Fr(-,\cc X_{\nis}\wedge T^n))$, we get a spectrum
   $$C_*(\Theta^\infty\Sigma^\infty_{\bb P^1,T}\cc X_{\nis})_f:=
     (C_*(\cc Fr(-,\cc X_{\nis}))_f,C_*(\cc Fr(-,\cc X_{\nis}\wedge T))_f,C_*(\cc Fr(-,\cc X_{\nis}\wedge T^2))_f,\ldots).$$
Denote by
   $$\rho:=\delta\circ\eta\circ\nu\circ\sigma:\Sigma^\infty_{\bb P^1}(\cc X)\to C_*(\Theta^\infty\Sigma^\infty_{\bb P^1,T}\cc X_{\nis})_f.$$
Then $\rho$ is a stable motivic weak equivalence of $\bb P^1$-spectra. Moreover, if $\cc X$
is represented by a smooth scheme $X$ (or, more generally, by a directed colimit of
simplicial smooth schemes from $\Delta^{\op}Fr_0(k)$), then
$C_*(\Theta^\infty\Sigma^\infty_{\bb P^1,T}\cc X_{\nis})_f$ is naturally isomorphic in $SH(k)$ to the
spectrum $M_{\bb P^{\wedge 1}}(\cc X)_f$ such that the composite isomorphism
   $$\Sigma^\infty_{\bb P^1}(\cc X)\xrightarrow\rho C_*(\Theta^\infty\Sigma^\infty_{\bb P^1,T}\cc X_{\nis})_f
     \xrightarrow\cong M_{\bb P^{\wedge 1}}(\cc X)_f$$
equals $\kappa_f$ of the theorem.

Moreover, the canonical morphism
   $$ M_{\bb P^{\wedge 1}}(\cc X)_f\to M_{\bb P^{\wedge 1}}(\cc X)\{1/2\}_f$$
is a $2^{-1}$-stable motivic equivalence of $\bb P^1$-spectra by
Proposition~\ref{frac12}. This finishes the proof of part (1) of
Theorem~\ref{Segal_Thm_II}.
\end{proof}

\section{Computing infinite $\bb P^1$-loop spaces}\label{infloop}

Let $\bl{\longrightarrow}{\Delta^{\op}Fr_0(k)}$ be the full
subcategory of $sShv_{\bullet}(Sm/k)$ consisting of directed
colimits of objects from $\Delta^{\op}Fr_0(k)$. Recall that
$\Delta^{\op}Fr_0(k)$ can be regarded as a full subcategory of
$sShv_{\bullet}(Sm/k)$ sending $X\in Fr_0(k)$ to $X_+\in
Shv_{\bullet}(Sm/k)$.

In order to compute $\Omega^{\infty}_{\bb P^1}\Sigma^{\infty}_{\bb
P^1}(\cc F)$ of any motivic space $\cc F\in sShv_{\bullet}(Sm/k)$
(see Theorem~\ref{infloopspaces}), we need  the following extension
of Theorem~\ref{Segal_Thm_II} to objects of
$\bl{\longrightarrow}{\Delta^{\op}Fr_0(k)}$:

\begin{thm}\label{Segal_Thm_Simpl}
Let $k$ be an infinite perfect field and $Y$ an object of
$\bl{\longrightarrow}{\Delta^{\op}Fr_0(k)}$. Then the following
statements are true:
\begin{enumerate}
\item The morphism $\kappa_f: \Sigma^{\infty}_{\bb P^1}Y_+\to M_{\bb P^{\wedge 1}}(Y)_f$
is a stable motivic equivalence of $\bb P^1$-spectra and the
morphism $\tau_f\circ\kappa_f:\Sigma^{\infty}_{\bb P^1}Y_+\to M_{\bb
P^{\wedge 1}}(Y)\{1/2\}_f$ is a $2^{-1}$-stable motivic equivalence
of $\bb P^1$-spectra.
\item If $\charr k\ne 2$ (respectively $k$ is of any characteristic) then the
$\bb P^1$-spectrum $M_{\bb P^{\wedge 1}}(Y)_f$ (respectively $M_{\bb
P^{\wedge 1}}(Y)\{1/2\}_f$) is a motivically fibrant
$\Omega$-spectrum in posisitive degrees. This means that for every
positive integer $n>0$ each motivic space $C_*(Fr(-,Y\times
T^{n}))_f$ (respectively $C_*(Fr(-,Y\times T^{n}))\{1/2\}_f$) is
motivically fibrant in the Morel--Voevod\-sky~\cite{MV} motivic
model category of simplicial Nisnevich sheaves and the structure map
   $$C_*(Fr(-,Y\times T^{n}))_f\to \Omega_{\bb P^1}(C_*(Fr(-,Y\times T^{n+1}))_f)$$
(respectively  $C_*(Fr(-,Y\times T^{n}))\{1/2\}_f\to \Omega_{\bb
P^1}(C_*(Fr(-,Y\times T^{n+1}))\{1/2\}_f)$) is a  weak equivalence
scheme\-wise.
\end{enumerate}
\end{thm}

\begin{proof}
The first statement of the theorem can be proved similar to
Theorem~\ref{Segal_Thm_II}(1) for any
$Y\in\bl{\longrightarrow}{\Delta^{\op}Fr_0(k)}$. Without loss of
generality it is enough to prove the second statement of the theorem
for $Y\in\Delta^{\op}Fr_0(k)$. Indeed, we use the facts that the
functor $M_{\bb P^{\wedge1}}(-)$ respects directed colimits and
directed colimits of locally fibrant objects are Nisnevich excisive
(even more: they are fibrant in the local flasque model structure of
sheaves in the sense of~\cite[4.6]{Is}).

We first observe that each space $C_*(Fr(-,Y\times T^{n}))$, $n>0$,
is locally connected, because it is the geometric realization of a
simplicial locally connected $H$-space $[k\in\Delta^{\op}\mapsto
C_*(Fr(-,Y_k\times T^{n}))]$ and $\pi_0^{\nis}(C_*(Fr(-,Y\times
T^{n})))=0$ by~\cite[7.1]{Gr}. Now Corollary~\ref{cor:A1_fibrant_3}
is true if we replace $X$ by $Y$ in it. Indeed, its proof relies on
connectedness of the corresponding spaces which we have just
verified and on Corollary~\ref{cor:A1_fibrant_2}. As a result, for
every positive integer $n>0$ each motivic space $C_*(Fr(-,Y\times
T^{n}))_f$ (respectively $C_*(Fr(-,Y\times T^{n}))\{1/2\}_f$) is
motivically fibrant in the Morel--Voevod\-sky~\cite{MV} motivic
model category of simplicial Nisnevich sheaves.

In order to show that the structure map
   $$C_*(Fr(-,Y\times T^{n}))_f\to \Omega_{\bb P^1}(C_*(Fr(-,Y\times T^{n+1}))_f),\quad n>0$$
(respectively  $C_*(Fr(-,Y\times T^{n}))\{1/2\}_f\to \Omega_{\bb
P^1}(C_*(Fr(-,Y\times T^{n+1}))\{1/2\}_f)$) is a  weak equivalence
scheme\-wise, we use Corollary~\ref{cor:A1_fibrant_3} (replacing $X$
by $Y$ in it) and the proof of Theorem~\ref{Segal_Thm_II}(2) (in
there we also use the fact that the geometric realization of a
simplicial stable local equivalence of $S^1$-spectra is a stable
local equivalence) to say that this is equivalent to showing that
the map
   $$M_{fr}(Y\times(\bb A^1//\bb G_m)^{n})_f\to\Omega_{\bb G_m^{\wedge 1}}\Omega_{S^1}(M_{fr}(Y\times(\bb A^1//\bb G_m)^{n}\otimes\bb G_m^{\wedge 1}\otimes S^1)_f)$$
(respectively  the map
$\wt{M}_{fr}(Y\times(\bb A^1//\bb G_m)^{n})\{1/2\}_f\to\Omega_{\bb G_m^{\wedge 1}}
\Omega_{S^1}(\wt{M}_{fr}(Y\times(\bb A^1//\bb G_m)^{n}\otimes\bb G_m^{\wedge 1}\otimes S^1)\{1/2\}_f)$)
is a schemewise stable weak equivalence of spectra. But the
latter follows from the Cancellation Theorem for framed motives~\cite{AGP}.
\end{proof}

\begin{cor}\label{Segal_Thm_Simplcor}
Let $k$ be an infinite perfect field. If $\charr k\ne 2$ (respectively
$k$ is of any characteristic), $n>0$ and
$\phi:\cc X\to\cc Y$ is a map between spaces from
$\bl{\longrightarrow}{\Delta^{\op}Fr_0(k)}$ such that
$\Sigma^{\infty}_{\bb P^1}(\phi)$ is an isomorphism in $SH(k)$
(respectively in $SH(k)[1/2]$) then the induced map
$\phi_*:C_*Fr(-,\cc X\times T^{n})\to C_*Fr(-,\cc Y\times T^{n})$
(respectively $\phi_*:C_*Fr(-,\cc X\times T^{n})\{1/2\}\to C_*Fr(-,\cc Y\times T^{n})\{1/2\}$)
is a local weak equivalence of motivic spaces.
\end{cor}

In the situation when $\cc
X\in\bl{\longrightarrow}{\Delta^{\op}Fr_0(k)}$ is such that the
space $C_*Fr(-,\cc X)$ is locally connected we arrive at the
following result:

\begin{thm}\label{infloopspaces2}
Let $k$ be an infinite perfect field of characteristic different
from 2. Suppose $\cc X\in\bl{\longrightarrow}{\Delta^{\op}Fr_0(k)}$
is such that the space $C_*Fr(-,\cc X)$ is locally connected. Then
$C_*Fr(-,\cc X)$ is an $\bb A^1$-local space and there is an
isomorphism
   $$\Omega^\infty_{\bb P^1}\Sigma^\infty_{\bb P^1}(\cc X)\cong C_*Fr(-,\cc X)$$
in $H_{\bb A^1}(k)$.
\end{thm}

\begin{proof}
The fact that $C_*Fr(-,\cc X)$ is an $\bb A^1$-local space is proved
similar to Corollary~\ref{cor:A1_fibrant_2}(2). The theorem would
follow if we showed that the $\bb P^1$-spectrum $M_{\bb P^{\wedge
1}}(\cc X)_f$ is motivically fibrant, because the suspension
spectrum $\Sigma_{\bb P^{1}}^\infty\cc X$ is stably equivalent to
$M_{\bb P^{\wedge 1}}(\cc X)_f$ by Theorem~\ref{Segal_Thm_Simpl}(1).
Now the fact that $M_{\bb P^{\wedge 1}}(\cc X)_f$ is motivically
fibrant repeats the proof of Theorem~\ref{Segal_Thm_Simpl}(2) word
for word.
\end{proof}

The proof of the preceding theorem shows the following

\begin{cor}\label{corinfloopspaces2}
Under the assumptions of Theorem~\ref{infloopspaces2} the $\bb
P^1$-spectrum $M_{\bb P^{\wedge 1}}(\cc X)_f$ is motivically
fibrant.
\end{cor}

By~\cite[3.1]{Bl} $sShv_\bullet(Sm/k)$ has the projective motivic
model structure in which generating cofibrations are given by
   $$X_+\wedge\partial\Delta^n_+\to X_+\wedge\Delta^n_+,\quad X\in Sm/k,\quad n\geq 0.$$
Equivalently, this family can be regarded as a family in
$\Delta^{\op}Fr_0(k)$ of the arrows $X\otimes\partial\Delta^n\to
X\otimes\Delta^n$, $n\geq 0$.

Let $\cc X\mapsto\cc X^c$ be the cofibrant replacement functor in
$sShv_\bullet(Sm/k)$ with respect to the projective model
structure.\label{blcof} Then $\cc
X^c\in\bl{\longrightarrow}{\Delta^{\op}Fr_0(k)}$, and hence
Theorem~\ref{Segal_Thm_Simpl} is applicable to it. It also follows
from Corollary~\ref{Segal_Thm_Simplcor} that each functor
   $$C_*Fr(-,(-)^c\times T^{n}):\cc X\in sShv_\bullet(Sm/k)\mapsto C_*Fr(-,\cc X^c\times T^{n})\in sShv_\bullet(Sm/k),\quad n\geq 1,$$
(respectively $\cc X\mapsto C_*Fr(-,\cc X^c\times T^{n})\{1/2\}$)
takes motivic weak equivalences to local weak equivalences whenever
$\charr k\ne 2$ (respectively $k$ is of any characteristic). Thus we
get functors
   $$C_*Fr(-,(-)^c\times T^{n}),C_*Fr(-,(-)^c\times T^{n})\{1/2\}:H_{\bb A^1}(k)\to H_{\nis}(k),\quad n\geq 1,$$
where $H_{\nis}(k)$ stands for the homotopy category of $sShv_\bullet(Sm/k)$
equipped with the local injective model structure.

Denote by $\Omega^\infty_{\bb P^1}\Sigma^\infty_{\bb P^1}(
H_{\nis}(k))$ the full subcategory of $H_{\nis}(k)$ consisting of
the infinite $\bb P^1$-loop spaces. Also, denote by
$\Sigma^\infty_{\bb P^1,1/2}:H_{\bb A^1}(k)\to SH(k)[1/2]$ the
composite functor
   $$H_{\bb A^1}(k)\xrightarrow{\Sigma^\infty_{\bb P^1}}SH(k)\to SH(k)[1/2],$$
where the right arrow is the standard localising functor inverting
2. The above arguments together with Theorem~\ref{Segal_Thm_Simpl}
imply the following

\begin{thm}\label{infloopspaces}
Let $k$ be an infinite perfect field. Then the following statements
are true for every field of characteristic different from 2:

\begin{enumerate}
\item The functor $C_*Fr(-,(-)^c\times T^{n})_f:H_{\bb A^1}(k)\to H_{\nis}(k)$, $n\geq 1$,
lands in $\Omega^\infty_{\bb P^1}\Sigma^\infty_{\bb P^1}( H_{\nis}(k))$.

\item For every $n\geq 1$ and $\cc X\in sShv_{\bullet}(Sm/k)$ the space
$C_*Fr(-,\cc X^c\times T^{n})_f$ has motivic homotopy type of
$\Omega^\infty_{\bb P^1}\Sigma^\infty_{\bb P^1}(\cc X\wedge T^n)$.
In particular, the functor $\Omega^\infty_{\bb
P^1}\Sigma^\infty_{\bb P^1}\circ(-\wedge T^n): H_{\bb
A^1}(k)\to\Omega^\infty_{\bb P^1}\Sigma^\infty_{\bb P^1}(
H_{\nis}(k))$ is isomorphic to the functor $\cc X\mapsto C_*Fr(-,\cc
X^c\times T^n)_f$.

\item For every $\cc X\in sShv_{\bullet}(Sm/k)$ the space
$\Omega_{\bb P^1}(C_*Fr(-,\cc X^c\times T)_f)$ has motivic homotopy
type of $\Omega^\infty_{\bb P^1}\Sigma^\infty_{\bb P^1}(\cc X)$. In
particular, the functor $\Omega^\infty_{\bb P^1}\Sigma^\infty_{\bb
P^1}:H_{\bb A^1}(k)\to\Omega^\infty_{\bb P^1}\Sigma^\infty_{\bb
P^1}( H_{\nis}(k))$ is isomorphic to the functor $\cc X\mapsto
\Omega_{\bb P^1}(C_*Fr(-,\cc X^c\times T)_f)$.

\item The functor $\Sigma^\infty_{\bb P^1}:H_{\bb A^1}(k)\to SH(k)$ is isomorphic to the
functor $\cc X\mapsto M_{\bb P^{\wedge 1}}(\cc X^c)_f$.
\end{enumerate}
In turn, if $k$ is of any characteristic, then:

\begin{itemize}
\item[$(1')$] The functor $C_*Fr(-,(-)^c\times T^{n})\{1/2\}_f:H_{\bb A^1}(k)\to H_{\nis}(k)$
lands in $\Omega^\infty_{\bb P^1}\Sigma^\infty_{\bb P^1}( H_{\nis}(k))$ for every $n\geq 1$.

\item[$(2')$] For every $n\geq 1$ and $\cc X\in sShv_{\bullet}(Sm/k)$ the space
$C_*Fr(-,\cc X^c\times T^{n})\{1/2\}_f$ has motivic homotopy type of
$\Omega^\infty_{\bb P^1}(\Sigma^\infty_{\bb P^1,1/2}(\cc X\wedge
T^n))$. In particular, the functor $\Omega^\infty_{\bb
P^1}(\Sigma^\infty_{\bb P^1,1/2}\circ(-\wedge T^n)): H_{\bb
A^1}(k)\to\Omega^\infty_{\bb P^1}\Sigma^\infty_{\bb P^1}(
H_{\nis}(k))$ is isomorphic to the functor $\cc X\mapsto C_*Fr(-,\cc
X^c\times T^n)\{1/2\}_f$.

\item[$(3')$] For every $\cc X\in sShv_{\bullet}(Sm/k)$ the space
$\Omega_{\bb P^1}(C_*Fr(-,\cc X^c\times T)\{1/2\}_f)$ has motivic
homotopy type of $\Omega^\infty_{\bb P^1}(\Sigma^\infty_{\bb
P^1,1/2}(\cc X))$. In particular, the functor $\Omega^\infty_{\bb
P^1}(\Sigma^\infty_{\bb P^1,1/2}):H_{\bb
A^1}(k)\to\Omega^\infty_{\bb P^1}\Sigma^\infty_{\bb P^1}(
H_{\nis}(k))$ is isomorphic to the functor $\cc X\mapsto\Omega_{\bb
P^1}(C_*Fr(-,\cc X^c\times T)\{1/2\}_f)$.

\item[$(4')$] $\Sigma^\infty_{\bb P^1,1/2}:H_{\bb A^1}(k)\to SH(k)[1/2]$ is isomorphic to the
functor $\cc X\mapsto M_{\bb P^{\wedge 1}}(\cc X^c)\{1/2\}_f$.
\end{itemize}
\end{thm}

\begin{cor}\label{infloopspacescor}
Let $k$ be an infinite perfect field. If $\charr k\ne 2$
(respectively $k$ is of any characteristic) and $n>0$ then for any
map of motivic spaces $\phi:\cc Y\to\cc Z$ such that
$\Sigma^{\infty}_{\bb P^1}(\phi)$ is an isomorphism in $SH(k)$
(respectively in $SH(k)[1/2]$) the induced map $\phi_*:C_*Fr(-,\cc
Y^c\times T^{n})\to C_*Fr(-,\cc Z^c\times T^{n})$ (respectively
$\phi_*:C_*Fr(-,\cc Y^c\times T^{n})\{1/2\}\to C_*Fr(-,\cc Z^c\times
T^{n})\{1/2\}$ is a local weak equivalence.
\end{cor}

\begin{rem}{\rm
In fact, assertion (3) of Theorem~\ref{infloopspaces} is true for
any local group-like completion of the motivic space $C_*Fr(-,\cc
X^c)_f$, functorial in $\cc X$. It follows from the Additivity
Theorem that the $\pi_0$-sheaf of $C_*Fr(-,\cc X^c)$ is a sheaf of
Abelian monoids. One of such local group-like completion functors is
given in the assertion. Another local group-like completion is given
by $\cc X\mapsto\Omega_{S^1}(C_*Fr(-,\cc X^c\otimes S^1)_f)$. In
particular, there is a motivic weak equivalence of spaces
   $$\Omega^\infty_{\bb P^1}\Sigma^\infty_{\bb P^1}(\cc X)\simeq\Omega_{S^1}(C_*Fr(-,\cc X^c\otimes S^1)_f)$$
for any $\cc X\in sShv_\bullet(Sm/k)$ and any infinite perfect field with $\charr k\ne 2$.

}\end{rem}

In turn, if $\cc X\in sShv_\bullet(Sm/k)$ is such that the space
$C_*Fr(-,\cc X^c)$ is locally connected we get the following result:

\begin{thm}\label{infloopspaces3}
Let $k$ be an infinite perfect field of characteristic different
from 2 and let $\cc X\in sShv_\bullet(Sm/k)$ be such that the space
$C_*Fr(-,\cc X^c)$ is locally connected. Then $C_*Fr(-,\cc X^c)$ is
an $\bb A^1$-local space and there is an isomorphism
   $$\Omega^\infty_{\bb P^1}\Sigma^\infty_{\bb P^1}(\cc X)\cong C_*Fr(-,\cc X^c)$$
in $H_{\bb A^1}(k)$.
\end{thm}

\begin{proof}
Since $\cc X^c\in\bl{\longrightarrow}{\Delta^{\op}Fr_0(k)}$, the
theorem follows from Theorem~\ref{infloopspaces2}.
\end{proof}

The proof of the preceding theorem shows the following

\begin{cor}\label{corinfloopspaces3}
Under the assumptions of Theorem~\ref{infloopspaces3} the $\bb
P^1$-spectrum $M_{\bb P^{\wedge 1}}(\cc X^c)_f$ is motivically
fibrant. Moreover, $\Sigma_{\bb P^1}^\infty\cc X$ is isomorphic to
$M_{\bb P^{\wedge 1}}(\cc X^c)_f$ in $SH(k)$.
\end{cor}

\section{Further applications of framed motives}\label{furtherappl}

Having applied the macinery of framed motives to prove
Theorem~\ref{Segal_Thm_II}, we want to give further applications.
One of the applications computes the suspension bispectrum
$\Sigma^\infty_{S^1}\Sigma^\infty_{\bb G}X_+$ of a smooth algebraic
variety $X$ in terms of twisted framed motives of $X$. Another
important application will be purely topological. It will compute
the classical sphere spectrum $\Sigma^\infty_{S^1}S^0$ as the framed
motive $M_{fr}(pt)(pt)$ of the point $pt=\spec k$ evaluated at the
point whenever the base field $k$ is algebraically closed of
characteristic zero.

Denote by $\bb G$ the cone $(\bb G_m)_+//pt_+$ of the embedding
$pt_+\bl 1\hookrightarrow(\bb G_m)_+$ in the category of pointed
simplicial presheaves $sPre_\bullet(Sm/k)$. It is termwise equal to
   $$(-,\bb G_m)_+,(-,\bb G_m)_+\vee(-,pt)_+,(-,\bb G_m)_+\vee(-,pt)_+\vee(-,pt)_+,\ldots$$
Moreover, its sheafification equals $(\bb G_m^{\wedge 1})_+$, which
is termwise equal to
   $$(-,\bb G_m)_+,(-,\bb G_m\sqcup pt)_+,(-,\bb G_m\sqcup pt\sqcup pt)_+,\ldots$$
The sheafification is represented in the category $\Delta^{\op}(Fr_0(k))$ by the object
$\bb G_m^{\wedge 1}$ (see Notation~\ref{edem}), which is termwise equal to
   $$\bb G_m,\bb G_m\sqcup pt,\bb G_m\sqcup pt\sqcup pt,\ldots$$

One of the models for Morel--Voevodsky's $SH(k)$ can be defined in
terms of $(S^1,\bb G)$-bispectra (see, e.g.,~\cite{Jar}). The main
$(S^1,\bb G)$-bispectrum we work with is given by the sequence of framed motives
   $M_{fr}^{\bb G}(X)=(M_{fr}(X),M_{fr}(X\times\bb G_m^{\wedge 1}),M_{fr}(X\times\bb G_m^{\wedge 2}),\ldots),\quad X\in Sm/k,$
where simplicial objects $\bb G_m^{\wedge n}\in \Delta^{\op}(Fr_0(k))$ are defined in Notation~\ref{edem}.
Each structure map is defined as the composition
   $$M_{fr}(X\times\bb G_m^{\wedge n})\to\underline{\Hom}((\bb G_m^{\wedge 1})_+,M_{fr}(X\times\bb G_m^{\wedge n+1}))
      \to\underline{\Hom}(\bb G,M_{fr}(X\times\bb G_m^{\wedge n+1})),$$
where the left map is defined as~\eqref{smashell}. We shall also
call $M_{fr}(X\times\bb G_m^{\wedge n})$ the {\it $n$-twisted framed
motive of $X$}, and write $M_{fr}(X)(n)$ to denote $M_{fr}(X\times\bb G_m^{\wedge n})$.
So we can write for brevity
$$M_{fr}^{\bb G}(X)=(M_{fr}(X),M_{fr}(X)(1),M_{fr}(X)(2),\ldots),$$
to denote the $(S^1,\bb G)$-bispectrum $M_{fr}^{\bb G}(X)$.

In turn, we can similarly define a bispectrum
   $$\wt{M}_{fr}^{\bb G}(X)\{1/2\}=(\wt{M}_{fr}(X)\{1/2\},\wt{M}_{fr}(X\times\bb G_m^{\wedge 1})\{1/2\},
     \wt{M}_{fr}(X\times\bb G_m^{\wedge 2})\{1/2\},\ldots),\quad X\in Sm/k.$$
It is worth to mention that the bispectrum $M_{fr}^{\bb G}(X)$ is
constructed in the same way as the bispectrum
   $$M_{\bb K}^{\bb G}(X)=(M_{\bb K}(X),M_{\bb K}(X)(1),M_{\bb K}(X)(2),\ldots),\quad X\in Sm/k,$$
where each $S^1$-spectrum $M_{\bb K}(X)(n)=M_{\bb K}(X\times\bb
G_m^{\wedge n})$ is the $n$-twisted $K$-motive of $X$ in the sense
of~\cite{GP2}. It was shown in~\cite{GP2} that $M_{\bb K}^{\bb
G}(pt)$ represents the bispectrum $f_0(KGL)$.

The main result of this section is as follows.

\begin{thm}\label{Motivic_Segal_Thm_III}
Let $k$ be an infinite perfect field. If $\charr k\ne 2$
(respectively $k$ is of any characteristic) then for any $X\in
Sm/k$, the canonical map of bispectra
$\Sigma^{\infty}_{S^1}\Sigma^{\infty}_{\bb G}X_+\to M_{fr}^{\bb G}(X)$
(respectively $\Sigma^{\infty}_{S^1}\Sigma^{\infty}_{\bb G}X_+\to
\wt{M}^{\bb G}_{fr}(X)\{1/2\}$) is a stable motivic weak equivalence
(respectively $2^{-1}$-stable motivic weak equivalence).
\end{thm}

\begin{proof}
If $\charr k\ne 2$ (respectively $k$ is of any characteristic) then
it is enough to prove that on bigraded presheaves $\pi^{\bb
A^1}_{*,*}(\gamma)$ (respectively $\pi^{\bb
A^1}_{*,*}(\gamma)\otimes\bb Z[1/2]$) is an isomorphism. We shall
prove the theorem for the case when $\charr k\ne 2$, because the
other case is proved in a similar fashion. Every bispectrum yields a
$S^1\wedge \bb G$-spectrum by taking diagonal. In order to avoid
massive notation, we prove the theorem for the case $X=pt$. The same
proof works for any $X\in Sm/k$. Let $r\geq 0$, $s$, $n\geq 1$ be
integers with $s+n\geq 0$. There is a commutative diagram in the
homotopy category $H_{\bb A^1}(k)$ of pointed motivic spaces
\footnotesize
$$\xymatrix{[S^r\wedge U_+\wedge (S^1\wedge \bb G)^{s+n}, (S^1\wedge \bb G)^{n}] \ar[rr]^{(1_n)}&&
[S^r\wedge U_+\wedge (S^1\wedge \bb G)^{s+n}, C_*Fr(({S}^1\otimes \bb G^{\wedge 1}_m)^{\wedge n})]  \\
            [S^r\wedge U_+\wedge (S^1\wedge \bb G)^{s+n}, (\bb A^1_+//(\bb G_m)_+)^{n}] \ar[rr]^{(2_n)}\ar[u]_{(u^{\wedge n})_*}\ar[d]^{(v^{\wedge n})_*} &&
            [S^r\wedge U_+\wedge (S^1\wedge \bb G)^{s+n}, C_*Fr((\bb A^1//\bb G_m)^{\wedge n})] \ar[u]^{{u}^n_*} \ar[d]_{v^{n}_*} \\
            [S^r\wedge U_+\wedge (S^1\wedge \bb G)^{s+n}, T^{n}] \ar[rr]^{(3_n)}\ar[d]^{(u^{\wedge (s+n)})^*} &&
            [S^r\wedge U_+\wedge (S^1\wedge \bb G)^{s+n}, C_*Fr(T^{n})] \ar[d]_{(u^{\wedge (s+n)})^*} \\
            [S^r\wedge U_+\wedge (\bb A^1_+//(\bb G_m)_+)^{s+n}, T^{n}] \ar[rr]^{(4_n)} &&
            [S^r\wedge U_+\wedge (\bb A^1_+//(\bb G_m)_+)^{s+n}, C_*Fr(T^{n})] \\
            [S^r\wedge U_+\wedge T^{s+n}, T^{n}] \ar[rr]^{(5_n)} \ar[u]_{(v^{\wedge (s+n)})^*} \ar[d]^{(\sigma^{\wedge (s+n)})^*} &&
            [S^r\wedge U_+\wedge T^{s+n}, C_*Fr(T^{n})] \ar[u]^{(v^{\wedge (s+n)})^*} \ar[d]_{(\sigma^{\wedge (s+n)})^*} \\
            [S^r\wedge U_+\wedge \bb P^{\wedge {s+n}}, T^{n}] \ar[rr]^{(6_n)}  &&
            [S^r\wedge U_+\wedge \bb P^{\wedge {s+n}}, C_*Fr(T^{n})]}$$
\normalsize
In this diagram all left vertical arrows are bijections, since the natural morphisms
$u:\bb A^1_+//(\bb G_m)_+\to S^1\wedge \bb G$, $v:\bb A^1_+//(\bb G_m)_+\to T$ and
$\sigma:\bb P^{\wedge 1}\to T$ are motivic equivalences.
All right vertical arrows are bijections, since the morphisms
   $$C_*Fr((\bb A^1//\bb G_m)^{\wedge n}) \xrightarrow{{u}^n_*} C_*Fr(({S}^1\otimes\bb G^{\wedge 1}_m)^{\wedge n}), \ \
C_*Fr((\bb A^1//\bb G_m)^{\wedge n}) \xrightarrow{v^{n}_*} C_*Fr(T^{n})$$
are local equivalences by Corollaries~\ref{potom1} and~\ref{potom2}.

Fit now each of twelve vertices of the diagram into a direct colimit over $n$ as follows.
For vertices on the right hand side we form direct colimits
with respect to the $T$-spectum with entries $\{C_*Fr(T^n)\}$,
with respect to the $\bb A^1_+//(\bb G_m)_+$-spectrum
with entries $\{C_*Fr((\bb A^1//\bb G_m)^{\wedge n})\}$ as well as
with respect to the $S^1\wedge \bb G$-spectrum with entries
$\{C_*Fr(({S}^1\otimes \bb G^{\wedge 1}_m)^{\wedge n})\}$. Likewise,
for vertices on the right hand side we form direct colimits
with respect to the $T$-spectum with entries $\{T^n\}$,
with respect to the $\bb A^1_+//(\bb G_m)_+$-spectum with entries $\{(\bb A^1//\bb G_m)^{\wedge n}\}$,
with respect to the $S^1\wedge \bb G$-spectrum with entries $\{(S^1\wedge \bb G)^{\wedge n}\}$.

The family of morphisms $\{(6_n)\}$ forms a morphism on the direct
colimit, since it corresponds to the $T$-spectra morphism $\{T^n\}
\to \{C_*Fr(T^n)\}$. For $i=5,4,3$ the families of morphisms
$\{(i_n)\}$ form morphisms on the direct colimits by the same
reason. The family of morphisms $\{(1_n)\}$ forms a morphism on the
direct colimit, since it corresponds to the $S^1\wedge \bb
G$-spectra morphism $\{(S^1\wedge \bb G)^{\wedge n}\} \to
\{C_*Fr(({S}^1\otimes \bb G^{\wedge 1}_m)^{\wedge n})\}$. Finally, the
family of morphisms $\{(2_n)\}$ forms a morphism on the direct
colimit, since it corresponds to the $\bb A^1_+//(\bb
G_m)_+$-spectra morphism $\{(\bb A^1//\bb G_m)^{\wedge n}\} \to
\{C_*Fr((\bb A^1//\bb G_m)^{\wedge n})\}$. In a similar fashion for each
vertical map the corresponding family of arrows forms a morphism on
the direct colimits.

In this way we get a commutative diagram consisting of twelve direct
colimits and morphisms between them. We also get a commutative
diagram consisting of twelve groups and homomorphisms between them.
In that diagram of groups all the vertical arrows are isomorphisms
as mentioned above. The bottom arrow is an isomorphism by
Theorem~\ref{Segal_Thm_II}, hence so is the top arrow. We conclude
that for $r\geq 0$ the map of presheaves $\pi^{\bb
A^1}_{r+2s,s}(\Sigma^{\infty}_{\bb G}
\Sigma^{\infty}_{S^1}(S^0))\xrightarrow{\gamma_*} \pi^{\bb
A^1}_{r+2s,s}(M^{\bb G}_{fr}(pt))$ is an isomorphism for any integer $s$.
In another words, the map $\gamma_*$ is an isomorphisms on
presheaves $\pi^{\bb A^1}_{a,b}$ with $2a-b\geq 0$. In particular,
for any $U\in Sm/k$ and any $t>0$ the map
   $$\pi^{\bb A^1}_{2a,a}(\Sigma^{\infty}_{\bb G} \Sigma^{\infty}_{S^1}(S^0))(U\times \bb G_m^{\times t})\xrightarrow{\gamma_*}
      \pi^{\bb A^1}_{2a,a}(M^{\bb G}_{fr}(pt))(U\times \bb G_m^{\times t})$$
is an isomorphism. Note that
$\pi^{\bb A^1}_{2a,a}(\Sigma^{\infty}_{\bb G} \Sigma^{\infty}_{S^1}(S^0))(U_+\wedge \bb G_m^{\wedge t})$
is a canonical direct summand of the group
$\pi^{\bb A^1}_{2a,a}(\Sigma^{\infty}_{\bb G} \Sigma^{\infty}_{S^1}(S^0))(U\times \bb G_m^{\times t})$,
and the group
$\pi^{\bb A^1}_{2a,a}(M_{fr}^{\bb G}(pt))(U\wedge \bb G_m^{\wedge t})$
is a canonical direct summand of the group
$\pi^{\bb A^1}_{2a,a}(M_{fr}^{\bb G}(pt))(U\times \bb G_m^{\times t})$.
Hence the map
\footnotesize
$$
\pi^{\bb A^1}_{2a-t,a}(\Sigma^{\infty}_{\bb G} \Sigma^{\infty}_{S^1}(S^0))(U)=\pi^{\bb A^1}_{2a,a}(\Sigma^{\infty}_{\bb G} \Sigma^{\infty}_{S^1}(S^0))(U_+\wedge \bb G_m^{\wedge t})
\xrightarrow{\gamma_*}
\pi^{\bb A^1}_{2a,a}(M^{\bb G}_{fr}(pt))(U_+\wedge \bb G_m^{\wedge t})= \pi^{\bb A^1}_{2a-t,a}(M^{\bb G}_{fr}(pt))(U)
$$
\normalsize
is an isomorphism too. Thus, the map $\gamma_*$
is an isomorphism on presheaves $\pi^{\bb A^1}_{a,b}$ with $2a-b < 0$, whence the theorem.
\end{proof}

\begin{cor}\label{coruhu}
Let $k$ be an infinite perfect field and $X$ be smooth. If $\charr
k\ne 2$ (respectively $k$ is of any characteristic) then
$\pi_{-n,-n}^{\bb A^1}((\Sigma^\infty_{S^1}\Sigma^\infty_{\bb G}
X_+)(pt))$ (respectively $\pi_{-n,-n}^{\bb
A^1}((\Sigma^\infty_{S^1}\Sigma^\infty_{\bb G} X_+)(pt))\otimes\bb
Z[1/2]$), $n\geq 0$, is the Grothendieck group of the commutative
monoid $\pi_0(C_*Fr(pt,X\times\bb G_m^{\wedge n}))$ (respectively
the Abelian group $\pi_0(C_*Fr(pt,X\times\bb G_m^{\wedge
n})\{1/2\})$).
\end{cor}

\begin{cor}\label{corugu}
Let $k$ be an infinite perfect field and let $X$ be smooth. If
$\charr k\ne 2$ (respectively $k$ is of any characteristic) then
$\pi_{-n,-n}^{\bb A^1}(\Sigma^\infty_{S^1}\Sigma^\infty_{\bb G}
X_+)(pt)$
   $$\pi_{-n,-n}^{\bb A^1}(\Sigma^\infty_{S^1}\Sigma^\infty_{\bb G} X_+)(pt)=H_0(\bb
     ZF(\Delta_k^\bullet,X\times\bb G_m^{\wedge n})),\quad n\geq 0$$
(respectively $\pi_{-n,-n}^{\bb
A^1}(\Sigma^\infty_{S^1}\Sigma^\infty_{\bb G} X_+)(pt)\otimes\bb
Z[1/2]=H_0(\bb ZF(\Delta_k^\bullet,X\times\bb G_m^{\wedge
n}))\otimes\bb Z[1/2]$). In particular, $\pi_{-n,-n}^{\bb
A^1}(\Sigma^\infty_{S^1}\Sigma^\infty_{\bb G} S^0)(pt)=H_0(\bb
ZF(\Delta_k^\bullet,\bb G_m^{\wedge n}))=K_n^{MW}(k)$ if $n\geq 0$
and $\charr k=0$.
\end{cor}

\begin{proof}
The fact that $H_0(\bb ZF(\Delta_k^\bullet,\bb G_m^{\wedge
n}))=K_n^{MW}(k)$, $n\geq 0$, has been proved by Neshitov
in~\cite{Nesh} for fields of characteristic zero. Thus the statement
recovers the celebrated theorem of Morel~\cite{Mor1} for
Milnor--Witt $K$-theory.
\end{proof}

It is also useful to have Theorem~\ref{Motivic_Segal_Thm_III} for
simplicial schemes or, more generally, for objects in
$\bl{\longrightarrow}{\Delta^{\op}Fr_0(k)}$ (cf.
Theorem~\ref{Segal_Thm_Simpl}).

\begin{thm}\label{Motivic_Segal_Thm_SimplBi}
Let $k$ be an infinite perfect field. If $\charr k\ne 2$
(respectively $k$ is of any characteristic) then for any $Y\in
\bl{\longrightarrow}{\Delta^{\op}Fr_0(k)}$, the canonical map of
bispectra $\Sigma^{\infty}_{S^1}\Sigma^{\infty}_{\bb G}Y_+\to M^{\bb
G}_{fr}(Y)$ (respectively $\Sigma^{\infty}_{S^1}\Sigma^{\infty}_{\bb
G}Y_+\to \wt{M}^{\bb G}_{fr}(Y)\{1/2\}$) is a stable motivic weak
equivalence (respectively $2^{-1}$-stable motivic weak equivalence).
\end{thm}

\begin{proof}
If we use Theorem~\ref{Segal_Thm_Simpl}, our proof repeats that of
Theorem~\ref{Motivic_Segal_Thm_III} word for word.
\end{proof}

Here is an application of the preceding theorem.

\begin{thm}\label{applic}
Suppose the base field $k$ is perfect infinite with $\charr k\neq2$
(respectively $k$ is of any characteristic). For any
$Y\in\bl{\longrightarrow}{\Delta^{\op}Fr_0(k)}$ one has a canonical
isomorphism
\begin{equation*}\label{KeyAdjunction1}
SH(k)(\Sigma^{\infty}_{\bb G} \Sigma^{\infty}_{S^1}X_+,
\Sigma^{\infty}_{\bb G} \Sigma^{\infty}_{S^1}Y_+[n])=
SH_{S^1}^{\nis}(k)(\Sigma^{\infty}_{S^1}X_+,M_{fr}(Y)[n]),\quad n\geq 0,
\end{equation*}
(respectively $SH(k)[1/2](\Sigma^{\infty}_{\bb G}
\Sigma^{\infty}_{S^1}X_+, \Sigma^{\infty}_{\bb G}
\Sigma^{\infty}_{S^1}Y_+[n])=
SH_{S^1}^{\nis}(k)(\Sigma^{\infty}_{S^1}X_+,M_{fr}(Y)\{1/2\}[n])$),
where $SH_{S^1}^{\nis}(k)$ is the ordinary stable homotopy category
of Nisnevich sheaves of $S^1$-spectra.
\end{thm}

\begin{proof}
Suppose the base field $k$ is perfect infinite with $\charr k\neq2$.
Consider a bispectrum
   $$M_{fr}^{\bb G}(Y)_f=(M_{fr}(Y)_f,M_{fr}(Y\times\bb G_m^{\wedge 1})_f,M_{fr}(Y\times\bb G_m^{\wedge 2})_f,\ldots)$$
obtained from $M_{fr}^{\bb G}(Y)$ by taking Nisnevich local stable
fibrant replacements at each level. It is shown similar
to~\cite[Theorem~B]{AGP} that $M_{fr}^{\bb G}(Y)_f$ is a
motivically fibrant bispectrum.

Theorem~\ref{applic} implies an isomorphism
   $$SH(k)(\Sigma^{\infty}_{\bb G} \Sigma^{\infty}_{S^1}X_+,\Sigma^{\infty}_{\bb G} \Sigma^{\infty}_{S^1}Y_+[n])=
        SH(k)(\Sigma^{\infty}_{\bb G}\Sigma^{\infty}_{S^1}X_+,M_{fr}^{\bb G}(Y)_f[n]),\quad n\geq 0.$$
But
   \begin{gather*}SH(k)(\Sigma^{\infty}_{\bb G}\Sigma^{\infty}_{S^1}X_+,M_{fr}^{\bb G}(Y)_f[n])\cong
       SH_{S^1}(k)(\Sigma^{\infty}_{S^1}X_+,M_{fr}(Y)_f[n])\cong\\
       SH_{S^1}^{\nis}(k)(\Sigma^{\infty}_{S^1}X_+,M_{fr}(Y)_f[n])
       \cong SH_{S^1}^{\nis}(k)(\Sigma^{\infty}_{S^1}X_+,M_{fr}(Y)[n]).
   \end{gather*}
If $k$ is of any characteristic then we can likewise show an isomorphism
$$SH(k)[1/2](\Sigma^{\infty}_{\bb G} \Sigma^{\infty}_{S^1}X_+,
\Sigma^{\infty}_{\bb G} \Sigma^{\infty}_{S^1}Y_+[n])=
SH_{S^1}^{\nis}(k)(\Sigma^{\infty}_{S^1}X_+,\wt{M}_{fr}(Y)\{1/2\}[n]).$$
It remains to observe that ${M}_{fr}(Y)\{1/2\}$ is stably equivalent
to $\wt{M}_{fr}(Y)\{1/2\}$.
\end{proof}

\begin{cor}\label{Motivic_Segal_Thm_Cor}
Suppose the base field $k$ is perfect infinite with $\charr k\neq2$
(respectively $k$ is of any characteristic). For any morphism
$\phi:Y\to Z$ in $\bl{\longrightarrow}{\Delta^{\op}Fr_0(k)}$ such
that $\Sigma^{\infty}_{\bb G} \Sigma^{\infty}_{S^1}(\phi)$ is an
isomorphism in $SH(k)$ (respectively in $SH(k)[1/2]$), the morphism
of framed motives $M_{fr}(\phi):M_{fr}(Y)\to M_{fr}(Z)$
(respectively $M_{fr}(\phi)\{1/2\}:M_{fr}(Y)\{1/2\}\to
M_{fr}(Z)\{1/2\}$) is a local stable equivalence of $S^1$-spectra.
\end{cor}

Let $\cc X\mapsto\cc X^c$ be the cofibrant replacement functor in
$sShv_\bullet(Sm/k)$ (see p.~\pageref{blcof}). Then $\cc X^c$ is in
$\bl{\longrightarrow}{\Delta^{\op}Fr_0(k)}$, and hence
Theorems~\ref{Motivic_Segal_Thm_SimplBi}-\ref{applic} are applicable
to it. It also follows from Corollary~\ref{Motivic_Segal_Thm_Cor}
that each functor
   $$M_{fr}((-)^c):\cc X\in sShv_\bullet(Sm/k)\mapsto M_{fr}(\cc X^c)\in Sp_{S^1}(sShv_\bullet(Sm/k))$$
(respectively $\cc X\mapsto M_{fr}(\cc X^c)\{1/2\}$) takes motivic
weak equivalences to stable local weak equivalences whenever $\charr
k\ne 2$ (respectively $k$ is of any characteristic). Thus we get
functors
   $$M_{fr}((-)^c),M_{fr}((-)^c)\{1/2\}:H_{\bb A^1}(k)\to SH_{S^1}^{\nis}(k),$$
where $SH_{S^1}^{\nis}(k)$ stands for the homotopy category of $Sp_{S^1}(sShv_\bullet(Sm/k))$
equipped with the stable local injective model structure.

In a similar fashion we define functors
   $$M^{\bb G}_{fr}((-)^c),M^{\bb G}_{fr}((-)^c)\{1/2\}:H_{\bb A^1}(k)\to SH(k).$$
Explicitly, they take a motivic space $\cc X$ to bispectra $M^{\bb
G}_{fr}(\cc X^c),M^{\bb G}_{fr}(\cc X^c)\{1/2\}$. In fact, we shall
extend both functors to $SH(k)$ in Section~\ref{mgfrfunctor}.

Denote by $\Omega^\infty_{\bb G}(SH_{S^1}^{\nis}(k))$ the full
subcategory of $SH_{S^1}^{\nis}(k)$ consisting of the infinite $\bb G$-loop spectra.
Also, denote by $\Sigma^\infty_{\bb G}\Sigma^\infty_{S^1,1/2}:H_{\bb A^1}(k)\to SH(k)[1/2]$ the composite functor
   $$H_{\bb A^1}(k)\xrightarrow{\Sigma^\infty_{\bb G}\Sigma^\infty_{S^1}}SH(k)\to SH(k)[1/2],$$
where the right arrow is the standard localising functor inverting
2. The above arguments together with
Theorems~\ref{Motivic_Segal_Thm_SimplBi}-\ref{applic} and
Corollary~\ref{Motivic_Segal_Thm_Cor} imply the following

\begin{thm}\label{infloopspectra}
Let $k$ be an infinite perfect field. Then the following statements are
true for every field of characteristic different from 2:

\begin{enumerate}
\item The functor $M_{fr}((-)^c)_f:H_{\bb A^1}(k)\to SH_{S^1}^{\nis}(k)$
lands in $\Omega^\infty_{\bb G}(SH_{S^1}^{\nis}(k))$, where ``$f$" refers
to the stable local fibrant replacement of $S^1$-spectra.

\item For every $\cc X\in sShv_{\bullet}(Sm/k)$ the spectrum
$M_{fr}(\cc X^c)_f$ has stable motivic homotopy type
of $\Omega^\infty_{\bb G}\Sigma^\infty_{\bb G}\Sigma^\infty_{S^1}(\cc X)$.
In particular, the functor
$\Omega^\infty_{\bb G}\Sigma^\infty_{\bb G}\Sigma^\infty_{S^1}:
H_{\bb A^1}(k)\to\Omega^\infty_{\bb G}(SH_{S^1}^{\nis}(k))$
is isomorphic to the functor $\cc X\mapsto M_{fr}(\cc X^c)_f$.

\item The functor $\Sigma^\infty_{\bb G}\Sigma^\infty_{S^1}:H_{\bb A^1}(k)\to SH(k)$ is isomorphic to the
functor $\cc X\mapsto M^{\bb G}_{fr}(\cc X^c)$.
\end{enumerate}
In turn, if $k$ is of any characteristic, then:

\begin{itemize}
\item[$(1')$]  The functor $M_{fr}((-)^c)\{1/2\}_f:H_{\bb A^1}(k)\to SH_{S^1}^{\nis}(k)$
lands in $\Omega^\infty_{\bb G}(SH_{S^1}^{\nis}(k))$, where ``$f$" refers
to the stable local fibrant replacement of $S^1$-spectra.

\item[$(2')$] For every $\cc X\in sShv_{\bullet}(Sm/k)$ the spectrum
$M_{fr}(\cc X^c)\{1/2\}_f$ has stable motivic homotopy type
of $\Omega^\infty_{\bb G}\Sigma^\infty_{\bb G}\Sigma^\infty_{S^1,1/2}(\cc X)$.
In particular, the functor
$\Omega^\infty_{\bb G}\Sigma^\infty_{\bb G}\Sigma^\infty_{S^1,1/2}:
H_{\bb A^1}(k)\to\Omega^\infty_{\bb G}(SH_{S^1}^{\nis}(k))$
is isomorphic to the functor $\cc X\mapsto M_{fr}(\cc X^c)\{1/2\}_f$.

\item[$(3')$] $\Sigma^\infty_{\bb G}\Sigma^\infty_{S^1,1/2}:H_{\bb A^1}(k)\to SH(k)$ is isomorphic to the
functor $\cc X\mapsto M^{\bb G}_{fr}(\cc X^c)\{1/2\}$.
\end{itemize}
\end{thm}

\begin{cor}\label{infloopspectracor}
Let $k$ be an infinite perfect field. If $\charr k\ne 2$ (respectively
$k$ is of any characteristic) then for any map
of motivic spaces $\phi:\cc Y\to\cc Z$ such that
$\Sigma^\infty_{S^1}\Sigma^\infty_{\bb G}(\phi)$ is an isomorphism in $SH(k)$
(respectively in $SH(k)[1/2]$) the induced map
$\phi_*:M_{fr}(\cc Y^c)\to M_{fr}(\cc Z^c)$
(respectively $\phi_*:M_{fr}(\cc Y^c)\{1/2\}\to M_{fr}(\cc Z^c)\{1/2\}$)
is a local stable equivalence.
\end{cor}

We finish the section with a topological application of framed motives. It gives
an explicit model for the classical sphere spectrum $\Sigma^\infty_{S^1} S^0$.

\begin{thm}\label{corres}
Let $k$ be an algebraically closed field of characteristic zero, with embedding
$k\hookrightarrow\bb C$.
Then the framed motive $M_{fr}(pt)(pt)$ of the point $pt=\spec k$
evaluated at $pt$ has stable homotopy type of the classical sphere spectrum
$\Sigma^\infty_{S^1}S^0$. In particular, $C_*Fr(pt,pt\otimes S^1)$ has homotopy type of
$\Omega^\infty_{S^1}\Sigma^\infty_{S^1}S^1$.
\end{thm}

\begin{proof}
By a theorem of Levine~\cite{Levine} the functor $c:SH\to SH(k)$,
induced by the functor
   $$sSets_\bullet\to sPre_\bullet(Sm/k)$$
sending a pointed simplicial set to the constant presheaf on $Sm/k$,
is fully faithful. The functor $c$ comes from a left Quillen functor (see the proof
of~\cite[6.5]{Levine}). Its right Quillen functor from bispectra to ordinary $S^1$-specra takes a
bispectrum $E=(E_0, E_1,\ldots)$ to $E_0(pt)$. Moreover,
$c$ induces an isomorphism
${\pi}_n(E)\to {\pi}_{n,0}^{\bb A^1}(c(E))$ for all spectra $E$.

Consider $M_{fr}^{\bb G}(pt)$. By
Theorem~\ref{Motivic_Segal_Thm_III} the canonical morphism
   $$\Sigma^\infty_{S^1}\Sigma^\infty_{\bb G}S^0=c(\Sigma^\infty_{S^1}S^0)\to M_{fr}^{\bb G}(pt)$$
is a motivic stable equivalence of bispectra. Consider a bispectrum
   $$M_{fr}^{\bb G}(pt)_f=(M_{fr}(pt)_f,M_{fr}(\bb G_m^{\wedge 1})_f,M_{fr}(\bb G_m^{\wedge 2})_f,\ldots)$$
obtained from $M_{fr}^{\bb G}(pt)$ by taking Nisnevich local stable
fibrant replacements at each level. By~\cite[Theorem~B]{AGP}
$M_{fr}^{\bb G}(pt)_f$ is a motivically fibrant bispectrum, and
hence a fibrant replacement of
$\Sigma^\infty_{S^1}\Sigma^\infty_{\bb G}S^0$.

It follows that $M_{fr}(pt)_f(pt)$ is a fibrant replacement of
$\Sigma^\infty_{S^1}S^0$, because each homomorphism
${\pi}_n(\Sigma^\infty_{S^1}S^0)\to {\pi}_{n,0}^{\bb
A^1}(c(\Sigma^\infty_{S^1}S^0)) \bl\cong\to\pi_{n,0}^{\bb
A^1}(M^{\bb G}_{fr}(pt)_f)=\pi_n(M_{fr}(pt)_f(pt))$ is an
isomorphism. It remains to observe that $M_{fr}(pt)_f(pt)$ is stably
equivalent to $M_{fr}(pt)(pt)$.
\end{proof}

\section{The big framed motive functor $\cc M^{b}_{fr}$}\label{mgfrfunctor}

In Section~\ref{furtherappl} we have computed the functor
$\Sigma^\infty_{S^1}\Sigma^\infty_{\bb G}:H_{\bb A^1}(k)\to SH(k)$
as the functor $M_{fr}^{\bb G}:H_{\bb A^1}(k)\to SH(k)$. We extend
the latter functor to bispectra below, but first we start with
preparations. We assume in this section that the base field $k$ is
infinite perfect with $\charr k\ne 2$.

We first give an explicit description of $Ex^\infty(C_*\cc Fr(\cc
X))$ for every space $\cc X\in sShv_\bullet(Sm/k)$, where
$Ex^\infty$ is the Kan complex. Voevodsky~\cite[Section~3]{VoeICM}
defined a realization functor from simplicial sets to Nisnevich
sheaves $|-|:sSets\to Shv_{\nis}(Sm/k)$ such that
$|\Delta[n]|=\Delta^n$, where $\Delta[n]$ is the standard
$n$-simplex. Under this notation the cosimpicial scheme
$\Delta^\bullet$ equals $|\Delta[\bullet]|$. For every $\ell\geq 0$
denote by $sd^\ell\Delta^\bullet$ the cosimplicial Nisnevich sheaf
$|sd^\ell\Delta[\bullet]|$. Under this notation we then have a
canonical isomorphism of motivic spaces
   $$Ex^\ell(C_*\cc Fr(\cc X))=\cc Fr(|sd^\ell\Delta[\bullet]|_+\wedge-,\cc X).$$
It follows that $Ex^\ell(C_*\cc Fr(\cc X))$ is a space with framed
correspondences as well as the space
   $$Ex^\infty(C_*\cc Fr(\cc X))=\colim_\ell Ex^\ell(C_*\cc Fr(\cc X)).$$
Then $Ex^\infty(C_*\cc Fr(\cc X))$ is a sectionwise fibrant pointed
simplicial set and the canonical map $C_*\cc Fr(\cc X)\to
Ex^\infty(C_*\cc Fr(\cc X))$ is a sectionwise weak equivalence.

For brevity we denote by $\mathbf G$ the sheafification of $\bb G$.
It equals the simplicial sheaf $(\bb G_m^{\wedge 1})_+\in
sShv_\bullet(Sm/k)$ (see Section~\ref{furtherappl}). Given a
$(S^1,\mathbf G)$-bispectrum $E$ of simplicial Nisnevich sheaves,
replace it by a stably cofibrant bispectrum $E^c$ in the stable
projective motivic model structure of bispectra associated with the
projective motivic structure on $sShv_\bullet(Sm/k)$ in the sense
of~\cite{Bl} (we shall with this model structure throughout the
section). Then each $(i,j)$-entry $E^c_{i,j}$ of $E^c$ belongs to
$\bl{\longrightarrow}{\Delta^{\op}Fr_0(k)}$. We set
   $$M_{fr}^{\bb G}(E)_{i,j}:=Ex^\infty(C_*Fr(E^c_{i,j})),\quad i,j\geq 0. $$
The structure maps in the $S^1$- and $\mathbf G$-direction
   $$Ex^\infty(C_*Fr(E^c_{i,j}))\to\underline{\Hom}(S^1,Ex^\infty(C_*Fr(E^c_{i+1,j}))),\
     Ex^\infty(C_*Fr(E^c_{i,j}))\to\underline{\Hom}(\mathbf G,Ex^\infty(C_*Fr(E^c_{i,j+1})))$$
are obviously induced by the structure maps $u_v,u_h$ of $E^c$
(see~\cite[p.~488]{Jar} for the relevant definitions on bispectra).
Precisely, they are compositions
   $$Ex^\infty(C_*Fr(E^c_{i,j}))\to\underline{\Hom}(S^1,Ex^\infty(C_*Fr(E^c_{i,j}\otimes S^1)))
     \xrightarrow{u_v}\underline{\Hom}(S^1,Ex^\infty(C_*Fr(E^c_{i+1,j})))$$
and
   $$Ex^\infty(C_*Fr(E^c_{i,j}))\to\underline{\Hom}(\mathbf G,Ex^\infty(C_*Fr(E^c_{i,j}\otimes\bb G_m^{\wedge 1})))
     \xrightarrow{u_h}\underline{\Hom}(\mathbf G,Ex^\infty(C_*Fr(E^c_{i,j+1})))$$
respectively.

For brevity we drop $Ex^\infty$ from notation and tacitly assume
below that all spaces like $C_*Fr(E^c_{i,j})$ are sectionwise
fibrant with framed correspondences. We then have a canonical
morphism of bispectra
   $$\zeta:E^c\to M_{fr}^{\bb G}(E).$$
Clearly, $\zeta$ is functorial in $E$.

Denote by $SH^{fr}(k)$ the full subcategory of $SH(k)$ consisting of
framed bispectra, i.e. those bispectra $\cc E$ such that each space
$\cc E_{i,j}$ is a space with framed correspondences and the
structure maps $\cc E_{i,j}\to\underline{\Hom}(S^1,\cc E_{i+1,j})$,
$\cc E_{i,j}\to\underline{\Hom}(\mathbf G,\cc E_{i,j+1})$ preserve
framed correspondences.

\begin{thm}\label{zeta}
For every $(S^1,\mathbf G)$-bispectrum of simplicial Nisnevich
sheaves $E$ the morphism of bispectra $\zeta:E^c\to M_{fr}^{\bb
G}(E)$ is a stable motivic equivalence. In particular, $E$ is
isomorphic to $M_{fr}^{\bb G}(E)$ in $SH(k)$ by the zigzag of
equivalences $E\leftarrow E^c\to M_{fr}^{\bb G}(E)$ and $M_{fr}^{\bb
G}$ induces an equivalence of categories $M_{fr}^{\bb
G}:SH(k)\xrightarrow{\simeq}SH^{fr}(k)$.
\end{thm}

\begin{proof}
Let $diag(E^c),diag(M_{fr}^{\bb G}(E))$ be the diagonal
$S^1\wedge\mathbf G$-spectra. They consist of motivic spaces
$(E^c_{0,0},E^c_{1,1},\ldots)$ and
$(C_*Fr(E^c_{0,0}),C_*Fr(E^c_{1,1}),\ldots)$ respectively. It
suffices to show that $diag(\zeta):diag(E^c)\to diag(M_{fr}^{\bb
G}(E))$ is a stable motivic equivalence of $S^1\wedge\mathbf
G$-spectra.

By~\cite[p.~496]{Jar} $diag(E^c)$ has a natural filtration
   $$diag(E^c)=\colim_n L_n (diag(E^c)),$$
where $L_n (diag(E^c))$ is the spectrum
   $$E_{0,0}^c,E_{1,1}^c,\ldots,E_{n,n}^c,E_{n,n}^c\wedge(S^1\wedge\mathbf G),
     E_{n,n}^c\wedge(S^1\wedge\mathbf G)^2,\ldots$$
In turn, $diag(M_{fr}^{\bb G}(E))$ has a natural filtration
   $$diag(M_{fr}^{\bb G}(E))=\colim_n M_{fr}^{S^1\wedge\bb G}(L_n (diag(E^c))),$$
where $M_{fr}^{S^1\wedge\bb G}(L_n (diag(E^c)))$ is the spectrum
   $$C_*Fr(E_{0,0}^c),C_*Fr(E_{1,1}^c),\ldots,C_*Fr(E_{n,n}^c),C_*Fr(E_{n,n}^c\otimes(S^1\otimes\bb G_m^{\wedge 1})),
     C_*Fr(E_{n,n}^c\otimes(S^1\otimes\bb G_m^{\wedge 1})^2),\ldots$$
It follows from Theorem~\ref{Motivic_Segal_Thm_SimplBi} that each
morphism $L_n (diag(E^c))\to M_{fr}^{S^1\wedge\bb G}(L_n
(diag(E^c)))$ is a stable motivic equivalence, and hence so is
$diag(\zeta)$.
\end{proof}

Observe that the composite functor
   $$H_{\bb A^1}(k)\xrightarrow{\Sigma^\infty_{\bb G}\Sigma^\infty_{S^1}} SH(k)
     \xrightarrow{M_{fr}^{\bb G}}SH(k)$$
is isomorphic to the functor $M_{fr}^{\bb G}$ of
Theorem~\ref{infloopspectra}(3). Thus the functor of
Theorem~\ref{zeta} extends that of Theorem~\ref{infloopspectra}(3).

Next, denote by $\cc M_{fr}^b$ the functor taking a bispectrum $E$
to $(\Theta^\infty_{\bb G}\circ\Theta^\infty_{S^1}\circ M_{fr}^{\bb
G})(E)$ and refer to it as the {\it big framed motive functor}.
Here $\Theta^\infty_{S^1}$ applies to each $S^1$-spectrum
of the bispectrum $M_{fr}^{\bb G}(E)$ similar to the
formula~\eqref{theta} followed by $\Theta^\infty_{\bb G}$ which
applies to the bispectrum $(\Theta^\infty_{S^1}\circ M_{fr}^{\bb
G})(E)$ in the $\mathbf G$-direction similar to~\eqref{theta}.

\begin{thm}\label{a1localspace}
The following statements are true for every bispectrum $E$:

\begin{enumerate}
\item the natural map $\mu:M_{fr}^{\bb G}(E)\to\cc M_{fr}^b(E)$ is a
stable motivic equivalence of bispectra;

\item for any $i,j\geq 0$ the space with framed correspondences $\cc M_{fr}^b(E)_{i,j}$
is $\bb A^1$-local as an ordinary motivic space;

\item the bispectrum $\cc M_{fr}^b(E)^f$, obtained from $\cc M_{fr}^b(E)$ by taking Nisnevich
local replacements $\cc M_{fr}^b(E)_{i,j}^f$ in all entries, is motivically fibrant.
\end{enumerate}
\end{thm}

\begin{proof}
(1). Since the projective motivic model structure on
$sShv_\bullet(Sm/k)$ is weakly finitely generated in the sense
of~\cite{DRO}, our assertion is proved similar to that
of~Lemma~\ref{spektrell}.

(2). Write $E^c=(E_{*,0}^c,E_{*,1}^c,\ldots)$ as a collection of
$S^1$-spectra. Then $M_{fr}^{\bb G}(E)$ is a collection of
$S^1$-spectra $(C_*Fr(E_{*,0}^c),C_*Fr(E_{*,1}^c),\ldots)$ and
   $$\Theta^\infty_{S^1}M_{fr}^{\bb G}(E)=(\Theta^\infty_{S^1}(C_*Fr(E_{*,0}^c)),\Theta^\infty_{S^1}(C_*Fr(E_{*,1}^c)),\ldots).$$
By construction, the $(i,j)$-entry equals
   $$\Theta^\infty_{S^1}M_{fr}^{\bb G}(E)_{i,j}:=\colim(C_*Fr(E_{i,j}^c)\to\underline{\Hom}(S^1,C_*Fr(E_{i+1,j}^c))
        \to\underline{\Hom}(S^2,C_*Fr(E_{i+2,j}^c))\to\cdots).$$
In each weight $j$, the $S^1$-spectrum $E_{*,j}^c$ has a natural
filtration $E^c_{*,j}=\colim_n L_n(E^c_{*,j})$, where
$L_n(E^c_{*,j})$ is the spectrum
   $$E_{0,j}^c,E_{1,j}^c,\ldots,E_{n,j}^c,E_{n,j}^c\wedge S^1,E_{n,j}^c\wedge S^2,\ldots$$
Then $C_*Fr(E^c_{*,j})=C_*Fr(\colim_n L_n(E^c_{*,j}))=\colim_nC_*Fr(L_n(E^c_{*,j}))$, where
$C_*Fr(L_n(E^c_{*,j})$ is the spectrum
   $$C_*Fr(E_{0,j}^c),C_*Fr(E_{1,j}^c),\ldots,C_*Fr(E_{n,j}^c),C_*Fr(E_{n,j}^c\otimes S^1),C_*Fr(E_{n,j}^c\otimes S^2),\ldots$$
Since each space $C_*Fr(E_{n,j}^c\otimes S^\ell)$, $\ell>0$, is $\bb
A^1$-local, then so is the space
$\underline{\Hom}(S^m,C_*Fr(E_{n,j}^c\otimes S^\ell))$, where $m\geq
0$. Therefore each space of
$\Theta^\infty_{S^1}(C_*Fr(E^c_{*,j}))=\Theta^\infty_{S^1}(\colim_nC_*Fr(L_n(E^c_{*,j})))
=\colim_n\Theta^\infty_{S^1}(C_*Fr(L_n(E^c_{*,j})))$ is $\bb
A^1$-local, because so is each
$\Theta^\infty_{S^1}(C_*Fr(L_n(E^c_{*,j})))$. We use here the fact
that a directed colimit of Nisnevich excisive spaces is Nisnevich
excisive to conclude that a directed colimit of $\bb A^1$-local
spaces is $\bb A^1$-local. We see that each
$\Theta^\infty_{S^1}M_{fr}^{\bb G}(E)_{i,j}$ is $\bb A^1$-local. If
we take a Nisnevich local fibrant resolution
$\Theta^\infty_{S^1}M_{fr}^{\bb G}(E)_{i,j}^f$ in each
$(i,j)$-entry, we obtain a bispectrum
$\Theta^\infty_{S^1}M_{fr}^{\bb G}(E)^f$. Then each $S^1$-spectrum
$\Theta^\infty_{S^1}M_{fr}^{\bb G}(E)_{*,j}^f$ of the bispectrum
$\Theta^\infty_{S^1}M_{fr}^{\bb G}(E)^f$ is motivically fibrant in
the local stable model structure of $S^1$-spectra. Indeed, this
follows from the fact that the structure maps of the $S^1$-spectrum
$\Theta^\infty_{S^1}(C_*Fr(E^c_{*,j}))$ are isomorphisms and that
the functor $\underline{\Hom}(S^1,-)$ preserves local equivalences.

Next, the $S^1$-spectrum in each weight $j$ of $\cc M_{fr}^b(E)$
equals by definition
   $$\colim(\Theta^\infty_{S^1}(C_*Fr(E^c_{*,j}))
     \to\underline{\Hom}(\mathbf G,\Theta^\infty_{S^1}(C_*Fr(E^c_{*,j+1})))
     \to\underline{\Hom}(\mathbf G^{\wedge 2},\Theta^\infty_{S^1}(C_*Fr(E^c_{*,j+2})))\to\cdots).$$
We have shown above that $(\Theta^\infty_{S^1}(C_*Fr(E^c_{*,j})))_f$
is a motivically fibrant $S^1$-spectrum. We claim that
   $$\underline{\Hom}(\mathbf G,\Theta^\infty_{S^1}(C_*Fr(E^c_{*,j+1})))\to
     \underline{\Hom}(\mathbf G,(\Theta^\infty_{S^1}(C_*Fr(E^c_{*,j+1})))_f)$$
is a levelwise local equivalence. In this case it will follow that
each space $\cc M_{fr}^b(E)_{i,j}$ is $\bb A^1$-local.

Since both spectra are sectionwise $\Omega$-spectra, it suffices to
prove that this arrow is a stable local equivalence. The presheaves
of stable homotopy groups of the left spectrum are $\bb
A^1$-invariant quasi-stable radditive with framed correspondences
(see~\cite[Introduction]{GP2}. Therefore our claim follows from the
following

\begin{sublem}
Let $\cc X$ be an $\bb A^1$-local motivic $S^1$-spectrum whose
presheaves of stable homotopy groups are homotopy invariant
quasi-stable radditive presheaves with framed correspondences
(see~\cite{GP4} for the definition of such presheaves). Suppose $\cc
X^f$ is a local stable fibrant replacement of $\cc X$. Then the map
of spectra $\underline{\Hom}(\mathbf G,\cc
X)\to\underline{\Hom}(\mathbf G,\cc X^f)$ is a local stable
equivalence.
\end{sublem}

\begin{proof}
First let us compute $\pi_*^{\nis}(\underline{\Hom}((\mathbb
G_m)_+,\cc X^f))$. We have $\cc X^f=\hocolim_{n\to-\infty}\cc
X^f_{\geq n}$, where $\cc X_{\geq n}$ is the naive $n$th truncation
of $\cc X$ in $Sp_{S^1}(sShv_\bullet(Sm/k))$. $\cc X$ has homotopy
invariant, quasi-stable radditive presheaves with framed
correspondences of stable homotopy groups $\pi_*(\cc X)$.
By~\cite[1.1]{GP4} the Nisnevich sheaves $\pi_*^{\nis}(\cc X_{\geq
n})$ are strictly homotopy invariant. If $\cc X^f_{\geq n}$ is a
stable local replacement of $\cc X_{\geq n}$, it follows from
Proposition~\ref{prop:A1_fibrant} that $\cc X^f_{\geq n}$ is
motivically fibrant, hence $\cc X_{\geq n}$ is $\bb A^1$-local in
$SH_{S^1}(k)$. Since
   $$\underline{\Hom}((\mathbb G_m)_+,\cc X^f)=
     \hocolim_{n\to-\infty}\underline{\Hom}((\mathbb G_m)_+,\cc X^f_{\geq n})$$
$\pi_*^{\nis}(\underline{\Hom}((\mathbb G_m)_+,\cc
X^f))=\colim_{n\to-\infty}\pi_*^{\nis}(\underline{\Hom}((\mathbb
G_m)_+,\cc X^f_{\geq n}))$.

Consider the Brown--Gersten convergent spectral sequence
   $$H^p_{\nis}(U\times\bb G_m,\pi_q^{\nis}(\cc X_{\geq n}))\Longrightarrow
     \pi_{q-p}(\cc X^f_{\geq n})(U\times\bb G_m),\quad U\in Sm/k.$$
It follows from~\cite[16.10-11]{GP4} that each presheaf $U\mapsto
H^p_{\nis}(U\times\bb G_m,\pi_q^{\nis}(\cc X_{\geq n}))$ is $\bb
A^1$-invariant quasi-stable radditive with framed correspondences.
If $U$ is a smooth local Henzelian scheme then
by~\cite[2.15(3')]{GP4} there is an embedding
   $$H^p_{\nis}(U\times\bb G_m,\pi_q^{\nis}(\cc X_{\geq n}))
     \hookrightarrow H^p_{\nis}(\bb G_{m,k(U)},\pi_q^{\nis}(\cc X_{\geq n})),$$
where $k(U)$ is the function field on $U$. By the Sublemma
in~\cite[Appendix~A]{GP2} we have that $H^p_{\nis}(\bb
G_{m,k(U)},\pi_q^{\nis}(\cc X_{\geq n}))=0$ for $p>0$, and hence
$H^p_{\nis}(U\times\bb G_{m},\pi_q^{\nis}(\cc X^f_{\geq n}))=0$ for
$p>0$. We can conclude that
   $$\pi_n^{\nis}(\underline{\Hom}((\mathbb G_m)_+,\cc X^f))
     =\pi_n^{\nis}(\cc X^f)(\mathbb G_m\times-).$$
It also follows that
   $$\pi_n^{\nis}(\underline{\Hom}(\mathbf G,\cc X^f))
     =(\pi_n^{\nis}(\cc X^f))_{-1}=(\pi_n^{\nis}(\cc X^f))_{-1}.$$

It remains to show that the morphism of $\bb A^1$-invariant
radditive quasi-stable framed sheaves
   $$(\pi_n(\underline{\Hom}(\mathbf G,\cc X)))^{\nis}=(\pi_n(\cc X))_{-1}^{\nis}\to(\pi_n^{\nis}(\cc X))_{-1}$$
is an isomorphism. Using~\cite[2.15(3')]{GP4} it suffices to check
that it is an isomorphism for every field extension $K/k$. The
homomorphism of Abelian groups
   $$(\pi_n(\cc X))_{-1}^{\nis}(K)=(\pi_n(\cc X))_{-1}(K)\to(\pi_n^{\nis}(\cc X))_{-1}(K)$$
is an isomorphism, because for any $\bb A^1$-invariant radditive
quasi-stable framed presheaf of Abe\-li\-an groups $\cc F$ and every
open $V\subset\bb A^1_K$, one has $\cc F(V)=\cc F^{\nis}(V)$ (see
the proof of~\cite[2.1]{GP4}).
\end{proof}

(3). By the previous assertion each space $\cc M_{fr}^b(E)_{i,j}$ is
$\bb A^1$-local, and hence $\cc M_{fr}^b(E)_{i,j}^f$ is a
motivically fibrant space. The proof of the assertion shows that
$\cc M_{fr}^b(E)_{i,j}^f$ can be computed as the $i$th space of the
the motivically fibrant $S^1$-spectrum
   $$\colim(\Theta^\infty_{S^1}(C_*Fr(E^c_{*,j}))^f
     \to\Omega_\mathbf G(\Theta^\infty_{S^1}(C_*Fr(E^c_{*,j+1}))^f)
     \to\Omega_{\mathbf G^{\wedge 2}}(\Theta^\infty_{S^1}(C_*Fr(E^c_{*,j+2}))^f)\to\cdots).$$
Moreover, this colimit is the $S^1$-spectrum in weight $j$ of the
bispectrum $\cc M_{fr}^b(E)^f$. By~\cite[4.12]{H} such a bispectrum
must be motivically fibrant. This completes the proof.
\end{proof}

\begin{defs}{\rm
(1) Given a $(S^1,\mathbf G)$-bispectrum $E$ and $i,j\geq 0$, denote
by $\underline{\pi}^{fr}_{i,j}(E)$ the graded Nisnevich sheaf
$\bigoplus_{n\geq 0}\pi_{n}^{\nis}(\cc M^b_{fr}(E)_{i,j})$ and refer
to such sheaves as {\it framed sheaves of $E$}.

(2) Denote by $SH^{fr}_{\nis}(k)$ the full subcategory of
$SH^{fr}(k)$ consisting of framed $(S^1,\mathbf G)$-bispectra whose
motivic spaces are $\bb A^1$-local as ordinary motivic spaces.

}\end{defs}

Theorem~\ref{a1localspace} implies the following

\begin{thm}\label{a1localspace1}
The following statements are true:

\begin{enumerate}
\item Framed sheaves detect stable weak equivalences of bispectra.
Namely, a map of bispectra $f:E\to E'$ is a stable motivic
equivalence if and only if
$\underline{\pi}^{fr}_{i,j}(E)\lra{f_*}\underline{\pi}^{fr}_{i,j}(E)$
is an isomorphism of sheaves for all $i,j\geq 0$.

\item For every bispectrum $E$ the bispectrum $\cc M^b_{fr}(E)$
belongs to $SH_{\nis}^{fr}(k)$. Moreover, the functor $\cc
M^b_{fr}:SH(k)\to SH_{\nis}^{fr}(k)$ is an equivalence of categories
and there is a natural isomorphism of endofunctors on $SH(k)$:
   $$\id\to\iota\circ\cc M^b_{fr},$$
where $\iota:SH_{\nis}^{fr}(k)\to SH(k)$ is the natural inclusion.
\end{enumerate}
\end{thm}

In other words, the preceding theorem says that the functor $\cc
M^b_{fr}$ converts classical Morel--Voevodsky's stable motivic
homotopy theory $SH(k)$ into an equivalent local homotopy theory of
$\bb A^1$-local framed bispectra from $SH_{\nis}^{fr}(k)$. The main
ingredient of this equivalent local homotopy theory is framed
motivic spaces of the form $C_*Fr(-,Y)$ with
$Y\in\Delta^{\op}Fr_0(k)$ a simplicial scheme as well as their
framed motives $M_{fr}(Y)$.

We document this as follows.

\begin{thm}\label{a1localspace2}
The Morel--Voevodsky stable motivic homotopy category $SH(k)$
can be defined as follows. Its objects are $(S^1,\mathbf G)$-bispectra and morphisms
between two bispectra $E,E'$ are given by the set $\pi_0(E^c,\cc M_{fr}^b(E')^f)$
of ordinary morphisms between bispectra $E^c$ and $\cc M_{fr}^b(E')^f$ modulo the naive homotopy.
In particular, $SH(k)(\Sigma^\infty_{S^1}\Sigma^\infty_{\bb G}X_+,E')=\pi_0(\cc M_{fr}^b(E')^f_{0,0}(X))$
for any $X\in Sm/k$.
\end{thm}

To conclude the section, the interested reader will easily construct
the ``framed functor with 1/2-coefficients" $\cc M_{fr}^b\{1/2\}$
and prove the same theorems from this section for $SH(k)[1/2]$ in
any characteristic of the base field $k$.

\section{Framed $\bb P^1$-spectra}\label{flamedspectra}

In Theorem~\ref{infloopspaces} we have shown that the
suspension spectrum $\Sigma^\infty_{\bb P^1}\cc X$ of a motivic
space $\cc X\in sShv_\bullet(Sm/k)$ is naturally equivalent to the spectrum
   $$M_{\bb P^{\wedge 1}}(\cc X^c)=(C_*\cc Fr(-,\cc X^c),C_*\cc Fr(-,\cc X^c\wedge T),
     C_*\cc Fr(-,\cc X^c\wedge T^2),\ldots).$$
It induces a functor of homotopy categories
   \begin{equation*}\label{mp1ha1}
    M_{\bb P^{\wedge 1}}:H_{\bb A^1}(k)\to SH(k),
   \end{equation*}
which is equivalent to $\Sigma^\infty_{\bb P^1}$ (see
Theorem~\ref{infloopspaces}). Observe that $M_{\bb P^{\wedge
1}}$ lands in the full subcategory of $\bb P^1$-spectra, denote it
by $SH^{fr}(k)$, whose motivic spaces are $\bb A^1$-invariant with
framed correspondences. This section will show the reader how to get
naturally framed $\bb P^1$-spectra. Another approach for bispectra
was illustrated in the previous section.

The purpose of this final section is to construct an equivalence of categories
   $$SH(k)\lra{\cong} SH^{fr}(k)$$
(see Theorem~\ref{shfrsp}). Moreover, we shall prove that every spectrum is isomorphic in
$SH(k)$ to a framed $\Omega$-spectrum from $SH^{fr}(k)$.

{\it Throughout the section $k$ is any field and by $\Spt^{\bb
P^1}(Sm/k)$ (respectively $\Spt^{T}(Sm/k)$) we mean the category of
$\bb P^1$-spectra (respectively $T$-spectra) associated with
simplicial Nisnevich sheaves}. We shall consider the level and
stable flasque model structures on $\Spt^{\bb P^1}(Sm/k)$ or
$\Spt^{T}(Sm/k)$ in the sense of~\cite{Is}. An advantage of the
model structures is that a filtered colimit of fibrant objects is
again fibrant~\cite[5.3, 6.7]{Is} (in fact, both model structures
are weakly finitely generated in the sense of~\cite{DRO}). A fibrant
$\bb P^1$-spectrum with respect to the stable motivic model
structure will also be referred to as an $\Omega$-spectrum.

The motivic equivalence $\sigma:\bb P^{\wedge 1}\to T$ induces an adjoint pair
   $$f:\Spt^{\bb P^1}(Sm/k)\rightleftarrows\Spt^{T}(Sm/k):g,$$
where $g$ is the forgetful functor. When proving
Theorem~\ref{Segal_Thm_II}(1) the reader may have observed that we
first replaced the suspension spectrum $\Sigma^\infty_{\bb P^1}\cc
X$ by the suspension $T$-spectrum $\Sigma^\infty_{T}\cc
X=f(\Sigma^\infty_{\bb P^1}\cc X)$ and then applied $\Theta^\infty$
to the $\bb P^1$-spectrum $\Sigma^\infty_{\bb P^1,T}\cc X:=
gf(\Sigma^\infty_{\bb P^1}\cc X)=g(\Sigma^\infty_{T}\cc X)$ in order
to get framed correspondences. We see that framed correspondences
are obtained from the $T$-spectrum $\Sigma^\infty_{T}\cc X$.

We want to extend this construction to spectra. Suppose $\cc E\in \Spt^{T}(Sm/k)$.
By~\cite[p.~496]{Jar} $\cc E=(\cc E_0,\cc E_1,\ldots)$ has a natural filtration
   $$\cc E=\colim_n L_n\cc E,$$
where $L_n\cc E$ is the spectrum
   $$\cc E_0,\cc E_1,\ldots,\cc E_n,\cc E_n\wedge T,\cc E_n\wedge T^2,\ldots$$
Denote by $L_n^{\bb P^1}\cc E$ the spectrum $g(L_n\cc E)$. By Lemma~\ref{spektrell} the natural map of spectra
   $$\eta_n:L_n^{\bb P^1}\cc E\to\Theta^\infty L_n^{\bb P^1}\cc E$$
is a stable equivalence. There is an isomorphism of spectra\footnotesize
   $$\iota_n:\Theta^\infty L_n^{\bb P^1}\cc E\lra{\cong}\Theta^\infty L_{n,fr}^{\bb P^1}\cc E:=(\underline{\Hom}(\bb P^{\wedge n},\cc Fr(\cc E_n)),\ldots,
       \underline{\Hom}(\bb P^{\wedge 1},\cc Fr(\cc E_n)),\cc Fr(\cc E_n),\cc Fr(\cc E_n\wedge T),\ldots).$$
\normalsize If we apply the Suslin complex functor $C_*$ levelwise, we get a spectrum\footnotesize
   $$C_*\Theta^\infty L_{n,fr}^{\bb P^1}\cc E:=(\underline{\Hom}(\bb P^{\wedge n},C_*\cc Fr(\cc E_n)),\ldots,
       \underline{\Hom}(\bb P^{\wedge 1},C_*\cc Fr(\cc E_n)),C_*\cc Fr(\cc E_n),C_*\cc Fr(\cc E_n\wedge T),\ldots).$$
\normalsize By~\cite[2.3.8]{MV} the natural map of spectra
   $$\delta_n:\Theta^\infty L_{n,fr}^{\bb P^1}\cc E\to C_*\Theta^\infty L_{n,fr}^{\bb P^1}\cc E$$
is a level motivic weak equivalence.

Passing to the colimit over $n$ we get that the composite map of
spectra
   $$\colim_n(\delta_n\iota_n\eta_n):g(\cc E)=\colim_n L_n^{\bb P^1}\cc E\to\colim_n C_*\Theta^\infty L_{n,fr}^{\bb P^1}\cc E$$
is a stable motivic weak equivalence. We use here the fact that a
sequential colimit of stable motivic weak equivalences is a stable
motivic weak equivalence by~\cite[3.12]{Jar}.

Denote by $\Theta^\infty C_*\cc Fr(\cc E)$ the spectrum\footnotesize
   $$\colim_n\underline{\Hom}(\bb P^{\wedge n},C_*\cc Fr(\cc E_n)),\colim_n\underline{\Hom}(\bb P^{\wedge n},C_*\cc Fr(\cc E_{1+n})),
        \colim_n\underline{\Hom}(\bb P^{\wedge n},C_*\cc Fr(\cc E_{2+n})),\ldots$$
\normalsize Also, denote by $C_*\cc Fr(\cc E)$ the spectrum $(C_*\cc
Fr(\cc E_0),C_*\cc Fr(\cc E_1),C_*\cc Fr(\cc E_2),\ldots)$. Each
structure map $C_*\cc Fr(\cc E_k)\to\underline{\Hom}(\bb P^{\wedge
1},C_*\cc Fr(\cc E_{k+1}))$, $k\geq 0$, is given by the natural
composition\footnotesize
   $$C_*\cc Fr(\cc E_k)\to\underline{\Hom}(\bb P^{\wedge 1},C_*\cc Fr(\cc E_{k}\wedge\bb P^{\wedge 1}))
      \xrightarrow{\sigma_*}\underline{\Hom}(\bb P^{\wedge 1},C_*\cc Fr(\cc E_{k}\wedge T))\xrightarrow{(\epsilon_k)_*}
       \underline{\Hom}(\bb P^{\wedge 1},C_*\cc Fr(\cc E_{k+1})),$$
\normalsize where $\epsilon_k:\cc E_{k}\wedge T\to\cc E_{k+1}$ is the structure map of $\cc E$.
Observe that for any $\cc X\in\Delta^{\op}\bl{\to}{Fr_0}(k)$ the spectrum
$M_{\bb P^{\wedge 1}}(\cc X)$ is isomorphic to $C_*\cc Fr(\Sigma^\infty_{\bb P^1,T}(\cc X))$ in $SH(k)$.

The spectrum $\colim_n C_*\Theta^\infty L_{n,fr}^{\bb P^1}\cc E$ is
naturally isomorphic to the spectrum $\Theta^\infty C_*\cc Fr(\cc
E)$ and the stable motivic equivalence
   $$\alpha:g(\cc E)\to\Theta^\infty C_*\cc Fr(\cc E)$$
factors as $g(\cc E)\lra{\beta}C_*\cc Fr(\cc
E)\lra{\gamma}\Theta^\infty C_*\cc Fr(\cc E)$. Here $\gamma$ equals
the stable motivic equivalence of Lemma~\ref{spektrell} and each
$\beta_n:\cc E_n\to C_*\cc Fr(\cc E_n)$ is the obvious map of
motivic spaces. The two-out-of-three property implies $\beta$ is a
stable motivic equivalence. It is plainly functorial in $\cc E\in
\Spt^{T}(Sm/k)$.

\begin{thm}\label{cfre}
For every $T$-spectrum $\cc E$ there is a natural stable motivic
equivalence of spectra $\beta:g(\cc E)\to C_*\cc Fr(\cc E)$,
functorial in $\cc E$. Moreover, a morphism of spectra $u:\cc
E\to\cc E'$ is a stable motivic equivalence if and only if so is
$C_*\cc Fr(u):C_*\cc Fr(\cc E)\to C_*\cc Fr(\cc E')$. In particular,
we have a functor
   $$C_*\cc Fr:\Ho(\Spt^{T}(Sm/k))\to\Ho(\Spt^{\bb P^1}(Sm/k)),$$
which is an equivalence of categories.
\end{thm}

\begin{proof}
The first statement has already been verified above. The second
statement follows from the first statement and the fact that $u:\cc
E\to\cc E'$ is a stable motivic equivalence if and only if so is
$g(u):g(\cc E)\to g(\cc E')$ (see~\cite[p.~477]{Jar}). The functor
   $$C_*\cc Fr:\Ho(\Spt^{T}(Sm/k))\to\Ho(\Spt^{\bb P^1}(Sm/k)),$$
is an equivalence of categories, because so is $g$ and $\beta:g(\cc E)\to C_*\cc Fr(\cc E)$
is a stable motivic equivalence, functorial in $\cc E$.
\end{proof}

\begin{lem}\label{cfretheta}
For every $T$-spectrum $\cc E$, each space of the $\bb P^1$-spectrum
$\Theta^\infty\cc Fr(\cc E)$ (respectively $\Theta^\infty C_*\cc
Fr(\cc E)$) is a motivic space with framed correspondences.
\end{lem}

\begin{proof}
$\Theta^\infty \cc Fr(\cc E)$ is the spectrum
   $$\colim_n\underline{\Hom}(\bb P^{\wedge n},\cc Fr(\cc E_n)),\colim_n\underline{\Hom}(\bb P^{\wedge n},\cc Fr(\cc E_{1+n})),
        \colim_n\underline{\Hom}(\bb P^{\wedge n},\cc Fr(\cc E_{2+n})),\ldots$$
Given a sheaf $F$ and $s>0$, we claim that $\underline{\Hom}(\bb
P^{\wedge s},\cc Fr(F))$ is a sheaf with framed correspondences (the
internal Hom is taken in the category of pointed Nisnevich sheaves).
To see this, define for any $U,X\in Sm/k$ and any $m,n$ a map of
pointed sets
$$\text{Hom}(U_+\wedge \bb P^{\wedge m},X_+\wedge T^m) \wedge\text{Hom}(X_+\wedge \bb P^{\wedge s+n},F\wedge T^n) \to
\text{Hom}(U_+\wedge \bb P^{\wedge (s+m+n)},F\wedge T^{m+n}).$$
If $\alpha: U_+\wedge \bb P^{\wedge m} \to X_+\wedge T^m$ and
$v: X_+\wedge \bb P^{\wedge s+n}\to F\wedge T^n$
are morphisms of pointed Nisnevich sheaves, then we define $\alpha^*(v)$ as the composite morphism
\begin{gather*}
U_+\wedge\bb P^{\wedge s}\wedge\bb P^{\wedge m}\wedge \bb P^{\wedge n}\cong
U_+\wedge\bb P^{\wedge m}\wedge \bb P^{\wedge s+n} \xrightarrow{\alpha\wedge\id} X_+\wedge T^m\wedge \bb P^{\wedge s+n}\cong\\
\cong T^m\wedge X_+\wedge \bb P^{\wedge s+n} \xrightarrow{\id\wedge v} T^m\wedge F\wedge T^n\cong F\wedge T^m\wedge T^n.
\end{gather*}
Passing to the colimit over $n$, we get that $\underline{\Hom}(\bb
P^{\wedge s},\cc Fr(F))$ is a sheaf with framed correspondences as
claimed.

Next, the composite morphism
   $$\underline{\Hom}(\bb P^{\wedge s},\cc Fr(\cc E_k))\to\underline{\Hom}(\bb P^{\wedge s+ 1},\cc Fr(\cc E_{k}\wedge T)) \xrightarrow{(\epsilon_k)_*}\underline{\Hom}(\bb P^{\wedge s+1},\cc Fr(\cc E_{k+1})),\quad s,k\geq 0,$$
is plainly a morphism of framed Nisnevich sheaves. Recall that the
left arrow is induced by the map taking $(v: X_+\wedge \bb P^{\wedge
s+n} \to\cc E_k\wedge T^n)\in\underline{\Hom}(\bb P^{\wedge s},\cc
Fr_n(\cc E_k))$ to the composition\footnotesize
   $$(X_+\wedge \bb P^{\wedge s+1+n}\cong X_+\wedge \bb P^{\wedge s+n}\wedge \bb P^{\wedge 1}\xrightarrow{v\wedge\sigma}
       \cc E_k\wedge T^n\wedge T\cong\cc E_k\wedge T\wedge T^n)\in
       \underline{\Hom}(\bb P^{\wedge s+1},\cc Fr_n(\cc E_{k}\wedge T)).$$
\normalsize It follows that each motivic space $(\Theta^\infty \cc
Fr(\cc E))_k=\colim_n\underline{\Hom}(\bb P^{\wedge n},\cc Fr(\cc
E_{k+n}))$ of the spectrum $\Theta^\infty \cc Fr(\cc E)$ has framed
correspondences. Obviously, the same is true for the spectrum
$\Theta^\infty C_*\cc Fr(\cc E)$.
\end{proof}

Suppose $\cc E\in \Spt^{T}(Sm/k)$. Then there are isomorphisms of
$\bb P^1$-spectra
   $$\Theta^\infty(g(\cc E))=\Theta^\infty(\colim_n L_n^{\bb P^1}\cc E)\cong\colim_n\Theta^\infty(L_n^{\bb P^1}\cc E))\cong\Theta^\infty\cc Fr(\cc E).$$
If $\cc E$ is levelwise motivically fibrant it follows that
$\Theta^\infty\cc Fr(\cc E)$ is an $\Omega$-spectrum, because so is
$\Theta^\infty(g(\cc E))$ by~\cite[4.6]{H}. Using
Lemma~\ref{cfretheta}, we have shown therefore the following

\begin{lem}\label{omegafr}
For every levelwise fibrant $T$-spectrum $\cc E$, the spectrum
$\Theta^\infty\cc Fr(\cc E)\in \Spt^{\bb P^1}(Sm/k)$ is an
$\Omega$-spectrum with framed correspondences.
\end{lem}

We are now in a position to prove that $SH(k)$ is canonically
equivalent to its full subcategory $SH^{fr}(k)$ of framed spectra,
i.e. those spectra whose motivic spaces are $\bb A^1$-invariant with
framed correspondences.

\begin{thm}\label{shfrsp}
The functor $C_*\cc Fr:\Spt^{T}(Sm/k)\to\Spt^{\bb P^1}(Sm/k)$ takes
a spectrum $\cc E$ to a framed spectrum $C_*\cc Fr(\cc E)$. The
composite functor
   $$SH(k)=\Ho(\Spt^{\bb P^1}(Sm/k))\xrightarrow{f}\Ho(\Spt^{T}(Sm/k))\xrightarrow{C_*\cc Fr}\Ho(\Spt^{\bb P^1}(Sm/k)),$$
induces an equivalence of categories $SH(k)\lra{\sim} SH^{fr}(k)$.
Moreover, every ${\bb P^1}$-spectrum is isomorphic in $SH(k)$ to a framed $\Omega$-spectrum.
\end{thm}

\begin{proof}
By construction, $C_*\cc Fr(\cc E)$ is an $\bb A^1$-invariant with
framed correspondences. The next statement follows from
Theorem~\ref{cfre} and the fact that $f$ is an equivalence of
categories (see~\cite[2.13]{Jar}). To show that every ${\bb
P^1}$-spectrum $E$ is isomorphic in $SH(k)$ to a framed
$\Omega$-spectrum, let $E^c$ be a cofibrant replacement of $E$ and
$f(E^c)_{lf}\in\Spt^T(Sm/k)$ be a level fibrant replacement of the
$T$-spectrum $f(E^c)$. Let $\Theta^\infty_T$ denote the
stabilization functor in the category of $T$-spectra. Since the
flasque motivic model structure on spaces is weakly finitely
generated (this follows from~\cite[3.10]{Is}), the $T$-spectrum
$\Theta^\infty_T(f(E^c)_{lf})$ is fibrant by~\cite[4.6]{H} and the
morphism $f(E^c)_{lf}\to\Theta^\infty_T(f(E^c)_{lf})$ is a stable
equivalence by~\cite[4.11]{H}. The canonical arrow
$g(\Theta^\infty_T(f(E^c)_{lf}))\to\Theta^\infty(gf(E^c)_{lf})$ is a
level weak equivalence by~\cite[p.~477]{Jar}, and hence the
composite morphism
   $$E^c\to gf(E^c)\to g(\Theta^\infty_T(f(E^c)_{lf}))\to\Theta^\infty(gf(E^c)_{lf})$$
is a stable weak equivalence, where the left arrow is the adjunction
unit morphism. The $\bb P^1$-spectrum $\Theta^\infty(gf(E^c)_{lf})$
is a framed $\Omega$-spectrum by Lemma~\ref{omegafr}, so the zigzag
   $$E\leftarrow E^c\to \Theta^\infty(gf(E^c)_{lf})$$
gives an isomorphism in $SH(k)$.
\end{proof}

The preceding theorem says that we can define the stable motivic
homotopy theory as framed $\bb P^1$-spectra or even framed
$\Omega$-spectra. However, such framed $\Omega$-spectra are hardly
amenable for explicit calculations in general, because they require
levelwise motivically fibrant replacements by construction. Instead,
the main point of this paper is to show that whenever the base field
$k$ is infinite perfect with $\charr k\ne2$, one can nevertheless
construct {\it exlicit\/ } fibrant $\Omega$-spectra using framed
correspondences of Voevodsky and the machinery of framed motives
introduced and developed in this paper.

Namely, suppose $\cc E=(\cc E_0,\cc E_1,\ldots)\in\Spt^{T}(Sm/k)$ is such that its framed spectrum $C_*\cc Fr(\cc E)$
satisfies the following conditions:

\begin{enumerate}
\item each space $C_*\cc Fr(\cc E_n)$, $n\geq 0$, is locally connected;
\item each space $C_*\cc Fr(\cc E_n)$, $n\geq 0$, is $\sigma$-invariant (i.e. it
takes the framed correspondence $\sigma_X=(X\times 0,X\times\bb A^1,pr_{\bb A^1},pr_X)$ of level one
to a weak equivalence of simplicial sets for every $X\in Sm/k$);
\item each structure morphism induces a motivic equivalence
   $$C_*\cc Fr(\cc E_n)_f\to\underline{\Hom}(\bb P^{\wedge 1},C_*\cc Fr(\cc E_{n+1})_f),\quad n\geq 0,$$
where the subscript ``$f$" refers to a local fibrant replacement.
\end{enumerate}
Then the spectrum $C_*\cc Fr(\cc E)_f:=(C_*\cc Fr(\cc E_0)_f,C_*\cc Fr(\cc E_1)_f,\ldots)$
is an $\Omega$-spectrum stably equivalent to $\cc E$. In particular, if a motivic space $\cc X$ is such that
its suspension $T$-spectrum satisfies $(1)-(3)$, then $C_*\cc Fr(\cc X)$ is locally
equivalent to the space $\Omega^\infty_{\bb P^1}\Sigma^\infty_{\bb P^1}\cc X$ (see p.~\pageref{ominfty}
for the definition of the latter space).
The hardest condition in practice is condition (3), where the machinery of framed motives
works in its full capacity.

\section{Concluding remarks}


To finish the paper, we should mention that Voevodsky~\cite{Voe2}
also defines another type of framed correspondences
$Fr^{rat}_*(U,X)$, in which regular functions on etale neighborhoods
of supports are replaced by {\it rational\/} functions. Additivity
Theorem~\ref{additivity} is also true for the motivic space
   $$U\in Sm/k\longmapsto C_*Fr^{rat}(U,X)\in\bb S_\bullet\quad X\in Sm/k.$$
It follows from~\cite{Voe2} that $C_*Fr^{rat}(U,X)$ is a group-like
space, and hence the $S^1$-spectrum
   $$M_{fr}^{rat}(X)=(C_*Fr^{rat}(-,X),C_*Fr^{rat}(-,X\otimes S^1),C_*Fr^{rat}(-,X\otimes S^2),\ldots)$$
is an $\Omega$-spectrum by the Segal machine~\cite{S}. We call it
the {\it rational framed motive of $X\in Sm/k$}.

\begin{conj}\label{trudno}
The natural map of $S^1$-spectra $M_{fr}(X)\to M_{fr}^{rat}(X)$ is a
local stable weak equivalence for all $X\in Sm/k$.
\end{conj}

Basing on the technique developed in the previous sections, it is
quite standard to develop homotopy theory of rational framed
motives. If the conjecture is true, one can then construct explicit
motivic fibrant replacements for spectra/bispectra of the form
$\Sigma_{\bb P^{1}}^\infty X_+$ and $\Sigma_{S^1}^\infty\Sigma_{\bb
G}^\infty X_+$ respectively in terms of spaces $C_*Fr^{rat}(-,X)$.

\end{document}